\documentclass[reqno,11pt]{amsart}
\usepackage{latexsym,amssymb,amsfonts,amsmath,amsthm}
\usepackage[usenames]{xcolor}
\usepackage{eucal}
\usepackage{mathtools}
\usepackage[colorlinks,citecolor=blue,urlcolor=blue]{hyperref}

\usepackage{tikz}
\usetikzlibrary{decorations.markings}
\usepackage{pgfplots}

\newtheorem{thm}{Theorem}[section] 

\newtheorem{lem}[thm]{Lemma}

\theoremstyle{definition}

\newtheorem{rem}[thm]{Remark}
\numberwithin{equation}{section}
\newcommand{\Real}{\mathbb R}
\newcommand{\eps}{\varepsilon}

\newcommand{\F}{\mathcal{F}}

\newcommand{\one}[1]{\mathbf{1}_{\{#1\}}}

\renewcommand{\P}{\mathbb{P}}

\newcommand{\E}{\mathbb{E}}

\DeclareMathOperator*{\sign}{sign}

\renewcommand{\Re}{\mathrm{Re}}
\renewcommand{\Im}{\mathrm{Im}}

\def\Xint#1{\mathchoice
   {\XXint\displaystyle\textstyle{#1}}%
   {\XXint\textstyle\scriptstyle{#1}}%
   {\XXint\scriptstyle\scriptscriptstyle{#1}}%
   {\XXint\scriptscriptstyle\scriptscriptstyle{#1}}%
   \!\int}
\def\XXint#1#2#3{{\setbox0=\hbox{$#1{#2#3}{\int}$}
     \vcenter{\hbox{$#2#3$}}\kern-.5\wd0}}

\def\dashint{\Xint-}

\makeatletter
\newcommand{\doublewidetilde}[1]{{%
  \mathpalette\double@widetilde{#1}%
}}
\newcommand{\double@widetilde}[2]{%
  \sbox\z@{$\m@th#1\widetilde{#2}$}%
  \ht\z@=.9\ht\z@
  \widetilde{\box\z@}%
}
\makeatother

\begin{document}

\title[Linear filtering with fractional noises]
{Linear filtering with fractional noises:  large time and small noise asymptotics}

\author{D. Afterman}%
\address{Department of Statistics,
The Hebrew University,
Mount Scopus, Jerusalem 91905,
Israel}
\email{Danielle.Afterman@mail.huji.ac.il }

\author{P. Chigansky}%
\address{Department of Statistics,
The Hebrew University,
Mount Scopus, Jerusalem 91905,
Israel}
\email{pchiga@mscc.huji.ac.il}

\author{M. Kleptsyna}%
\address{Laboratoire de Statistique et Processus,
Universite du Maine,
France}
\email{marina.kleptsyna@univ-lemans.fr}

\author{D. Marushkevych}%
\address{Laboratoire de Statistique et Processus,
Universite du Maine,
France}
\email{dmytro.marushkevych.etu@univ-lemans.fr}

\thanks{P. Chigansky is supported by ISF  1383/18
grant}



\date{\today}%

\begin{abstract}
The classical state-space approach to optimal estimation of stochastic processes is efficient when 
the driving noises are generated by martingales. In particular, the weight function of the 
optimal linear filter, which solves a complicated operator equation in general, simplifies to  
the Riccati ordinary differential equation in the martingale case. This reduction lies in the 
foundations of the Kalman-Bucy approach to linear optimal filtering. In this paper we consider a 
basic Kalman-Bucy model with noises, generated by independent fractional Brownian motions, and 
develop a new method of asymptotic analysis of the integro-differential filtering equation arising 
in this case. We establish existence of the steady-state error limit and find its asymptotic scaling 
in the high signal-to-noise regime. Closed form expressions are derived in a number of important cases.
 
\end{abstract}

\maketitle 


\section{Introduction}
\subsection{The Kalman-Bucy problem}
In its most basic form, the Kalman-Bucy filtering problem \cite{KB61} is concerned with estimation of the 
state process, generated by the linear stochastic equation 
\begin{equation}\label{Xeq}
X_t = \beta \int_0^t X_s ds + W_t,   
\end{equation}
given a trajectory of the observation process 
\begin{equation}\label{Yeq}
Y_t = \mu \int_0^t X_s ds + \sqrt{\eps} V_t.
\end{equation}
Here $\beta$ and $\mu\ne 0$ are fixed real constants, $\eps>0$ is the observation noise intensity parameter, 
and $W=(W_t; t\in \Real_+)$ and $V=(V_t; t\in \Real_+)$ are independent Brownian motions.

The filtering problem consists of finding the optimal estimator $\widehat X_t = \E(X_t|\F^Y_t)$, 
whose mean squared error
$
P_t = \E (X_t-\widehat X_t)^2
$ 
is minimal among all functionals, measurable with respect to $\F^Y_t=\sigma\big\{Y_s, s\le t\big\}$. 
For the linear Gaussian model \eqref{Xeq}-\eqref{Yeq}, this problem has a famously elegant solution, discovered in \cite{KB61}.
The filtering estimator in this case can be generated by the stochastic differential equation 
\[
d\widehat X_t = \beta \widehat X_t dt + \frac{\mu P_t}{\eps} (dY_t - \mu \widehat X_t dt), 
\]
and the corresponding minimal error $P_t$ solves the Riccati o.d.e.
\begin{equation}\label{Riccati}
\dot P_t = 2\beta P_t + 1 -  (\mu/\sqrt\eps)^2 P_t^2,  
\end{equation}
subject to zero initial conditions. Elementary analysis of \eqref{Riccati} shows that the 
filtering error converges to the steady-state limit 
\begin{equation}\label{PT}
\lim_{T\to\infty} P_T\Big(\beta, \frac \mu {\sqrt{\eps}}\Big) = \frac{\beta + \sqrt{\beta^2+\mu^2/\eps}}{\mu^2/\eps} 
\end{equation}
and reveals its scaling with respect to the noise intensity 
\begin{equation}\label{PTeps}
P_T\Big(\beta, \frac \mu {\sqrt{\eps}}\Big) = \frac{\sqrt{\eps}}{\mu}\big(1+o(1)\big),\quad \text{as\ } \eps  \to 0, \quad \forall T>0.
\end{equation} 
These limiting quantities are of considerable interest, as they exhibit the fundamental accuracy limitations
in the problem.

\subsection{Fractional noises}
A natural generalization of the system \eqref{Xeq}-\eqref{Yeq} is obtained by replacing $W$ and $V$ 
with independent {\em fractional} Brownian motions (fBm) with the Hurst exponents $H_1$ and $H_2$, 
respectively. Recall that the fBm is a centred Gaussian process with covariance function  
\begin{equation}\label{KV}
K(s,t)  = \frac 1 2\Big(s^{2H} + t^{2H} - |s-t|^{2H}\Big), \quad s,t\in \Real_+,
\end{equation}
where $H\in (0,1)$ is its Hurst exponent.

This process coincides with the standard Brownian motion for $H=\frac 1 2$, but otherwise exhibits a rich diversity of properties, 
which makes it an interesting mathematical object and an important tool in modelling, \cite{PT01}. 
In particular, it is neither a semi-martingale nor a Markov process. For $H>\frac 1 2$ increments of the fBm are positively 
correlated and have long range dependence 
\[
\sum_{n=1}^\infty \E V_1 (V_{n+1}-V_n)=\infty.
\]
This property makes it useful in design and analysis of engineering systems, \cite{BP88}.

Consider the integro-differential equation
\begin{equation}\label{maineq}
\frac{\partial}{\partial s} \int_0^T   g_T(r)\frac{\partial}{\partial r} K_V(r,s)dr
\\
+ \frac{\mu^2}{\eps} \int_0^T K_X(r,s) g_T(r)dr = \frac{\mu}{\sqrt \eps} K_X(s,T),
\end{equation} 
where $K_V(s,t)$ and $K_X(s,t)$ are the covariance functions of the fBm $V$ and the state process $X$.
If this equation has a sufficiently regular solution $g_T(\cdot)$ (see Appendix \ref{sec:A} for some details) 
then standard calculations, see e.g. \cite[Lemma 10.2]{LS1}, 
show the optimal estimator is given by the stochastic integral 
$$
\widehat X_T = \frac 1 {\sqrt{\eps}}\int_0^T g_T(s)dY_s,
$$
and the filtering error is determined by the solution to \eqref{maineq} through  
the formula 
\begin{equation}\label{PTfun}
P_T \Big(\beta, \frac{\mu}{\sqrt \eps}\Big)= \frac {\sqrt \eps} \mu \left(
\frac{\partial}{\partial s} \int_0^T   g_T(r)\frac{\partial}{\partial r} K_V(r,s)dr
\right)_{\displaystyle\big|s:=T}.
\end{equation}

In the standard Kalman-Bucy case these equations simplify in two ways. First, when the observation noise is white, i.e.,   
$H_2=\frac 1 2$, the integro-differential terms in \eqref{maineq} and \eqref{PTfun} reduce  merely to $g_T(s)$, and 
\eqref{maineq} takes the form of an integral equation. 
Further, when the state noise is white as well, $H_1=\frac 1 2$, the state process is Markov, and its covariance 
kernel $K_X(s,t)$ factorizes into a product of exponentials. This allows to express solutions to \eqref{maineq} in terms of 
the Riccati o.d.e. \eqref{Riccati}.

In the more general, fractional setting under consideration, equation \eqref{maineq} retains its 
integro-differential form, and questions of solvability and asymptotic behaviour of its solution  
remained till now mainly open.
\medskip

\begin{enumerate}
\addtolength{\itemsep}{0.4\baselineskip}

\item How does the filtering error scale with observation noise intensity as $\eps\to 0$? 

\item Does it converge to a limit as $T\to\infty$?

\item How are these two asymptotics related? 

\item Do the limits admit of reasonably explicit expressions? 

\end{enumerate}

\medskip
\noindent  
Answering such questions requires an entirely different approach, which is the main focus of the paper.

\section{The main results}
Let $X$ and $Y$ be the processes generated by equations \eqref{Xeq} and \eqref{Yeq}, driven by independent fBm's
$W$ and $V$ with the Hurst parameters $H_1, H_2\in (0,1)$, respectively. Define 
$
P_T\big(\beta, \frac \mu {\sqrt \eps}\big) := \E \big(X_T-\E (X_T|\F^Y_T)\big)^2.
$
To avoid trivialities $\mu\ne 0$ is assumed throughout, but the values of all other parameters can be arbitrary.  

\subsection{General asymptotics}

Our principal result is the following theorem. 

\begin{thm}\label{thm:main}\
The large time limit exists
\begin{equation}\label{large-time}
P_\infty\Big(\beta, \frac \mu {\sqrt \eps}\Big) = \lim_{T\to\infty} P_T\Big(\beta, \frac \mu {\sqrt \eps}\Big),
\end{equation}
and, for any $T>0$, the filtering error satisfies the scaling property
\begin{equation}\label{small-noise}
\lim_{\eps\to 0}\eps^{-\nu} P_T\Big(\beta, \frac \mu {\sqrt \eps}\Big) = P_\infty\big(0, \mu\big) \quad \text{with\ } 
 \nu = \frac{H_1}{1+H_1-H_2}.
\end{equation}
\end{thm}

\begin{rem}
\

\medskip
\noindent
a) As in the standard Kalman-Bucy problem, the first order term of the small noise asymptotics \eqref{small-noise} 
does not depend on the interval length $T$ or the drift of the state process $\beta$, cf. \eqref{PTeps}. 
This is not entirely intuitive in the fractional case, since the memory of the optimal filter in the non-Markov setup, 
and the more so for processes with long range dependence, does not have to be a priori negligible  as $\eps\to 0$. 

\medskip
\noindent
b) The rate $\nu$ in \eqref{small-noise} coincides with the optimal minimax rate in the nonparametric problem of estimating a 
deterministic function observed in fractional type noise, \cite{W96}. 
This agrees with the smoothness of the fBm paths, which are Holder continuous with an exponent arbitrarily close to $H$.   
In particular, the estimators suggested in \cite{W96} should be rate optimal for the filtering problem 
under consideration, however, with a suboptimal constant. The dependence of $\nu$ on $H_1$ and $H_2$ agrees with the intuition, 
that the filtering accuracy should improve with path regularity of the processes which generate the noises.
\end{rem}

\subsection{Special cases}

In principle, the limit in \eqref{large-time} is derived in the proof as an explicit but a rather cumbersome expression. 
It can be significantly simplified in a number of  meaningful cases, as detailed in the 
theorems below. The key ingredient of the emerging  formulas is the complex valued {\em structural function}
\begin{equation}\label{Lambda}
\Lambda(z; H_1,H_2) = (z^2-\beta^2) \kappa(H_2) \left(\frac z i\right)^{1-2H_2} -\frac{\mu^2}{\eps} \kappa(H_1) \left(\frac z i\right)^{1-2H_1},
\end{equation} 
where $z$ takes values in the upper half of the complex plane and  
\begin{equation}\label{kappa}
\kappa(H)=\Gamma(2H+1)\sin (\pi H).
\end{equation}
Its domain is extended to the lower half-plane through conjugation  
\[
\Lambda(z; H_1,H_2)=\overline{\Lambda(\overline{z}; H_1,H_2)}.
\]

The structure of the filtering problem turns out to be largely determined by the configuration of zeros
of this function. A simple calculation shows that 
$\Lambda(z; H_1,H_2)$ has the unique complex zero $z_0$ in the first quadrant when $H_1>H_2$. As $H_1$ approaches $H_2$ this zero moves 
towards positive real semiaxis, and, at $H_1=H_2$, degenerates to the purely real value 
\[
t_0 = \sqrt{\beta^2 + \mu^2/\eps}.
\]
When $H_1< H_2$ it has no zeros at all.

\subsubsection{State/observation noises of the same type}

The following result details the limiting behaviour of the filtering error, when the state and observation noises have the same Hurst exponent.

\begin{thm}\label{thm:1} Let $H_1=H_2 =: H\in (0,1)$, then 
\begin{equation}\label{KlLBfla}
P_\infty\Big(\beta, \frac \mu {\sqrt \eps}\Big)  = \frac 1 2 \Gamma(2H+1)  t_0^{-2H} \Big(1+\sin (\pi H)\frac{t_0+\beta}{t_0-\beta}\Big),
\end{equation}
and, consequently\footnote{here and below,  $g(\eps)\asymp h(\eps)$ stands for $\lim_{\eps\to 0}g(\eps)/h(\eps)=1$ },  
\begin{equation}\label{KlLBfla_small_noise}
 P_T\Big(\beta, \frac \mu {\sqrt \eps}\Big) \asymp \frac 1  2    \Gamma(2H+1) \big( 1+ \sin (\pi H) \big) (\eps/\mu^2)^{  H }, \quad \text{as}\ \eps\to 0.
\end{equation}
\end{thm}
 
\begin{rem}

Formula \eqref{KlLBfla} was previously derived in \cite{KlLB02} for $H\in (\frac 1 2,1)$, using a completely different method, 
based on the innovation representation of the fBm from \cite{NVV99}. 
This approach does not easily extend to the complementary case $H\in (0,\frac 1 2)$, unlike the method suggested in this paper.

\end{rem}

\subsubsection{Fractional state/white observation noise}

To formulate further results, define the limit 
\[
\Lambda^+(t;H_1,H_2) = \lim_{\Im(z)>0, z\to t}\Lambda(z;H_1,H_2), \quad t\in \Real_+, 
\]
which coincides with the expression in \eqref{Lambda} after replacing $z$ with $t\in \Real_+$. 
Let $\theta(t;H_1,H_2)$ be the argument of $\Lambda^+(t;H_1,H_2)$, chosen so that it varies continuously 
with $t\in \Real_+$ and  $\lim_{t\to\infty}\theta(t; H_1,H_2)\in [-\pi, \pi]$. 
This choice defines $\theta(t; H_1,H_2)$ in the unique way, and it is a completely explicit function.

\medskip

The following theorem details the precise error asymptotics in the filtering problem with fractional state process and white noise
observations.

\begin{thm}\label{thm:2}
Let $H:=H_1\in (0,1)\setminus \{\frac 1 2\}$ and $H_2=\frac 1 2$, then
\begin{equation}\label{Pinfty1}
P_\infty\Big(\beta, \frac {\mu}{\sqrt \eps}\Big) = \frac \eps {\mu^2} \left(\frac 1 \pi \int_0^\infty  \theta(t; H, \tfrac 1 2)  dt +\beta  
+
\begin{rcases}   
\begin{dcases}
2\Re (z_0) & \text{if \ } H>\tfrac 1 2 \\
0 & \text{if \ }H<\tfrac 1 2 
\end{dcases}
\end{rcases}
\right),
\end{equation} 
where $z_0$ is the unique zero of $\Lambda(z; H, \frac 1 2)$ in the first quadrant.
Consequently, 
\begin{equation}\label{Peps0}
P_T\Big(\beta, \frac {\mu}{\sqrt \eps}\Big) \asymp  \frac{\kappa(H)^{\frac 1{2H+1}}}{\sin \frac \pi {2H+1}} (\eps/\mu^2)^{\frac{2H}{2H+1}}, \quad \text as \ \eps\to 0.
\end{equation} 
\end{thm}

\begin{rem}
\

\medskip
\noindent 
a) In the stable case with $\beta<0$, the following alternative expression for the filtering error can be obtained, using the 
spectral theory of stationary processes,
\begin{equation}\label{spectr}
P_\infty\Big(\beta, \frac {\mu}{\sqrt \eps}\Big) = \frac {\eps} {\mu^2} \frac 1 {2\pi} \int_{-\infty}^\infty \log \bigg(1+ \frac{\mu^2}\eps \kappa(H) \frac{|\omega|^{1-2H}}{\beta^2+\omega^2}\bigg) d\omega. 
\end{equation}
The spectral approach is not applicable in the non-stationary case $\beta \ge 0$ and, in fact, 
this formula can be seen to coincide with \eqref{Pinfty1} only for $\beta<0$, but not otherwise. 

\medskip
\noindent
b) The expression in \eqref{Pinfty1} has the right and the left limits at $H=\frac 1 2$, which coincide with the 
classic formula \eqref{PT}. 
While the root of $\Lambda(z; H_1,H_2)$ and the integral in \eqref{Pinfty1} do not seem to admit any closed
form formulae, both are not hard to compute numerically for any concrete values of the parameters. 

\medskip
\noindent
c) Formula \eqref{Peps0} can also be obtained using asymptotic approximation  of the eigenvalues and eigenfunctions 
of the covariance operator of the fractional Ornstein-Uhlenbeck (fOU) process, \cite{ChKM-AiT}. This approximation however 
is not uniform with respect to $T$, and therefore the large time limiting error \eqref{Pinfty1} cannot be derived using 
the same method.  

\end{rem}

\subsubsection{White state/fractional observation noise}

To formulate the results in the complementary case of white state  and fractional observation noises  define 
\begin{equation}\label{Xzdef}
X(z) =  (-z)^{ 3/2- H  }\exp \left(\frac 1 \pi \int_0^\infty \frac{\theta\big(t;\frac 1 2, H \big)}{t-z}dt\right), \quad z\in \mathbb{C}\setminus \Real_+,
\end{equation}
with $H\in (0,1)$.
This function is holomorphic on the cut plane with a jump discontinuity across the positive real semiaxis $\Real_+$. 
In the course of the proof the limits $X^+(|\beta|)$ and $X^-(|\beta|)$ are shown to coincide, and their 
common value will be denoted by $X(|\beta|)$.

\begin{thm}\label{thm:3}
Let $H_1=\frac 1 2$ and $H:=H_2 \in (0,1)\setminus \{\frac 1 2\}$. Then 
\begin{equation}\label{showme_T}
P_\infty \Big(\beta, \frac{\mu}{\sqrt\eps }\Big) =  \frac 1 {2\beta} \left(  \frac {X(-\beta)}{X(\beta)}
\begin{rcases}
\begin{dcases}
\Big|\frac{z_0+\beta}{z_0-\beta}\Big|^2  & \text{if\ } H<\tfrac 1 2\\
1 & \text{if\ }  H>\tfrac 1 2
\end{dcases} 
\end{rcases}
-1\right),
\end{equation} 
where $z_0$ is the zero of $\Lambda(z; \frac 1 2, H)$ in the first quadrant. 
Consequently, 
\begin{equation}\label{showme_eps}
P_\infty \Big(\beta, \frac{\mu}{\sqrt\eps }\Big)\asymp  \frac{\kappa(H)^{\frac 1 {3-2H}}}{\sin \frac{\pi}{3-2H}}\big(\eps/\mu^2\big)^{\frac 1 {3-2H}}, \quad \eps\to 0.
\end{equation} 
\end{thm}

\begin{rem}
Numerical evaluation of $X(z)$ at $z:=\beta>0$ involves computation of the Cauchy principal value of the integral in \eqref{Xzdef}.
The following identity, proved in Lemmas \ref{lem:Xz} and \ref{lem:XzXz} below,
\[
X(\beta)X(-\beta) = \frac 1 {\kappa(H)}\frac {\mu^2}{\eps} \left(\frac 1{|\beta^2-z_0^2|^2}\right)^{\one{H<\frac 1 2}},
\]
can be more convenient for this purpose.
\end{rem}

\section{Related literature}

\subsection{Integro-differential equations}
To the best of our knowledge integro-differential equations such as \eqref{maineq} do not have a general theory. As 
mentioned above, for the white observation noise $H_2=\frac 1 2$, problem \eqref{maineq} 
reduces to integral equation of the second kind. Such equations have been studied since the pioneering works 
of Fredholm, and their unique solvability in various spaces is very well understood. Nevertheless, even in 
this relatively standard setting, quantifying dependence of the solutions on parameters, such as $T$ and $\eps$ 
in our context, can be a highly nontrivial matter.  
Essentially, the only case in which a complete theory is available is that of Kalman-Bucy, 
when reduction of \eqref{maineq} to the Riccati o.d.e. is possible. This reduction has far reaching implications, way beyond 
the scalar problem considered in this paper. It leads to a complete characterisation of the limit behaviour of the 
optimal error in terms of such notions as controllability and observability (see, e.g., \cite{KS72}).

\subsection{Stationary problem}

The stationary version of the filtering problem for \eqref{maineq} on the semi-infinite time horizon $[0,\infty)$ 
can be solved within the framework of the Kolmogorov-Wiener spectral theory, \cite{Roz77}. 
In some cases it yields closed form formulas for the steady-state  error in the form of
integrals over spectral densities such as \eqref{spectr}.
However this approach is strictly limited to the stable state equation \eqref{Xeq} with $\beta<0$, 
even in the standard Kalman-Bucy problem. In fact, overcoming this difficulty was the main impetus behind 
the state-space approach pioneered in \cite{KB61}.

\subsection{Nonlinear filtering}
Optimal error analysis in the more general, nonlinear filtering problem attracted much attention 
in the more recent past. Questions of existence and uniqueness of the large time limit of the filtering error
was addressed first in \cite{Ku71} and continued to generate much research over the years;  surveys of different approaches can be found in \cite{A11}, \cite{Bu11}, \cite{ChLvH11},  \cite{KV11}, \cite{S11}.   
The steady state error is never explicit beyond the linear problem, and consequently various techniques of computing its 
lower bounds have been suggested, see  \cite{Z11}. Exact small noise error asymptotics was derived for a number of 
models, including diffusions  \cite{P86}, \cite{ZD88}, \cite{PZ05} and finite state chains \cite{KhZ96}, \cite{SBB00}. 
Let us stress, however, that all these results are concerned exclusively with the Markov case and therefore do not apply to the 
filtering models with fractional noises.

\subsection{Filtering with fractional noises}
Filtering in systems driven by the fractional Brownian motion have been addressed by many authors, including 
\cite{LB98}, \cite{CD99}, \cite{KKA98b}, \cite{KLR00b}, \cite{XX05}, \cite{Dun11}, both in linear and nonlinear settings. 
However, most of the literature is concerned with derivation of the filtering equation, 
rather than evaluation of the optimal error, which remained mainly elusive so far. 

\subsection{Contribution of this paper}
The contribution of this paper is twofold. From the perspective of stochastic filtering theory, 
it suggests a method of asymptotic error analysis in a non-Markov system with fractional type driving noises. 
Existence of the steady-state error beyond the Markov setting remained largely unexplored, and this is 
probably one of the first systematic takes on the subject. 
Besides a qualitative asymptotic picture as in Theorem \ref{thm:main}, our method yields closed form expressions 
for the filtering error limits in several cases of interest, such as  Theorems \ref{thm:1}, \ref{thm:2} 
and \ref{thm:3}.

Another contribution is on the more technical side, and it consists of constructing a solution to equation
\eqref{maineq}, which was previously known to be solvable in some special cases, see e.g. \cite{CCK}.
Our approach is inspired by a technique, introduced in the mathematical physics literature, \cite{Ukai}, \cite{Gibbs69}, 
and its recent applications to fractional stochastic processes, \cite{ChK}, \cite{ChKM20}, 
\cite{N19}. Previously it was used in eigenproblems, i.e., homogeneous integral equations of the second kind,
with the objective of approximating the sequence of its solutions. 
The {\em non-homogeneous} problem under consideration in this paper requires a complete revision of this technique 
at least from two standpoints. First, equation \eqref{maineq} has the unique solution, and this time the goal 
is its asymptotic analysis with respect to parameters. Moreover, the main object of interest is not the solution itself,
but its particular functional \eqref{PTfun}.

\section{Preliminaries}

\subsection{Notations, conventions and tools}
The proof uses some basic tools from complex analysis. Unless otherwise stated, the standard range 
$z\in (-\pi,\pi]$ will be used for  principal branches of the common multivalued functions.
We will frequently encounter functions, which are holomorphic on the cut planes $\mathbb{C}\setminus \Real$ or 
$\mathbb{C}\setminus \Real_+$, with a finite jump discontinuity across the cut.  For such a {\em sectionally holomorphic} 
function $\Psi(z)$,  the limits across the real line will be denoted by 
\[
\begin{aligned}
&
\Psi^+ (t) := \lim_{ \Im(z)>0, z\to t} \Psi(z),
\\
&
\Psi^- (t) := \lim_{\Im(z)<0, z\to t} \Psi(z),
\end{aligned}\qquad t\in \Real. 
\]

Often we will need to solve the Hilbert problem of finding a function $\Psi(z)$, 
which is sectionally holomorphic on $\mathbb{C}\setminus \Real_+$ and satisfies the boundary condition 
\[
\Psi^+(t) -\Psi^-(t) = \phi(t), \quad t\in \Real_+,
\]
for a given function $\phi(\cdot)$. When $\phi(\cdot)$ is H\"older on $\Real_+\cup \{\infty\}$, 
by the Sokhotski-Plemelj theorem, the unique solution  to this problem has the form 
\[
\Psi(z) = \frac 1 {2\pi i} \int_0^\infty \frac{\phi(t)}{t-z}d t + P(z),
\] 
where $P(z)$ is a polynomial of a finite degree, whose growth  as $z\to\infty$ matches that of $\Psi(z)$. 
A comprehensive account of such boundary value problems can be found in, e.g.,  monograph \cite{Gahov}.
  
When dependence on parameters is important, they will be added to the notations: 
for example, $g(x)$, $g_T(x)$ or $g(x; \eps,T)$ will denote the same function, depending on the context. 
It will also be convenient to use $\mu_\eps := \mu/\sqrt{\eps}$ and reparameterize the problem by $\alpha_1 := 2-2H_1$ and 
$\alpha_2 := 2-2H_2$, which take values in $(0,2)$.
Finally, we will write  $r(u)\asymp q(u)$ when $r(u) = q(u) (1+o(1))$ for both $u:=T\to\infty$  and  $u:=\eps\to 0$. 

\subsection{Proof preview}\label{sec:preview} 

In essence, the proof amounts to constructing the solution to \eqref{maineq} in a form, more amenable to asymptotic analysis. 
This is done by exploiting the structure of the Laplace transform of its solution, 
\begin{equation}\label{gTz}
\widehat g (z) = \int_0^T e^{-zx} g (x)dx, \quad z\in \mathbb{C},
\end{equation}
revealed by the representation formula, which is derived in Lemma \ref{lem:4.1} below,
\begin{equation}\label{gTz_expr}
\widehat g(z)  = -\frac 1 {\Lambda(z)}
\bigg(
 (z+\beta) \Big(\Phi_0(z)+e^{-zT} \Phi_1(-z)\Big) +  \mu_\eps^2  N_{\alpha_1}(z) 
\Big(\psi(0)+\frac {1}{\mu_\eps}e^{-zT}\Big)
\bigg).
\end{equation}
This expression involves the following elements.

\medskip 

\begin{enumerate}
\addtolength{\itemsep}{0.7\baselineskip}
\renewcommand{\theenumi}{\roman{enumi}}

\item  The complex function 
\begin{equation}\label{Nalpha}
N_\alpha(z) =  \kappa_\alpha \begin{cases}
\big(z/i\big)^{\alpha-1}, & \Im\{z\}>0,\\
\big(-z/i\big)^{\alpha-1}, & \Im\{z\} <0,
\end{cases}
\end{equation}
where, cf. \eqref{kappa},
$$
\kappa_\alpha = \kappa\big(1-\tfrac\alpha 2\big) = \frac{(1-\alpha)(1-\alpha/2)}{\Gamma(\alpha)}\frac \pi {\cos \frac \alpha 2 \pi}>0, \quad \alpha \in (0,2)\setminus \{1\}.
$$
This function is sectionally holomorphic on $\mathbb{C}\setminus \Real$, and its limits across the real
line satisfy the obvious symmetries 
\begin{equation}\label{Nsym}
N_\alpha^+(t) = N_\alpha^-(-t) \quad \text{and}\quad   N_\alpha^+(t)  = \overline{N_\alpha^-(t)}.
\end{equation}

\item\label{ii} 
The {\em structural } function of the problem, cf. \eqref{Lambda},
\begin{equation}\label{Lambdaz}
\Lambda(z) = (z^2-\beta^2) N_{\alpha_2}(z) - \mu_\eps^2 N_{\alpha_1}(z),
\end{equation}
which inherits the discontinuity of  $N_{\alpha_j}(z)$'s along the real line and is holomorphic elsewhere.
It does not vanish on the cut plane when $\alpha_1>\alpha_2$ and has four simple complex zeros, placed symmetrically 
in each quadrant, when $\alpha_1<\alpha_2$.  In the case $\alpha_1=\alpha_2$, it has two purely real zeros. 
Configuration of zeros has a determining effect on the solution.

\item\label{iv}  
Functions $\Phi_0(z)$ and $\Phi_1(z)$ are sectionally holomorphic on $\mathbb{C}\setminus \Real_+$. 
They are defined explicitly as certain functionals of $g(\cdot)$, involving the Cauchy integrals, but their particular form
is inessential, except for the growth  \eqref{apriori_est} as $z\to 0$ and $z\to\infty$.

\item  The quantity $\psi(0)$, also determined by $g(\cdot)$, is 
constant with respect to $z$, see \eqref{psidef} below. 

\end{enumerate}

\medskip

The crucial feature of representation \eqref{gTz_expr} is its singularities.
Since the integration in \eqref{gTz} is carried out over a finite interval, the Laplace transform $\widehat g(z)$ 
is an entire function, and thus all singularities in \eqref{gTz_expr} must be removable.
This includes  discontinuity across the real line, whose removal yields equations \eqref{eq:6.11} below. These 
equations bind together the limits $\Phi_0^\pm(t)$ and $\Phi_1^\pm (t)$ at all $t\in \Real_+$ and
can be viewed as boundary conditions on $\Real_+$ for the functions $\Phi_0(z)$ and $\Phi_1(z)$, 
which are holomorphic elsewhere.

Finding all such functions satisfying the particular growth estimates, mentioned in \eqref{iv}, is known as the Hilbert boundary 
value problem. In our case, all its solutions can be expressed in terms of auxiliary integral equations  of the general form 
\begin{equation}\label{pjt}
p(t) =  \big(A_{\eps, T} \,  p\big)(t) + f(t), \quad t\in \Real_+,
\end{equation}
where $A_{\eps,T}$ is an integral operator with an explicit kernel, and $f(\cdot)$ is either a specific function  
or a finite degree polynomial.

The functions $\Phi_0(z)$ and $\Phi_1(z)$ can be expressed in terms  of  solutions to these equations and several 
unknown constants. The number of these constants is determined by the configuration of zeros of $\Lambda(z)$ as 
mentioned in \eqref{ii}. Substitution of the expressions for $\Phi_0(z)$ and $\Phi_1(z)$ into \eqref{gTz_expr} 
yields an expression for the Laplace transform $\widehat g(z)$, which is therefore determined by solutions to \eqref{pjt} and 
the constants. The solution to \eqref{maineq} can then be found by inverting the Laplace transform. 
Consequently, the filtering error, determined by the functional \eqref{PTfun}, can also be expressed in terms of 
solutions to equations \eqref{pjt}. 

For example,  when $\alpha_1>\alpha_2$, there are no zeros, and as it turns out, the only unknown constant
in this case is $\psi(0)$ from \eqref{gTz_expr}. It can be found using the a priori condition 
\begin{equation}\label{cond0}
\left(
\frac{\partial}{\partial s} \int_0^T   g(r)\frac{\partial}{\partial r} K_V(r,s)dr
\right)_{\displaystyle\big|s:=0}=0,
\end{equation}
implied by \eqref{maineq} as  $K_X(0,t)=0$ for all $t\in [0,T]$.  
When $\alpha_1\le \alpha_2$, the  function $\Lambda(z)$ has several zeros, which appear in \eqref{gTz_expr} as simple poles. 
Removing these poles leads to a system of linear algebraic equations, which along with \eqref{cond0},  determine all the unknown 
coefficients.

At the first glance, so constructed representation does not appear any simpler than the original problem itself, 
since equations \eqref{pjt} cannot be solved explicitly. 
Remarkably though, a significant simplification is possible 
due to the properties of operator $A_{\eps,T}$, which force the first term in the right hand side  of \eqref{pjt} to 
vanish asymptotically as either $T\to \infty$ or $\eps\to 0$. Consequently, otherwise non-explicit function $p(t)$ 
can be approximated asymptotically by the forcing function $f(t)$. This is where the assertions of Theorem 
\ref{thm:main} come from. In fact, the limit $P_\infty\big(\beta, \frac\mu{\sqrt\eps}\big)$ can be found in a closed, though 
rather complicated form.

Further simplifications are possible in the special cases, when either of the functions 
$N_{\alpha_1}(z)$ or $N_{\alpha_2}(z)$ in formula \eqref{gTz_expr} degenerate  to 1, as in 
Theorem \ref{thm:2} and Theorem \ref{thm:3}, or when they remain non-degenerate, but coincide as 
in Theorem \ref{thm:1}. The ultimate expressions in all three cases are obtained by somewhat different calculations, 
which are detailed in Sections \ref{sec:case1}-\ref{sec:case3}.

\section{Proof of Theorem \ref{thm:main}}
In the notations introduced above, the covariance function of the fBm in \eqref{Yeq} has the form, cf. \eqref{KV}, 
\begin{equation}\label{Kst}
K_V(s,t)  = \frac 1 2\Big(s^{2-\alpha_2} + t^{2-\alpha_2} - |s-t|^{2-\alpha_2}\Big),
\end{equation}
with $\alpha_2\in (0,2)$. 
The covariance function of the fractional Ornstein-Uhlenbeck state process $X$, generated by \eqref{Xeq}, is 
\begin{equation}\label{KX}
K_X(s,t) = \int_0^s e^{\beta(s-u)} \frac{\partial}{\partial u}\int_0^t e^{\beta(t-v)} \frac \partial {\partial v}K_W(u,v) dvdu,
\end{equation}
where $K_W(s,t)$ is the  kernel in \eqref{Kst} with $\alpha_2$ replaced by $\alpha_1$.

\subsection{The Laplace transform} 

The following lemma details the structure of the Laplace transform of the solution to the main filtering equation 
and its relation to the filtering error.

\begin{lem}\label{lem:4.1}
Let $g(\cdot)$ solve equation \eqref{maineq} with  $K_V(s,t)$ and $K_X(s,t)$ as above.

\medskip 
\noindent 
1. 
The Laplace transform $\widehat g(z)$, defined in \eqref{gTz},  satisfies representation \eqref{gTz_expr} where 
\begin{equation}\label{psidef}
\psi(r) =  e^{-\beta r} \int_r^T e^{\beta \tau}g(\tau)d\tau  -\frac 1{\mu_\eps} e^{\beta(T-r)},  
\end{equation}
and the functions $\Phi_0(z)$ and $\Phi_1(z)$ are sectionally holomorphic on $\mathbb{C}\setminus\Real_+$,
satisfying 
\begin{equation}\label{apriori_est}
|\Phi_1(z)|\vee |\Phi_0(z)|  =   
\begin{cases}
O\big(z^{(\alpha_1\wedge \alpha_2-1)\wedge 0}\big), & z\to 0, \\
O\big(z^{(\alpha_2-1)\vee  0}\big), & z\to\infty.
\end{cases}
\end{equation}

\medskip 
\noindent 
2.  The following condition holds 
\begin{equation}\label{cond1}
\lim_{\Re(z)\to\infty} z \bigg(
N_{\alpha_2}(z) \widehat g(z)
- \frac 1 {2\pi i} \int_0^\infty\frac 1{t-z} \big(N_{\alpha_2}^+(t)-N_{\alpha_2}^-(t)\big)
\widehat g(t)dt
\bigg)=0,
\end{equation}
and the filtering error \eqref{PTfun} is given by the limit
\begin{multline}\label{PTfla}
P_T =   \frac 1 {\mu_\eps} \lim_{\Re(z)\to\infty} z \bigg(N_{\alpha_2}(-z) e^{-zT} \widehat g(-z)  \\
  -\frac 1 {2\pi i}  \int_0^\infty \frac 1 {t-z} \big(N_{\alpha_2}^+(t)   - N_{\alpha_2}^-(t) \big) e^{-tT}\widehat g(-t) dt\bigg).
\end{multline}
\end{lem}

\medskip

The proof of this lemma uses the operator       
\begin{equation}\label{vfa}
v_{f, \alpha} (s) := \frac {\partial}{\partial s} \int_0^T (1-\tfrac {\alpha} 2)|s-r|^{1-\alpha} \sign(s-r)f(r)dr,
\end{equation} 
which acts on sufficiently regular functions $f$.
Since for $\alpha \in (0,2)$
\[
|x-y|^{1-\alpha} \sign(x-y)= \frac 1 {\Gamma(\alpha)} \int_0^\infty t^{\alpha-1} (x-y)e^{-t|x-y|}dt,   
\]
we can  write 
\begin{equation}\label{predst}
v_{f, \alpha} (x)  = \frac 1 {c_\alpha}\frac {d}{dx} \int_0^\infty t^{\alpha-1} u_{f} (x,t) dt,
\end{equation}
where $c_\alpha:= \frac {\Gamma(\alpha)} {1-\tfrac {\alpha} 2} $ and 
\[
u_{f}(x,t) := \int_0^T  (x-y)e^{-t|x-y|} f(y)dy.
\]
In addition, let us define  another auxiliary function
\[
w_f (x,t) := \int_0^T  e^{-t|x-y|} f(y)dy.
\]

\begin{lem} 
The Laplace transform of \eqref{vfa} satisfies 
\begin{equation}\label{eq:6.3}
  \widehat v_{f, \alpha} (z)  =  N_\alpha(z) \widehat f(z)+ e^{-zT} \Psi_{f,1}(-z) +  \Psi_{f,0}(z)
\end{equation}
where  $N_\alpha(z)$ is defined in \eqref{Nalpha} and 
\begin{equation}\label{Psidefn}
\begin{aligned}
\Psi_{f,1}(z):= &\phantom{+\ } \frac 1 {c_\alpha} \int_0^\infty    \frac{t^\alpha  }{t-z}  u_{f} (T,t)  dt + \frac 1 {c_\alpha} z\int_0^\infty   \frac{t^{\alpha-1}}{(t-z)^2}w_f(T,t)dt,\\
\Psi_{f,0}(z):= & -\frac 1 {c_\alpha}\int_0^\infty  \frac {t^\alpha }{t-z}  u_f(0,t) dt + \frac 1 {c_\alpha} z \int_0^\infty  \frac {t^{\alpha-1} } {(t-z)^2} w_f(0,t) dt.
\end{aligned}
\end{equation}

\end{lem}

\begin{proof}
Differentiating $u_{f}(x,t)$ once with respect to $x$ gives
\begin{align*}
u'_{f}(x,t) & =  
\frac {d}{dx} \left(
\int_0^x  (x-y)e^{-t(x-y)} f(y)dy
- 
\int_x^T  (y-x)e^{-t(y-x)} f(y)dy
\right) =\\
&
w_f(x,t)
- t\int_0^x    (x-y)e^{-t(x-y)} f(y)dy  
- t \int_x^T   (y-x)e^{-t(y-x)} f(y)dy,
\end{align*}
and, consequently,
\begin{equation}\label{ubnd}
\begin{aligned}
u'_{f}(0,t) =\, & w_f(0,t) + t u_f(0,t), \\
u'_{f}(T,t) =\, & w_f(T,t) - t u_f(T,t).
\end{aligned}
\end{equation}
Similarly, 
\begin{align*}
w'_f (x,t)  = & \frac d{dx} \int_0^x  e^{-t(x-y)} f(y)dy+\frac d{dx}\int_x^T  e^{-t(y-x)} f(y)dy =\\
&
 - t\int_0^x    e^{-t(x-y)} f(y)dy
+ t\int_x^T  e^{-t(y-x)} f(y)dy
\end{align*}
and 
\begin{equation}\label{wbnd}
\begin{aligned}
w'_f (0,t)  =  & \phantom{+\ } t\, w_f (0,t), \\
w'_f (T,t)  = & - t  w_f (T,t).
\end{aligned}
\end{equation}

Taking further derivative gives the system of equations 
\begin{equation}\label{diffeq}
\begin{aligned}
u''_f(x,t) =\,  
&
t^2  u_f(x,t )  + 2w'_f(x,t),   
\\
w''_f (x,t)  =\, &
 t^2  w_f (x,t) - 2t     f(x). 
\end{aligned}
\end{equation}
Applying the Laplace transform to the first equation we obtain 
\[
\widehat u''_f(z,t) =
2\widehat w'_f(z,t)    + t^2  \widehat u_f(z,t). 
\]
or, equivalently, 
\begin{align*}
&
e^{-zT}   u_f'(T,t)-u_f'(0,t) + 
e^{-zT}  z u_f(T,t)-zu_f(0,t) + z^2\widehat u_f(z,t)
=\\
&
2 e^{-zT}w_f(T,t)-2w_f(0,t) + 2z \widehat w_f(z,t)   + t^2  \widehat u_f(z,t ).
\end{align*}
Collecting the terms and using the boundary conditions \eqref{ubnd}, this can be written as
\begin{align*}
&
e^{-zT}  \Big(     (z-t)   u_f(T,t) -w_f(T,t) \Big)    - (z+t) u_f(0,t) 
  +  w_f(0,t) 
\\
&
 - 2z \widehat w_f(z,t)  + (z^2  - t^2)  \widehat u_f(z,t ) =0,
\end{align*}
Due to the usual relation between Laplace transforms of a function and its derivatives, and again, in view of the boundary conditions 
\eqref{ubnd}, we can further write 
\begin{equation}\label{ufzt}
\begin{aligned}
\widehat u_f(z,t )= & \,   \frac{2z}{z^2-t^2} \widehat w_f(z,t) 
-e^{-zT}  \Big(  \frac{  u_f(T,t)}{z+t} 
 -  \frac{w_f(T,t)}{z^2-t^2}\Big) \\
 &
+   \frac {u_f(0,t)}{z-t}  - \frac {w_f(0,t) } {z^2-t^2}.
\end{aligned}
\end{equation}
A similar calculation shows that the second equation in \eqref{diffeq} along with the corresponding boundary conditions \eqref{wbnd} 
yields 
$$
  e^{-zT}   (z-t)  w_f (T,t)     -(z+t)\, w_f (0,t)  + (z^2-t^2)\widehat w_f(z,t)   
 + 2t  \widehat f(z)    =0
$$
and 
\[
  \widehat w_f(z,t) = -e^{-zT}    \frac{  w_f (T,t)}{z+t}     +\frac {w_f (0,t)}{z-t}     -\frac{ 2t}{ z^2-t^2}  \widehat f(z).
\]
Combining this with \eqref{ufzt}, we obtain
\[
\widehat u_f(z,t ) =   - \frac{4 z t}{(z^2-t^2)^2}   \widehat f(z)
+   \frac {u_f(0,t)}{z-t} + \frac {w_f(0,t) } {(z -t)^2} 
-e^{-zT}  \left( \frac{  u_f(T,t)}{z+t}  +   \frac{w_f(T,t)}{(z +t)^2} \right).
\]
By definition \eqref{predst}, 
\[
\begin{aligned}
c_\alpha  \widehat v_{f, \alpha} (z)  =\, & 
e^{-zT} \int_0^\infty t^{\alpha-1}   u_{f} (T,t) dt 
- \int_0^\infty t^{\alpha-1} u_{f} (0,t) dt \\
&
+ z
\int_0^\infty t^{\alpha-1} \widehat u_{f} (z,t) dt.
\end{aligned}
\]
Substituting the expression for $\widehat u_f(z,t )$, we arrive at \eqref{eq:6.3}
with 
\[
N_\alpha(z)  = -\frac 1 {c_\alpha} 4 z^2 \int_0^\infty   \frac{ t^\alpha }{(t^2-z^2)^2}    dt.
\]
The simpler expression \eqref{Nalpha} is derived by the standard contour integration.
\end{proof}

\medskip
We are now in position to proceed with the proof of Lemma \ref{lem:4.1}.
\medskip

\begin{proof}[Proof of Lemma \ref{lem:4.1}]
\

\medskip
\noindent 
1. Observe that $K_X(s,t)$ in \eqref{KX} is differentiable in $s\in (0,T)$ and 
\begin{align*}
\frac {\partial}{\partial s} K_X(s,t) = & \, \beta K_X(s,t) + 
\frac{\partial}{\partial s}\int_0^t  e^{\beta(t-v)} \frac {\partial }{\partial v} K_W(s,v)dv.
\end{align*}
Hence taking derivative of \eqref{maineq} we get 
\begin{align*}
&
\frac {\partial^2} {\partial s^2}  \int_0^T g(r) \frac {\partial }{\partial r} K_V(r,s)dr + 
\beta \left(\mu_\eps^2 \int_0^T  K_X(s,r)  g(r)dr-\mu_\eps    K_X(s,T) \right)
+ \\
&
\mu_\eps^2 \frac{\partial}{\partial s}\int_0^T  g(r)
\int_0^r e^{\beta(r-v)} \frac{\partial}{\partial v}K_W(s,v)  dv
 dr
 =  
\mu_\eps \frac{\partial}{\partial s}\int_0^T e^{\beta(T-v)} \frac{\partial}{\partial v}K_W(s,v) dv.
\end{align*}
In view of \eqref{maineq},  the expression in brackets here can be replaced with
\[
\mu_\eps^2 \int_0^T  K_X(s,r)  g(r)dr-\mu_\eps    K_X(s,T)= -\frac{\partial}{\partial s} \int_0^T   g(r)\frac{\partial}{\partial r} K_V(r,s)dr,
\]
and the last term in the left had side with
\[
\int_0^T  \int_0^r  g(r) e^{\beta(r-v)} \frac{\partial}{\partial v}K_W(s,v)  dvdr = 
\int_0^T \frac{\partial}{\partial r}K_W(s,r) e^{-\beta r}\int_r^T e^{\beta u}g(u)du dr,
\]
after integration by parts. Plugging these expressions, we arrive at 
\begin{align*}
\frac{\partial^2}{\partial s^2} \int_0^T g(r) \frac {\partial}{\partial r}K_V(s,r) dr &
- 
\beta \frac{\partial}{\partial s} \int_0^T g(r)\frac{\partial}{\partial r}K_V(s,r)dr
\\
&
+  \mu_\eps^2 \frac{\partial}{\partial s}\int_0^T \psi(r)
\frac{\partial}{\partial r} K_W(s,r)   
 dr =0,
\end{align*}
with  $\psi(r)$ as defined in \eqref{psidef}.
In terms of the transformation introduced in \eqref{vfa},
this equation is equivalent to 
\[
\frac \partial {\partial s}  v_{g, \alpha_2 }(s) 
- \beta   v_{g, \alpha_2}(s) + \mu_\eps^2  v_{\psi, \alpha_1}(s)  =0.
\]
Applying the Laplace transform to both sides and using the condition $v_{g,\alpha_2}(0)$ $ = 0$, implied by \eqref{maineq}, we obtain
\begin{equation}\label{eq:6.1}
e^{-zT}v_{g,\alpha_2}(T) + (z -\beta) \widehat v_{g, \alpha_2 }(z) 
+ 
\mu_\eps^2  \widehat v_{\psi, \alpha_1}(z)  =0.
\end{equation} 
Similarly, the Laplace transform of \eqref{psidef} yields the relation
\begin{equation}\label{psihatg}
 (z+\beta)\widehat \psi(z)  =   \psi(0)+  \frac 1 {\mu_\eps} e^{-zT} -  \widehat g(z).
\end{equation}
Combining \eqref{eq:6.1}, \eqref{psihatg} and  \eqref{eq:6.3} with $f:=g$ and $f:=\psi$ gives the  
representation, claimed in \eqref{gTz_expr}
$$
\widehat \psi(z) \Lambda(z) = 
e^{-zT} \Phi_1(-z)  +     \Phi_0(z) + 
    N_{\alpha_2}(z) (z -\beta)  \Big(\psi(0)+  \frac 1 {\mu_\eps} e^{-zT} \Big),
$$
with
\begin{equation}\label{Phi0Phi1}
\begin{aligned}
\Phi_0(z) :=\, &  \phantom{+\ }\Psi_{g,0}(z)(z -\beta) +  \mu_\eps^2    \Psi_{\psi,0}(z), \\
\Phi_1(z) :=\, & -\Psi_{g,1}(z)(z +\beta) +    \mu_\eps^2  \Psi_{\psi,1}(z)+ v_{g,\alpha_2}(T).
\end{aligned}
\end{equation}
The Cauchy integrals in \eqref{Psidefn} define sectionally holomorphic functions on $\mathbb{C}\setminus \Real_+$, and 
the estimates \eqref{apriori_est} are derived from \eqref{Psidefn} and \eqref{Phi0Phi1} by standard calculations.  

\medskip 
\noindent 
2. Subtracting the limits of equation 
\begin{equation}\label{eq:6.3g}
\widehat v_{g, \alpha_2} (z)  =  N_{\alpha_2}(z) \widehat g(z)+ e^{-zT} \Psi_{g,1}(-z) +  \Psi_{g,0}(z)
\end{equation}
as $z\to t\in \Real_+$ in the upper and lower half-planes, gives the boundary condition 
\[
\Psi_{g,0}^+(t)-\Psi_{g,0}^-(t)=  -\big(N_{\alpha_2}^+(t) - N_{\alpha_2}^-(t)\big)\widehat g(t), \quad t>0.
\]
Since the function in the right hand side is H\"older on $\Real_+\cup \{\infty\}$, and  $\Psi_{g,0}(z)$ vanishes as $z\to \infty$, 
applying the Sokhotski-Plemelj formula gives
$$
\Psi_{g,0}(z) = -\frac 1 {2\pi i} \int_0^\infty \frac{1}{t-z}\big(N_{\alpha_2}^+(t) - N_{\alpha_2}^-(t)\big)\widehat g(t)dt, \quad 
z\in \mathbb{C}\setminus \Real_+.
$$
Condition \eqref{cond1} now follows, since $v_{g,\alpha_2}(0)=0$ and, in view of \eqref{eq:6.3g}, 
\[
v_{g,\alpha_2}(0) =   \lim_{\Re(z)\to \infty} z \widehat v_{g, \alpha_2}(z) =
\lim_{\Re(z)\to \infty} z 
\Big(
N_{\alpha_2}(z) \widehat g(z) +  \Psi_{g,0}(z)
\Big).
\]
Formula \eqref{PTfla} is obtained similarly, since by \eqref{PTfun}, 
\[
P_T =   \frac 1 {\mu_\eps} v_{g, \alpha_2} (T)  = \frac 1 {\mu_\eps} \lim_{\Re(z)\to\infty} ze^{-zT}\widehat v_{g, \alpha_2} (-z). 
\]
\end{proof}

\begin{rem}\label{rem:4.2}
For $\alpha_2 \in (0,1)$, the first term  in the brackets in both \eqref{cond1} and \eqref{PTfla} vanishes, and 
these equations reduce to 
\begin{equation}\label{cond1_alpha_small}
\frac 1 {2\pi i}\int_0^\infty  \Big(N_{\alpha_2}^+(t)-N_{\alpha_2}^-(t)\Big)
\widehat g(t)dt
=0
\end{equation} 
and  
\begin{equation}\label{PT_alpha_small}
P_T =   \frac 1 {\mu_\eps}  
   \frac 1 {2\pi i}  \int_0^\infty   \Big(N_{\alpha_2}^+(t)   - N_{\alpha_2}^-(t) \Big) e^{-tT}\widehat g(-t) dt,
\end{equation}
respectively. 
For $\alpha_2\in (1,2)$ this first term diverges to infinity as $z\to\infty$ and compensates 
by the leading asymptotic term of the integral. Hence the useful information is actually contained in the second order 
asymptotics of these expressions. 
\end{rem}

\noindent
The next lemma reveals several important properties of the structural function.

\begin{lem}\label{lem:4.2}
The function $\Lambda(z)$ defined in \eqref{Lambdaz} is sectionally holomorphic on $\mathbb{C}\setminus \Real$  with a jump 
discontinuity across the real line, and 
its limits $\Lambda^\pm (t)$, $t\in \Real$ satisfy  
\begin{equation}\label{Lambdasym}
\Lambda^+(t) = \Lambda^-(-t)\quad \text{and}\quad \overline{\Lambda^+(t)} = \Lambda^-(t).
\end{equation}
It does not vanish on the cut plane, except, possibly, at simple zeros. More precisely, 

\medskip

\begin{enumerate}
\addtolength{\itemsep}{0.7\baselineskip}
\renewcommand{\theenumi}{\alph{enumi}}
 
\item   $\Lambda(z)$ has  no zeros if $\alpha_1>\alpha_2$, or

\item   a pair of purely real zeros at $\pm t_0$ with $t_0 = \sqrt{\beta^2+\mu_\eps^2}$, if $\alpha_1=\alpha_2$,  or

\item   complex zeros at $\pm z_0$ and $\pm \overline{z}_0$ for some $z_0$ 
with $\arg(z_0) \in (0, \frac \pi 2)$, if $\alpha_1<\alpha_2$.

\end{enumerate}
 
\end{lem}

\begin{proof}
The analytic structure of $\Lambda(z)$, the discontinuity and properties \eqref{Lambdasym} are inherited from 
$N_\alpha(z)$, cf. \eqref{Nsym}. 
The symmetric structure of zeros is obvious from the definition of $\Lambda(z)$, and hence it suffices to 
locate its zeros only in the first quadrant. 
Since $N_\alpha(z)$ may vanish only at the origin, for $z := \rho e^{i\phi}$ with $\rho\in \Real_+$ and 
$\phi \in [0,\frac \pi 2]$,
\[
-\frac{\Lambda(z)}{N_{\alpha_2}(z)} =\;   \mu_\eps^2 \frac{N_{\alpha_1}(z)}{N_{\alpha_2}(z)}- z^2+\beta^2= 
\mu^2_\eps   \frac{\kappa_{\alpha_1}  }{\kappa_{\alpha_2} } \rho^{\alpha_1-\alpha_2}e^{i(\phi -\frac \pi 2)(\alpha_1-\alpha_2)} 
-   \rho^2e^{2\phi i}+\beta^2.
\]
Equating the imaginary and real parts of this expression to zero, we get  
\begin{align*}
&
\rho^2\sin (2\phi)
-\mu_\eps^2   \frac{\kappa_{\alpha_1}  }{\kappa_{\alpha_2} } \rho^{-\delta}\sin  (\tfrac \pi 2-\phi)\delta  
   =0, 
\\
&
\rho^2\cos(2\phi)-\mu_\eps^2   \frac{\kappa_{\alpha_1}  }{\kappa_{\alpha_2} } \rho^{-\delta}\cos  (\tfrac \pi 2-\phi)\delta
    = \beta^2,
\end{align*}
where $\delta:= \alpha_2-\alpha_1$. The angle $\phi = \frac \pi 2$ is inconsistent with the second equation
and $\phi =0$ with the first equation, unless $\delta=0$ as well. In this case, that is, when $\alpha_1=\alpha_2$, 
$\phi=0$ is the only possibility, and there are two real zeros as claimed. 

If $\alpha_1>\alpha_2$ the first equation is inconsistent for any $\rho>0$, and hence $\Lambda(z)$ does not have zeros in this case.
For $\alpha_1<\alpha_2$ the absolute value $\rho$ can be expressed in terms of $\phi$ using the first 
equation 
\begin{equation}\label{rhophi}
\rho 
=
\left(\mu_\eps^2   \frac{\kappa_{\alpha_1}  }{\kappa_{\alpha_2} }    
\right)^{ \frac 1 {2+\delta}} 
\left(  \frac {\sin (\tfrac \pi 2-\phi)\delta}{\sin (2\phi)}  
\right)^{ \frac 1 {2+\delta}}.
\end{equation}  
Plugging this into the second equation we get 
\[
\sin \big(\widetilde \phi\delta\big)^{ -\frac {\delta} {2+\delta}}
\sin (2\widetilde \phi)^{ -\frac 2 {2+\delta}} 
\sin \big(\widetilde \phi(2+\delta)  
\big)
    = 
-\beta^2\left(\mu_\eps^2   \frac{\kappa_{\alpha_1}  }{\kappa_{\alpha_2} } \right)^{ -\frac 2 {2+\delta}}
\]
where $\widetilde \phi:= \tfrac \pi 2 -\phi\in (0, \frac \pi 2)$ was defined for brevity. 
The left hand side is a continuous decreasing function of $\widetilde \phi$,  it diverges to $-\infty$ as 
$\widetilde \phi\to \tfrac\pi 2$ and has a positive finite limit at $\widetilde \phi=0$.
Hence this equation has the unique root $\phi_0$ and, consequently, $\Lambda(z)$ has the unique zero in the first quadrant 
at $z_0:= \rho_0 e^{i\phi_0}$ with $\rho_0$ given by \eqref{rhophi} with $\phi$ replaced by $\phi_0$.
\end{proof}

\subsection{The equivalent problem}\label{sec:equiv}
In this subsection we formulate a different problem, which is equivalent to solving equation \eqref{maineq}.
The key observation to this end is that all singularities in expression \eqref{gTz_expr} must be removable,
since the Laplace transform  in \eqref{gTz} defines an entire function $\widehat g(z)$. 
In particular, its limits as $z\to t\in \Real$ in the upper and lower half-planes must coincide, which implies
\begin{align*}
&
(t+\beta) \frac{\Phi_0^+(t)+e^{-tT} \Phi_1(-t)}{\Lambda^+(t)} +  \mu_\eps^2 \frac{N_{\alpha_1}^+(t)}{\Lambda^+(t)}
\Big(\psi(0)+\frac {1}{\mu_\eps}e^{-tT}\Big) = \\
  &
(t+\beta) \frac{\Phi_0^-(t)+e^{-tT} \Phi_1(-t)}{\Lambda^-(t)} +  \mu_\eps^2 \frac{N_{\alpha_1}^-(t)}{\Lambda^-(t)}
\Big(\psi(0)+\frac {1}{\mu_\eps}e^{-tT}\Big), \quad t\in \Real_+
\end{align*}
and
\begin{align*}
&
 (t+\beta) \frac{\Phi_0(t)+e^{-tT} \Phi_1^-(-t)}{\Lambda^-(t)} +  \mu_\eps^2 \frac{N_{\alpha_1}^-(t)}{\Lambda^-(t)}
\Big(\psi(0)+\frac {1}{\mu_\eps}e^{-tT}\Big)
= \\
&
 (t+\beta) \frac{\Phi_0(t)+e^{-tT} \Phi_1^+(-t)}{\Lambda(z)} +  \mu_\eps^2 \frac{N_{\alpha_1}^-(t)}{\Lambda^-(t)}
\Big(\psi(0)+\frac {1}{\mu_\eps}e^{-tT}\Big), \quad t\in \Real_-
\end{align*}
In view of symmetries \eqref{Lambdasym} and formula \eqref{Lambdaz}, these equations can be written as 
\begin{equation}\label{eq:6.11}
\begin{aligned}
 \Phi_0^+  & (t)   -  \frac {\Lambda^+(t)}{\Lambda^-(t)}    \Phi_0^-(t) 
=\, 
e^{-tT} \Phi_1(-t)\Big(\frac{ \Lambda^+(t) }{\Lambda^-(t)}-1\Big) \\
& 
+
\Big(\frac{\Lambda^+(t)}{\Lambda^-(t)} N_{\alpha_2}^-(t)-  N_{\alpha_2}^+(t)\Big)(t -\beta)  \Big(\psi(0)+  \frac 1 {\mu_\eps} e^{-tT} \Big), 
\quad t\in \Real_+
 \\
 \Phi_1^+  & (t)  -   \frac{\Lambda^+(t)}{ \Lambda^-(t)}  \Phi_1^-(t)
=\,  
  e^{-tT}\Phi_0(-t) \Big( \frac{\Lambda^+(t) }{ \Lambda^-(t)} - 1\Big) \\
&   - \Big(   \frac {\Lambda^+(t)} { \Lambda^-(t)} N_{\alpha_2}^-(t)    
   -   N_{\alpha_2}^+(t) \Big) (t +\beta)  \Big(e^{-tT}\psi(0)+  \frac 1 {\mu_\eps}   \Big), \quad t\in \Real_+.
\end{aligned}
\end{equation}
In addition, removal of the poles in \eqref{gTz_expr} implies that the expression in the brackets therein
must vanish at the zeros of $\Lambda(z)$,
\begin{equation}\label{poles}
\begin{aligned}
& 
(z+\beta) \Big(\Phi_0(z)+e^{-zT} \Phi_1(-z)\Big) +  \mu_\eps^2  N_{\alpha_1}(z) 
\Big(\psi(0)+\frac {e^{-zT}}{\mu_\eps}\Big)=0, \\ 
& \quad \forall z\in \Big\{\zeta : \Lambda(\zeta)=0\Big\}.
\end{aligned}
\end{equation}

At this point the proof splits into several cases, corresponding to the three possible zeros configurations of $\Lambda(z)$,
described in Lemma \ref{lem:4.2}, and the computation of filtering error, as explained in Remark \ref{rem:4.2}.  
While the specific calculations are somewhat different in each case, they are based on the same technique, which we 
will detail for $\alpha_1>\alpha_2\in (0,1)$, omitting all other cases.

Define $\theta(t) := \arg\big(\Lambda^+(t)\big)$, choosing the argument branch so that $\theta(t)$ is continuous on 
$(0,\infty)$ and $\theta(\infty) := \lim_{t\to\infty}\theta(t)$ belongs to the interval $(-\pi, \pi)$. 
This defines $\theta(t)$ uniquely, and, for $\alpha_1>\alpha_2$,
\[
\theta(\infty) = \frac{1-\alpha_2}{2}\pi
\quad 
\text{and}
\quad
\theta(0+) = \frac{1-\alpha_2}{2}\pi+\pi.
\]
In what follows we will need a function $X(z)$, which is sectionally holomorphic on $\mathbb{C}\setminus \Real_+$, satisfies 
the boundary condition 
\begin{equation}\label{homo}
\frac{X^+(t)}{X^-(t)} = \frac{\Lambda^+(t)}{\Lambda^-(t)} = e^{2i\theta(t)}, \quad t\in \Real_+,
\end{equation}
and does not vanish on the cut plane. 
Finding all such functions is known as the Hilbert boundary value problem, whose solutions are given by the 
Sokhotski-Plemelj formula 
\begin{equation}\label{Xz}
X(z) = (-z)^{k-\theta(\infty)/\pi}\exp \left(\frac 1 \pi \int_0^\infty \frac {\theta(t)-\theta(\infty)}{t-z}dt\right),
\end{equation}
where $k$ is an arbitrary integer. The choice of $k$ controls the growth of $X(z)$ at the origin and at infinity
\begin{equation}\label{Xzest}
 X(z)  = 
\begin{cases}
O(z^{k-\theta(0+)/\pi}), & z\to 0, \\
O(z^{k-\theta(\infty)/\pi}), & z\to \infty.
\end{cases}
\end{equation}

Define a pair of auxiliary functions 
\begin{equation}\label{SzDz}
S(z)   := \frac{\Phi_0(z)+\Phi_1(z)}{2X(z)}\quad\text{and}\quad
D(z)   := \frac{\Phi_0(z)-\Phi_1(z)}{2X(z)}.
\end{equation}
In view of \eqref{eq:6.11} and \eqref{homo}, these functions satisfy the {\em decoupled} boundary conditions 
\begin{equation}\label{decoupled}
\begin{aligned}
S^+(t) - S^-(t) = & \phantom{+\ \, } 2i e^{-tT}h(t)  S(-t)  + f_S(t), \\
D^+(t) - D^-(t) = & - 2i e^{-tT}h(t) D (-t)  + f_D(t),
\end{aligned}
\qquad t\in \Real_+,
\end{equation}
where we defined 
\begin{equation}\label{fSfD}
\begin{aligned}
f_S(t) :=  \frac 1 2 & \Big(\frac {N_{\alpha_2}^-(t)}{X^-(t)} -  \frac{N_{\alpha_2}^+(t)}{X^+(t)}\Big)\cdot \\
&
\Big( 
(t -\beta)  \Big(\psi(0)+  \frac {1} {\mu_\eps}e^{-tT}  \Big)
- 
(t +\beta)  \Big(e^{-tT}\psi(0)+  \frac 1 {\mu_\eps}   \Big)\Big), \\
f_D(t) := \frac 1 2 &  \Big(\frac{N_{\alpha_2}^-(t)}{X^-(t)} -  \frac{N_{\alpha_2}^+(t)}{X^+(t)}\Big)\cdot \\
&
\Big(
(t -\beta)  \Big(\psi(0)+  \frac {1} {\mu_\eps}e^{-tT}  \Big)+  (t +\beta)  \Big(e^{-tT}\psi(0)+  \frac 1 {\mu_\eps}   \Big)
\Big),
\end{aligned}
\end{equation}
and the real valued function
\[
h(t):=   \frac{X(-t)}{X^+(t)}e^{ i\theta(t)}\sin \theta(t) = 
 \exp\left(-\frac 1 \pi \int_0^\infty \theta'(s) \log \left|\frac{t+s}{t-s}\right|ds\right)\sin \theta(t).
\]

Due to estimates \eqref{apriori_est} and \eqref{Xzest}, the choice $k=1$ in \eqref{Xz} guarantees that $S(-t)$ and $D(-t)$
are integrable and, moreover, square integrable near the origin, and implies that $S(z)$ and $D(z)$ vanish as $z\to\infty$. 
Hence by the Sokhotski-Plemelj theorem, applied to \eqref{decoupled}, these functions must satisfy the equations
\begin{equation}\label{eq:6.13}
\begin{aligned}
S(z) = & \phantom{+\ }\frac 1 \pi \int_0^\infty  \frac{ e^{-tT}h(t)}{t-z} S(-t)dt  + F_S(z), \\
D(z) = & -\frac 1 \pi \int_0^\infty  \frac{ e^{-tT}h(t)}{t-z} D(-t)dt  + F_D(z),
\end{aligned}
\end{equation}
where we defined 
\begin{equation}\label{F_SF_D}
F_S(z) :=   \frac 1{2\pi i} \int_0^\infty \frac {f_S(t)} {t-z}  dt \quad \text{and}\quad 
F_D(z) :=   \frac 1{2\pi i} \int_0^\infty \frac {f_D(t)} {t-z}  dt.
\end{equation}

Consider now a pair of auxiliary integral equations 
\begin{equation}\label{pqeq}
\begin{aligned}
p(t) = & \phantom{+\ }\frac 1 \pi \int_0^\infty  \frac{ e^{-\tau T}h(\tau )}{\tau+t} p(\tau)d\tau  + F_S(-t), \\
q(t) = & -\frac 1 \pi \int_0^\infty  \frac{ e^{-\tau T}h(\tau)}{\tau+t}q(\tau)d\tau  + F_D(-t),
\end{aligned}  \qquad t\in \Real_+.
\end{equation}
In view of \eqref{fSfD} the restrictions $F_S(-t)$ and $F_D(-t)$ are real valued functions. 
The operator in the right hand side
\[
(A f)(t)=\frac 1 \pi \int_0^\infty  \frac{ e^{-\tau T}h(\tau )}{\tau+t} f(\tau)d\tau
\]
is a contraction on $L^2(\Real_+)$, see \cite[Lemma 5.6]{ChK}, and a calculation as in \cite[Lemma 5.7]{ChK}  shows that
$F_S, F_D\in L^2(\Real_+)$. Consequently,  equations \eqref{pqeq} have unique solutions  $p, q\in L^2(\Real_+)$.

Comparing \eqref{eq:6.13} and \eqref{pqeq} shows that 
\[
S(z) = p(-z) \quad \text{and}\quad D(z)=q(-z), \quad z\in \mathbb{C}\setminus \Real_+,
\]
where $p(z)$ and $q(z)$ are the analytic extensions. Then, by definition \eqref{SzDz},
\begin{equation}\label{PhiPhi}
\begin{aligned}
\Phi_0(z) = &  X(z) \big(p(-z)+q(-z)\big), \\
\Phi_1(z) = & X(z) \big(p(-z)-q(-z)\big).
\end{aligned}
\end{equation}

Let us summarize our findings so far. Given the unique solutions to the integral equations \eqref{pqeq}, 
we can compute the functions $\Phi_0(z)$ and $\Phi_1(z)$ by means of \eqref{PhiPhi} and plug them into \eqref{gTz_expr}.
The constant  $\psi(0)$ can be found by plugging the obtained expression for $\widehat g(z)$ into \eqref{cond1},
or equivalently in this case, into \eqref{cond1_alpha_small}. Applying the inverse Laplace transform to $\widehat g(z)$, 
gives a function which solves \eqref{maineq} and belongs to $L^1([0,T])\cap \Lambda^{H-\frac 1 2}_T$, where $\Lambda^{H-\frac 1 2}_T$ is a space of nonrandom functions, on which the stochastic integral with respect to fBm can be defined and has suitable properties
(see Appendix \ref{sec:A}). 
Thus  the original equation is reduced to an equivalent problem of 
solving integral equations \eqref{pqeq}. The filtering error $P_T$ is found by substitution of the 
expression for $\widehat g(z)$ into \eqref{PT_alpha_small}.

\subsection{Asymptotic analysis}
While for any fixed values of the parameters, the equivalent problem derived above does not appear 
any simpler than the original equation, it does simplify drastically when either of the limits $T\to\infty$ or 
$\eps \to 0$ is taken. 
The key to the asymptotic analysis are the estimates 
\begin{equation}\label{asym_est}
\big|p(z) - F_S(-z)\big|\le C \frac 1 z \frac  1 T,\quad \big|q(z) - F_D(-z)\big|\le C \frac 1 z \frac  1 T,  
\end{equation}
where $C$ is a constant independent of $T$ and $\eps$. These bounds are derived exactly as in \cite[Lemma 5.7]{ChK}
and we omit the proof.

\subsubsection{Large time asymptotics}
Upon substitution of expression \eqref{gTz_expr} into the integral in \eqref{cond1_alpha_small}, the latter simplifies, 
asymptotically as $T\to\infty$, to 
\begin{align*} 
&
\int_0^\infty  \Big(N_{\alpha_2}^+(t)-N_{\alpha_2}^-(t)\Big)
\widehat g(t)dt \asymp \\
&
\int_0^\infty  \Big(N_{\alpha_2}^+(t)-N_{\alpha_2}^-(t)\Big) \bigg(
(t+\beta) \frac{\Phi_0^+(t) }{\Lambda^+(t)} dt 
+ \psi(0)\mu_\eps^2  \frac{N_{\alpha_1}^+(t)}{\Lambda^+(t)} \bigg) dt.
\end{align*}
Due to \eqref{PhiPhi} and estimates \eqref{asym_est}, the first term satisfies  
\begin{align*}
& 
\int_0^\infty  \Big(N_{\alpha_2}^+(t)-N_{\alpha_2}^-(t)\Big)  (t+\beta) \frac{\Phi_0^+(t) }{\Lambda^+(t)} dt \asymp\\
&
\int_0^\infty  \Big(N_{\alpha_2}^+(t)-N_{\alpha_2}^-(t)\Big)  (t+\beta) \frac{X^+(t)  }{\Lambda^+(t)} \big(p^-(-t)+q^-(-t)\big)dt =\\
&
\int_0^\infty  \Big(N_{\alpha_2}^+(t)-N_{\alpha_2}^-(t)\Big)  (t+\beta) \frac{X^+(t)  }{\Lambda^+(t)} \Big(F_S^+(t)+F_D^+(t)\Big)dt,
\end{align*}
where, by definitions \eqref{fSfD},
\begin{equation}\label{FFSD}
\begin{aligned}
&
F_S(z)+F_D(z) = \frac 1 {2\pi i} \int_0^\infty \frac {f_S(t)+f_D(t) }{t-z}dt = \\
&
\frac 1 { 2\pi i} \int_0^\infty \frac {1 }{t-z} \Big(\frac {N_{\alpha_2}^-(t)}{X^-(t)} -  \frac{N_{\alpha_2}^+(t)}{X^+(t)}\Big)(t -\beta)
\Big( 
  \psi(0)+   \frac 1 {\mu_\eps} e^{-tT}  
\Big)dt \asymp\\
& 
\psi(0) \frac 1 { 2\pi i} \int_0^\infty 
\Big(\frac {N_{\alpha_2}^-(t)}{X^-(t)} -  \frac{N_{\alpha_2}^+(t)}{X^+(t)}\Big)\frac {t -\beta }{t-z}  dt =: \psi(0) R(z; \beta, \mu_\eps).
\end{aligned}  
\end{equation}
It follows that 
\begin{equation}\label{Ifla} 
\int_0^\infty  \Big(N_{\alpha_2}^+(t)-N_{\alpha_2}^-(t)\Big)
\widehat g(t)dt
\asymp \psi(0) I(\beta, \mu_\eps), 
\end{equation}
where the quantity 
\begin{align*}
&
I(\beta, \mu_\eps) :=  \\
&\frac 1{2\pi i}\int_0^\infty  \Big(N_{\alpha_2}^+(t)-N_{\alpha_2}^-(t)\Big)  
\bigg(
(t+\beta) \frac{X^+(t)  }{\Lambda^+(t)} R^+ (t;\beta, \mu_\eps)   
+  \mu_\eps^2  \frac{N_{\alpha_1}^+(t)}{\Lambda^+(t)} 
\bigg) dt
\end{align*}
does not depend on $T$. A lengthy but otherwise direct calculation shows that this expression is nonzero and 
therefore condition \eqref{cond1_alpha_small} implies  that $\psi(0) \to 0$ as $T\to\infty$. 

Similarly, we can simplify expression \eqref{PT_alpha_small},
\begin{align*}
&
P_T =  
\frac 1 {\mu_\eps} \frac 1 {2\pi i}  \int_0^\infty   \Big(N_{\alpha_2}^+(t)   - N_{\alpha_2}^-(t) \Big) e^{-tT}\widehat g(-t) dt \asymp \\
&
\phantom{+\ }\frac 1 {2\pi i}  \int_0^\infty   \Big(N_{\alpha_2}^+(t)   - N_{\alpha_2}^-(t) \Big) \Big(
( t-\beta) \frac 1 {\mu_\eps} \frac{ \Phi_1^+( t)}{\Lambda^+(t)} 
- \frac{N_{\alpha_1}^+(t)}{\Lambda^+(t)}
\Big)
dt   \asymp\\
& 
\phantom{+\ }\frac 1 {2\pi i}  \int_0^\infty   \Big(N_{\alpha_2}^+(t)   - N_{\alpha_2}^-(t) \Big)  
( t-\beta) \frac 1 {\mu_\eps} \frac{   X^+(t) }{\Lambda^+(t)} \big(F_S^+(t)-F_D^+(t)\big)
dt   \\
& 
-\frac 1 {2\pi i}  \int_0^\infty   \Big(N_{\alpha_2}^+(t)   - N_{\alpha_2}^-(t) \Big)  \frac{N_{\alpha_1}^+(t)}{\Lambda^+(t)}dt, 
\end{align*}
where  $e^{-tT}\widehat g(-t)$ is computed using \eqref{gTz_expr}.  
Since $\psi(0)$ remains bounded as $T\to \infty$, \eqref{fSfD} implies
\begin{align}\label{Qz}
F_S(z)-F_D(z) = & \frac 1 {2\pi i} \int_0^\infty \frac {f_S(t)-f_D(t) }{t-z}dt \asymp \\
&
\nonumber
    \frac 1 { 2\pi i} \frac 1 {\mu_\eps} \int_0^\infty 
 \Big(  \frac{N_{\alpha_2}^+(t)}{X^+(t)} -\frac {N_{\alpha_2}^-(t)}{X^-(t)}   \Big) \frac {t +\beta }{t-z}  dt 
 =: Q(z; \beta, \mu_\eps),
\end{align}  
and  hence, as claimed in \eqref{large-time}, $P_T$ converges to the limit
\begin{equation}\label{Pinfty_expr}
P_\infty(\beta, \mu_\eps) :=
\frac 1 {2\pi i}  \int_0^\infty   \Big(N_{\alpha_2}^+(t)   - N_{\alpha_2}^-(t) \Big)  
\bigg(
 \frac { t-\beta} {\mu_\eps} \frac{   X^+(t) }{\Lambda^+(t)} Q^+(t;\beta, \mu_\eps)
-  \frac{N_{\alpha_1}^+(t)}{\Lambda^+(xt)}\bigg)
dt. 
\end{equation}

\subsubsection{Small noise asymptotics}
To emphasise the dependence on $\eps$ and other parameters we will add them to the notations, writing 
$\Lambda(z;\beta, \mu_\eps)$ for $\Lambda(z)$, etc. 
In view of definitions \eqref{Nalpha} and \eqref{Lambdaz},  the structural function satisfies the scaling property
\[
\Lambda\big(\eps^{-\gamma}z;\beta, \mu_\eps \big) =   
\eps^{-\gamma(1+\alpha_2)}  \Lambda\big(z; \eps^\gamma \beta, \mu\big),\qquad \gamma := \frac 1 {2+\alpha_2-\alpha_1}>0.
\]
Consequently,  
$
\theta(\eps^{-\gamma} t;\beta, \mu_\eps) 
=   \theta\big(t; \eps^\gamma\beta, \mu\big) 
$
and, by definition \eqref{Xz},
\begin{align*}
X\big(\eps^{-\gamma} z; \beta, \mu_\eps\big) =\ &
(-\eps^{-\gamma} z)^{1-\frac{1-\alpha_2}{2}}
\exp \left(\frac 1 \pi \int_0^\infty \frac {\theta( t; \eps^\gamma\beta, \mu)-\theta(\infty)}{  t- z}d t\right) = \\
&
\eps^{-\frac 1 2\gamma(1+\alpha_2)} X\big(z; \eps^\gamma \beta, \mu\big).
\end{align*}
Substituting formula \eqref{gTz_expr}, expressions \eqref{PhiPhi} and estimates \eqref{asym_est} into the integral in 
\eqref{cond1_alpha_small} and changing the integration variable accordingly, we obtain  
\begin{align*}
&
 \int_0^\infty  \Big(N_{\alpha_2}^+(t)-N_{\alpha_2}^-(t)\Big)\widehat g(t)dt \asymp \\
&
-   \eps^{-  \gamma  \frac{1+\alpha_2}2}\int_0^\infty  
\Big(N_{\alpha_2}^+(t)-N_{\alpha_2}^-(t)\Big)  (t + \eps^{ \gamma}\beta) 
\frac{
  X^+\big(t; \eps^\gamma \beta, \mu\big)\big(F_S^+(\eps^{-\gamma}t)+F_D^+(\eps^{-\gamma}t)\big)
 }
 {  \Lambda^+\big(t; \eps^\gamma \beta, \mu\big)} dt \\
&
-
 \eps^{-\gamma \alpha_2}\mu^2\psi(0) \int_0^\infty  \Big(N_{\alpha_2}^+(t)-N_{\alpha_2}^-(t)\Big) 
    \frac{N_{\alpha_1}^+(t)}{  \Lambda^+\big(t; \eps^\gamma \beta, \mu\big)}
 dt,
\end{align*}
as $\eps\to 0$. 
Here, in view of \eqref{fSfD} and \eqref{F_SF_D}, 
\begin{align*}
&
F_S(\eps^{-\gamma}z)+F_D(\eps^{-\gamma} z) = \frac 1 {2\pi i} \int_0^\infty \frac {f_S(\eps^{-\gamma}t)+f_D(\eps^{-\gamma}t) }{ t-  z}d t \asymp \\
&
\eps^{\frac 1 2\gamma(1-\alpha_2)}\psi(0) \frac 1 {2\pi i} 
\int_0^\infty 
 \bigg(\frac {  N_{\alpha_2}^-(  t)}{
  X^-\big(t; \eps^\gamma \beta, \mu\big)
  } -  \frac{N_{\alpha_2}^+( t)}{
   X^+\big(t; \eps^\gamma \beta, \mu\big)
   }\bigg)\frac {t -\eps^\gamma \beta }{ t-  z}
d t = \\
&
\eps^{\frac 1 2\gamma(1-\alpha_2)}\psi(0) R\big(z; \eps^\gamma \beta, \mu\big),
\end{align*}
where $R(\cdot)$ was defined in \eqref{FFSD}. Consequently, cf. \eqref{Ifla},
\[
\int_0^\infty  \Big(N_{\alpha_2}^+(t)-N_{\alpha_2}^-(t)\Big)\widehat g(t)dt
\asymp \eps^{-\gamma \alpha_2} \psi(0) I(\eps^\gamma \beta, \mu), 
\]
and thus condition \eqref{cond1_alpha_small} implies $\psi(0) = o(\eps^{\gamma \alpha_2})$ as $\eps\to 0$.

The filtering error asymptotics is deduced from \eqref{PT_alpha_small} by similar calculations,  
\begin{align*}
&
P_T(\beta, \mu_\eps) \asymp\,  
 \eps^{ \gamma(2-\alpha_1)} \frac 1 {2\pi i}  
\int_0^\infty   \Big(N_{\alpha_2}^+( t)   - N_{\alpha_2}^-( t) \Big) \cdot \\
&
\bigg(
 \eps^{  -\frac \gamma 2(3-\alpha_1)}(t-\eps^\gamma \beta)\frac 1 \mu \frac{ 
 X^+\big(t; \eps^\gamma \beta, \mu\big) 
}{
 \Lambda^+\big(t; \eps^\gamma \beta, \mu\big)
} \Big(F_S^+(\eps^{-\gamma}t) -F_D^+(\eps^{-\gamma}t)\Big)
-      
\frac{N_{\alpha_1}^+( t)}{\Lambda^+\big(t; \eps^\gamma \beta, \mu\big)}
\bigg)dt.
\end{align*}
Here, cf. \eqref{Qz}, 
\begin{align*}
&
 F_S (\eps^{-\gamma}z) -F_D^+(\eps^{-\gamma}z)  =
\frac 1{2\pi i} \int_0^\infty \frac {f_S(\eps^{-\gamma}t)-f_D(\eps^{-\gamma}t)} {t-z}  dt =\\
&
\eps^{\frac 1 2 + \frac 1  2 \gamma (1-\alpha_2)}\frac 1{2\pi i} \frac 1 \mu\int_0^\infty 
\Big(\frac {N_{\alpha_2}^+( t)}
{  X^+\big(t; \eps^\gamma \beta, \mu\big)} -  \frac{N_{\alpha_2}^-( t)}{
 X^-\big(t; \eps^\gamma \beta, \mu\big)
}\Big) \frac {t +\eps^\gamma\beta} {t-z} dt,
\end{align*}
and consequently   
\begin{align*}
&
P_T(\beta, \mu_\eps)     \asymp\, 
\eps^{ \gamma(2-\alpha_1)}
 \frac 1 {2\pi i}  
\int_0^\infty   \Big(N_{\alpha_2}^+( t)   - N_{\alpha_2}^-( t) \Big) \cdot \\
&
  \bigg(
  (t-\eps^\gamma \beta)\frac 1 \mu \frac{ 
 X^+\big(t; \eps^\gamma \beta, \mu\big) 
}{
 \Lambda^+\big(t; \eps^\gamma \beta, \mu\big)
}  Q(t;\eps^{\gamma }\beta, \mu)
-      
\frac{N_{\alpha_1}^+( t)}{\Lambda^+\big(t; \eps^\gamma \beta, \mu\big)}
\bigg)dt  
\asymp\   \eps^{ \gamma(2-\alpha_1)} P_\infty(0, \mu),
\end{align*}
where $P_\infty(\cdot)$ is exactly the function obtained in \eqref{Pinfty_expr}. 
This proves the asymptotics claimed in \eqref{small-noise}.

\section{Proof of Theorem \ref{thm:1}}\label{sec:case1}
In this section we derive the large time limit \eqref{KlLBfla}, from which small noise asymptotics \eqref{KlLBfla_small_noise} 
follows by Theorem \ref{thm:main} in the obvious way.

\subsection{The equivalent problem} 
For $\alpha_1=\alpha_2 =: \alpha\in (0,2)$,  expression \eqref{gTz_expr} for the Laplace transform  reduces to
\begin{equation}\label{hatg_equal_alphas}
\widehat g(z) =   -\frac{z+\beta}{z^2-t_0^2} \frac{\Phi_0(z)+e^{-zT} \Phi_1(-z)}{  N_{\alpha}(z)}
-  \frac{\mu_\eps^2 }{z^2-t_0^2}\Big(\psi(0)+\frac {1}{\mu_\eps}e^{-zT}\Big),
\end{equation}
where $t_0^2 = \beta^2+\mu_\eps^2$, cf. \eqref{Lambdaz}. In this case,
\[
\frac{\Lambda^+(t)}{\Lambda^-(t)} = \frac{N_{\alpha}^+(t)}{ N_{\alpha}^-(t) } = 
e^{(1-\alpha)\pi i}, \quad t\in \Real, 
\]
and equations \eqref{eq:6.11} simplify to
\[
\begin{aligned}
& 
\Phi_0^+(t)   -e^{(1-\alpha)\pi i}  \Phi_0^-(t)
=
e^{-tT} \Phi_1(-t)\Big(e^{(1-\alpha)\pi i} -1\Big), \\
&
\Phi_1^+(t)   - e^{(1-\alpha)\pi i}\Phi_1^-(t)
=
e^{-tT}\Phi_0(-t)\Big( e^{(1-\alpha)\pi i}
-
1\Big),
\end{aligned}  \qquad t\in \Real_+.
\]
The sectionally holomorphic function in \eqref{Xz} reduces to 
\begin{equation}\label{Xzalpha}
X(z) = (-z)^{\frac {\alpha-1} 2},\quad z\in \mathbb C\setminus \Real_+,
\end{equation}
with the constant jump across the real line
\[
\frac{X^+(t)}{X^-(t)} =
e^{(1-\alpha)\pi i}, \quad t\in \Real_+.
\]

In this case the functions defined in \eqref{SzDz} satisfy, cf. \eqref{decoupled},  
\[
\begin{aligned}
S^+(t) - S^-(t) & = \phantom{+ } 2i e^{-tT} h  S(-t),
\\
D^+(t) - D^-(t) & = -  2i e^{-tT} h  D(-t),
\end{aligned}\qquad t\in \Real_+,
\]
with the constant $h= \sin \big(\tfrac {1-\alpha} 2\pi\big)$.
In view of estimates \eqref{apriori_est} and expression \eqref{Xzalpha},  functions $S(z)$ and $D(z)$ grow sublinearly 
as $z\to \infty$ and their restrictions to negative reals are (square) integrable near the origin. 
Consequently, by the Sokhotski-Plemelj theorem, cf. \eqref{eq:6.13},
\begin{align*}
S(z) = & \phantom{+\ }\frac 1 \pi \int_0^\infty \frac {e^{-tT} h}{t-z}S(-t)dt + k^S_0 \\
D(z) = & -\frac 1 \pi \int_0^\infty \frac {e^{-tT} h}{t-z}D(-t)dt + k^D_0,
\end{align*}
where $k^S_0$ and $k^D_0$ are some constants, yet to be determined. 
The relevant auxiliary integral equations in this case are
\[
\begin{aligned}
p_0(t) = & \phantom{+\ } \frac 1 \pi \int_0^\infty \frac {e^{-\tau T}h}{\tau+t}p_0(\tau)d\tau + 1,
\\
q_0(t) = & -\frac 1 \pi \int_0^\infty \frac {e^{-\tau T} h}{\tau+t}q_0(\tau)d\tau + 1,
\end{aligned}  \qquad t\in \Real_+.
\]
They have unique solutions, such that $A p_0, Aq_0\in L^2(\Real_+)$ and, by linearity, 
$
S(z) = k^S_0 p_0(-z) 
$
and 
$
D(z) = k^D_0 q_0(-z),
$
so that, cf. \eqref{PhiPhi},
\begin{equation}\label{PhicScD} 
\begin{aligned}
\Phi_0(z) = & X(z) \Big(k^S_0 p_0(-z)+k^D_0 q_0(-z)\Big),   \\
\Phi_1(z) = & X(z) \Big(k^S_0 p_0(-z)-k^D_0 q_0(-z)\Big).
\end{aligned}
\end{equation}

Substituting these formulas into \eqref{hatg_equal_alphas}, we obtain an expression for the Laplace transform, which depends 
on the unknown constants $\psi(0)$, $k^S_0$ and $k^D_0$. These constants can be found from the linear algebraic system, 
consisting of \eqref{cond1} and the two additional equations, obtained by the poles removal in \eqref{hatg_equal_alphas},
\begin{equation}\label{more_eq}
\begin{aligned}
& 
(t_0+\beta) \frac{\Phi_0^+(t_0)+e^{-t_0T} \Phi_1(-t_0)}{N_{\alpha}^+(t_0)}
+  \mu_\eps^2   \Big(\psi(0)+\frac {1}{\mu_\eps}e^{-t_0T}\Big)=0, \\
&
(t_0-\beta) \frac{e^{-t_0 T} \Phi_0(-t_0)+ \Phi_1^-(t_0)}{N_{\alpha}^-(t_0)}
-   \mu_\eps^2  \Big(e^{-t_0 T}\psi(0)+\frac {1}{\mu_\eps} \Big)=0.
\end{aligned}
\end{equation} 
Once this system is solved, the Laplace transform $\widehat g(z)$ becomes completely specified and the filtering error can  
be computed by means of equation \eqref{PTfla}.
  
\subsection{Large time limit $\boldsymbol{\alpha \in (0,1)}$}  
The main element of the asymptotic analysis is the estimates similar to \eqref{asym_est},
\[
\big|p_0(z) - 1\big|\le C \frac 1 z \frac  1 T,\quad \big|q_0(z) - 1\big|\le C \frac 1 z \frac  1 T.
\]
Due to these bounds and equations \eqref{PhicScD}, conditions \eqref{more_eq} simplify  as $T\to \infty$ to 
\begin{equation}\label{more_eq_asymp}
\begin{aligned}
& 
(t_0+\beta)  \frac{X^+(t_0)}{N_{\alpha}^+(t_0)} k^S_0   +(t_0+\beta) \frac{X^+(t_0)   }{N_{\alpha}^+(t_0)} k^D_0  
+  \mu_\eps^2   \psi(0) \asymp 0, \\
&
 (t_0-\beta) \frac{ X^-(t_0)}{N_{\alpha}^-(t_0)} k^S_0 - (t_0-\beta)\frac{ X^-(t_0)}{N_{\alpha}^-(t_0)} k^D_0 
  \asymp     \mu_\eps.
\end{aligned}\end{equation} 
Further calculations are carried out somewhat differently, depending on the values of $\alpha$, as explained in Remark \ref{rem:4.2}.

Let us first consider the case $\alpha\in (0,1)$.
The restriction of $\widehat g(z)$ to the real line, needed in 
\eqref{cond1_alpha_small}, is found by taking the limit $z\to t\in \Real_+$ in \eqref{hatg_equal_alphas}, 
either in the upper or lower half planes. By subtracting from $\widehat g(t)$ the first equation in \eqref{more_eq}, 
plugging the result into \eqref{cond1_alpha_small} and taking $T\to\infty$ we obtain the asymptotics 
\begin{equation}\label{eq:5.5}
\begin{aligned}
&
\frac 1 {2\pi i}\int_0^\infty  \Big(N_{\alpha_2}^+(t)-N_{\alpha_2}^-(t)\Big)\widehat g(t)dt  \asymp\\
&
(k^S_0+k^D_0)\frac 1 {2\pi i}\int_0^\infty   
\frac{N_{\alpha}^+(t)-N_{\alpha}^-(t)}{t^2-t_0^2}\Big(
(t_0+\beta) \frac{X_0^+(t_0)}{N_{\alpha}^+(t_0)}- (t+\beta)\frac{X_0^+(t) }{  N_{\alpha}^+(t)}
\Big)
dt.
\end{aligned}
\end{equation}
The latter integral is well defined, since  singularity at $t_0$ is integrable, and it does not vanish for all $\alpha\in (0,1)$. 
Therefore \eqref{cond1_alpha_small} implies that $k^S_0 + k^D_0 \to 0$ as $T\to\infty$ and, due to \eqref{more_eq_asymp}, we also 
have $\psi(0)\to 0$ and 
\begin{equation}\label{cScDlim}
k^S_0 -k^D_0 \xrightarrow[T\to\infty]{}  \frac {\mu_\eps}{t_0-\beta} \frac{N_{\alpha}^-(t_0)} { X^-(t_0)}.
\end{equation}

The expression \eqref{KlLBfla} can now be derived by using \eqref{PT_alpha_small}:
\begin{equation}\label{oneway}
\begin{aligned}
& 
P_\infty(\beta,\mu_\eps) \stackrel{(\mathrm{a})}{\asymp}  \\
& 
\frac {k^S_0-k^D_0} {\mu_\eps} \frac 1 {2\pi i}  \int_0^\infty   
\frac{N_{\alpha}^+(t)   - N_{\alpha}^-(t)}{t^2-t_0^2} \Big( (t-\beta)\frac{ X^-(t) }{  N_{\alpha}^-(t)}-(t_0-\beta) \frac{X^-(t_0)}{N_{\alpha}^-(t_0)}\Big) dt \stackrel{(\mathrm{b})}{\asymp}\\
&
\frac 1 {2\pi i}  \int_0^\infty   
\frac{N_{\alpha}^+(t)   - N_{\alpha}^-(t)}{t^2-t_0^2} \Big( \frac{t-\beta}{t_0-\beta}\frac{N_{\alpha}^-(t_0)} { X^-(t_0)} \frac{ X^-(t) }{  N_{\alpha}^-(t)}-1\Big) dt \stackrel{(\mathrm{c})}{=}\\
& 
\kappa_\alpha \frac {\cos \frac \alpha 2 \pi} { \pi }  \int_0^\infty   
\frac{t^{\alpha-1}}{t^2-t_0^2} \Big( \frac{t-\beta}{t_0-\beta}(t/t_0)^{\frac {1-\alpha} 2}-1\Big) dt \stackrel{(\mathrm{d})}{=} \\
& 
 \frac{\Gamma(3-\alpha)} 2t_0^{\alpha-2}
\Big(1+\sin (\tfrac \alpha 2 \pi)\frac{t_0+\beta}{t_0-\beta}\Big),
\end{aligned}
\end{equation}
where (a) is obtained by plugging $e^{-Tt}\widehat g(-t)$ from \eqref{hatg_equal_alphas} and subtracting 
the second equation from \eqref{more_eq}, the limit (b) holds due to \eqref{cScDlim}, equality (c) follows by
by substitution of the explicit formulas from \eqref{Nalpha} and \eqref{Xzalpha} and (d) is computed by the 
standard contour integration and simplified using elementary trigonometry.

\subsection{Large time limit $\boldsymbol{\alpha \in (1,2)}$}
In view of \eqref{hatg_equal_alphas} and \eqref{PhicScD}, the first term in the brackets in \eqref{cond1} satisfies
\begin{equation}\label{eq:5.7}
\begin{aligned}
z  N_{\alpha}(z) \widehat g(z) = & -   \Phi_0(z)  + O(z^{\alpha-2}) = 
-X(z) \big(k^S_0  +k^D_0 \big) + O(z^{\alpha-2})= \\
&
-(-z)^{\frac {\alpha-1} 2}\big(k^S_0  +k^D_0 \big) + O(z^{\alpha-2}), \quad \text{as\ \ }  \Re(z)\to\infty.
\end{aligned}
\end{equation}
Similarly to \eqref{eq:5.5}, the second term satisfies, as $T\to\infty$, 
\begin{align*}
&
\frac 1 {2\pi i} \int_0^\infty\frac 1{t-z} \big(N_{\alpha}^+(t)-N_{\alpha}^-(t)\big)
\widehat g(t)dt \asymp\\
&
\frac {k^S_0+k^D_0} {2\pi i} \int_0^\infty\frac {N_{\alpha}^+(t)-N_{\alpha}^-(t)}{t-z} 
   \frac{1 }{t^2-t_0^2} \Big( (t_0+\beta)  \frac{X^+(t_0)}{N_{\alpha}^+(t_0)}  
-  (t+\beta)
\frac{X^+(t) }{  N_{\alpha}^+(t)}\Big) dt.
\end{align*}
The latter integral can be written as the sum of three parts,
\begin{align*}
J_1(z):= &\phantom{+\ }
\frac 1 {2\pi i} \int_0^\infty\frac {N_{\alpha}^+(t)-N_{\alpha}^-(t)}{t-z} 
\frac{t_0+\beta }{t^2-t_0^2} 
\Big(   \frac{X^+(t_0)}{N_{\alpha}^+(t_0)} - \frac{X^+(t) }{  N_{\alpha}^+(t)} \Big) dt,
\\
J_2(z):= &\phantom{+\ }
\frac 1 {2\pi i} \int_0^\infty\frac {N_{\alpha}^+(t)-N_{\alpha}^-(t)}{t-z}\frac{t_0 }{t(t +t_0)}  \frac{X^+(t) }{  N_{\alpha}^+(t)}dt, \
\\
J_3(z):= &-
\frac 1 {2\pi i} \int_0^\infty\frac {N_{\alpha}^+(t)-N_{\alpha}^-(t)}{t-z} \frac 1 t \frac{X^+(t) }{  N_{\alpha}^+(t)}dt,
\end{align*}
where both $zJ_1(z)$ and $zJ_2(z)$ converge to finite limits as $z\to\infty$ and  
\begin{align*}
-z J_3(z) = 
& z\frac 1 {2\pi i} \int_0^\infty\frac {N_{\alpha}^+(t)-N_{\alpha}^-(t)}{t-z} \frac 1 t \frac{X^+(t) }{  N_{\alpha}^+(t)}dt =\\
&
 -z\frac {\sin (\tfrac{\alpha-1}{2}\pi)} { \pi  } \int_0^\infty\frac {t^{\frac {\alpha-1}2-1} }{t-z}  dt =  (-z)^{\frac {\alpha-1}{2} }.
\end{align*}
This term cancels out with \eqref{eq:5.7} in the limit \eqref{cond1}, which  therefore, takes the form 
\[
\big(k^S_0  +k^D_0 \big)\lim_{\Re(z)\to\infty}  
\Big(zJ_1(z)+zJ_2(z) \Big)
=0.
\]
A calculation shows that the limit here remains non-zero for all $\alpha\in (1,2)$. Consequently, $k^S_0+k^D_0\to 0$ as $T\to\infty$ and
 \eqref{cScDlim} remains true. 
Similarly, the first term in \eqref{PTfla} has the asymptotics
\[
z  N_{\alpha}(-z) e^{-zT}\widehat g(-z) =  (-z)^{\frac {\alpha-1} 2} (k^S_0-k^D_0)
+ O(z^{\alpha-2}), \quad z\to\infty,
\]
which compensates the leading order term in the integral. Hence we obtain 
\begin{align*}
& P_T (\beta, \mu_\eps)  \asymp \\
&
(k^S_0 -  k^D_0)\frac {t_0-\beta} {\mu_\eps}     
\frac {1 } {2\pi i}  \int_0^\infty \Big(N_{\alpha}^+(t)   - N_{\alpha}^-(t)\Big)  
\frac{ 1}{t^2-t_0^2}\Big(\frac{   X^+(t)}{  N_{\alpha}^+(t)}-   \frac{ X^+(t_0)}{N_{\alpha}^+(t_0)}\Big) dt \\
&
 -(k^S_0 -  k^D_0)\frac {t_0} {\mu_\eps} \frac {1 } {2\pi i}  \int_0^\infty \Big(N_{\alpha}^+(t)   - N_{\alpha}^-(t) \Big)
 \frac{ 1}{t(t+t_0)} \frac{   X^+(t)}{  N_{\alpha}^+(t)}  dt, \quad  T\to\infty.
\end{align*}
Upon substitution of \eqref{Nalpha} and \eqref{Xzalpha} these integrals can be evaluated explicitly by 
means of standard contour integration. Then plugging the limit \eqref{cScDlim} and simplifying the obtained 
trigonometric formulas, we arrive at the very same expression, derived in \eqref{oneway}.

\section{Proof of Theorem \ref{thm:3}}\label{sec:case3}

\subsection{The equivalent problem} 
The equivalent problem in Subsection \ref{sec:equiv} was derived for $\alpha_2 \in (0,1)$, and in the 
complementary case $\alpha_2 \in (1,2)$ it takes a somewhat different form.  
The function $\theta(t)=\arg\big(\Lambda^+(t)\big)$ is now negative with 
$
\theta(0+) = -\pi
$
and
$
\theta(\infty)= \frac {1-\alpha_2}{2}\pi,
$
and, consequently, in view of estimates \eqref{apriori_est}  and \eqref{Xzest}, the appropriate choice of the 
factor in \eqref{Xz} is $k=-1$. The functions $S(z)$ and $D(z)$, defined in \eqref{SzDz}, grow at most linearly 
as $z\to\infty$ and, therefore,
\begin{equation}\label{SSDD}
\begin{aligned}
S(z) = & - k^S_1 p_1(-z) + k^S_0 p_0(-z) + p(-z), \\
D(z) = & - k^D_1 q_1(-z) + k^D_0 q_0(-z) + q(-z),
\end{aligned}
\end{equation}
where $p_j(z), q_j(z)$ and $p(z), q(z)$ are solutions to  auxiliary integral equations 
\eqref{qqpp} and \eqref{pqeq}, respectively. Combining \eqref{SzDz} with \eqref{SSDD} yields the 
expressions for $\Phi_0(z)$ and $\Phi_1(z)$ and, in turn, for the Laplace transform $\widehat g(z)$ in \eqref{gTz_expr},
specified up to unknown constants $k^S_j$, $k^D_j$ and $\psi(0)$.
These constants are found using \eqref{cond1} and the conditions, implied by removal of the poles, 
\begin{equation}\label{cond_poles}
\begin{aligned}
&  (z_0+\beta) \Big(\Phi_0(z_0)+e^{-z_0T} \Phi_1(-z_0)\Big) +  \mu_\eps^2   
\Big(\psi(0)+\frac {1}{\mu_\eps}e^{-z_0T}\Big)=0,\\
&
  (z_0-\beta) \Big(\Phi_1(z_0)+e^{-z_0T}\Phi_0(-z_0)\Big)  -  \mu_\eps^2  
\Big(\frac {1}{\mu_\eps}+ e^{- z_0T}\psi(0)\Big) =0.
\end{aligned}
\end{equation}
Finally, the limit filtering error can be computed using \eqref{PTfla}.

\subsection{Large time limit $\boldsymbol{\alpha \in (0,1)}$}
Let $\alpha_1=1$ and $\alpha_2=\alpha\in (0,1)$.
Our starting point is the expression for the limiting error \eqref{Pinfty_expr}. Using the special form of the structural function
in this case,
\begin{equation}\label{spLambda}
\Lambda(z) = (z^2-\beta^2) N_\alpha(z)-\mu_\eps^2,
\end{equation}
and property \eqref{homo},  the integral in \eqref{Qz} simplifies to 
\begin{align*}
Q(z) = 
&
   \frac 1 { 2\pi i} \frac 1 {\mu_\eps}\int_0^\infty \frac {t +\beta }{t-z}
\Big(\frac{N_{\alpha}^+(t)}{X^+(t)}-\frac {N_{\alpha}^-(t)}{X^-(t)}\Big) dt = \\
&
\mu_\eps  \frac {1} { 2\pi i}\int_0^\infty \frac {1}{t-z}\frac {1 }{t -\beta}
\Big(\frac{ 1}{X^+(t)}-\frac { 1}{X^-(t)}\Big) dt = 
 \frac {\mu_\eps} {z-\beta} \Big(H(z)-H(\beta)\Big),
\end{align*}
where we defined 
\begin{equation}\label{Hz}
H(z):= \frac 1 { 2\pi i} \int_0^\infty \ \frac{1} {t-z}  
\Big(\frac{ 1}{X^+(t)}-\frac { 1}{X^-(t)}\Big) dt,
\end{equation}
and $H(\beta)$ stands for the common value of $H^+(\beta)=H^-(\beta)$. These two limits coincide, since
$\Lambda(z)$ in \eqref{spLambda} satisfies $\Lambda^+(\beta)=\Lambda^-(\beta)$ and hence, in view of \eqref{homo},  
\[
X^+(\beta)=X^-(\beta)=: X(\beta)\in \Real, \quad \beta >0. 
\]

By the Sokhotski-Plemelj theorem 
\[
H^+(t)-H^-(t)=\frac 1 {X^+(t)}-\frac 1 {X^-(t)}, \quad t\in \Real_+,
\]
and hence  $H(z)-1/X(z)$ is an entire function.  
Since it vanishes as $z\to \infty$, it must coincide with the zero function and hence
\begin{equation}\label{Qexpr}
Q(z)=\frac {\mu_\eps} {z-\beta} \Big(\frac 1 {X(z)}-\frac 1{X(\beta)}\Big).
\end{equation}
Plugging this formula into \eqref{Pinfty_expr} yields 
\begin{equation}\label{plugme}
P_\infty(\beta, \mu_\eps) =
 -\frac 1{X(\beta)} \frac 1 {2\pi i}  \int_0^\infty   \Big(N_{\alpha}^+(t)   - N_{\alpha}^-(t) \Big)  
    \frac{   X^+(t) }{\Lambda^+(t)}     
dt.    
\end{equation}
Further simplification is possible due to the following lemmas.

\begin{lem}\label{lem:Xz}
\begin{equation}\label{ident1}
X(z)X(-z) = -\frac1 {\kappa_\alpha}\Lambda(z), \quad z\in \mathbb{C}\setminus \Real.
\end{equation}
\end{lem}

\begin{proof}
By definition \eqref{Xz} with $k=1$,
\begin{equation}\label{Upsi}
X(z)X(-z)=
(-z)^{1-\theta(\infty)/\pi}z^{1-\theta(\infty)/\pi}\exp \big(\Upsilon(z)\big),
\end{equation}
where  
\[
\Upsilon(z) := \frac 1 \pi \int_0^\infty \frac {\theta(t)-\theta(\infty)}{t-z}dt+
 \frac 1 \pi \int_0^\infty \frac {\theta(t)-\theta(\infty)}{t+z}dt.
\]
Define the function, cf. \eqref{spLambda},
\[
\widetilde \Lambda(z) :=\frac{\Lambda(z)}{z^2N_\alpha(z)} = 
1-\beta^2 z^{-2}  -\frac{\mu_\eps^2}{z^2N_\alpha(z)}.
\]
Since $\arg\big(N_\alpha^+(t)\big) = \theta(\infty)$ and in view of symmetries \eqref{Nsym} and \eqref{Lambdasym},   
\[
\theta(t)-\theta(\infty) =\arg\big(\widetilde \Lambda^+(t)\big) = \frac 1 {2i} \log \frac{\widetilde \Lambda^+(t)}{\widetilde \Lambda^-(t)} =: \widetilde \theta(t).
\]
The angle function $\widetilde \theta(t)$ is odd, $\widetilde \theta(t)=-\widetilde \theta(-t)$, and hence we can write 
\begin{align*}
\Upsilon(z) =\, & \frac 1 \pi \int_0^\infty \frac {\widetilde \theta(t) }{t-z}dt+
 \frac 1 \pi \int_0^\infty \frac {\widetilde \theta(t)}{t+z}dt = 
 \frac 1 \pi \int_{-\infty}^\infty \frac {\widetilde \theta(t) }{t-z}dt =\\
& 
\frac 1 {2\pi i} \int_{-\infty}^\infty \frac {\log \widetilde \Lambda^+(t)}{t-z}dt 
-
\frac 1 {2\pi i} \int_{-\infty}^\infty \frac {\log \widetilde \Lambda^-(t)}{t-z}dt.
\end{align*}
The latter integrals are well defined since $\log \widetilde \Lambda^\pm(t)=O(|t|^{-1-\alpha})$ as $|t|\to\infty$.
An elementary calculation shows that $\widetilde \Lambda(z)$ does not cross the branch cut of the logarithm for any $z\in \mathbb C\setminus \Real$. Hence these integrals can be computed by integrating the function 
$f(\zeta)=\log \widetilde \Lambda(\zeta)/(\zeta-z)$ over the circular arcs in the lower and upper half planes.
Standard residue calculus then shows that  
$
\Upsilon(z) = \log \widetilde \Lambda (z),
$
and the claimed formula is obtained by plugging this into \eqref{Upsi}. 
\end{proof}

\noindent
The large time limit \eqref{showme_T} for $H>\frac 1 2$ follows from \eqref{plugme} and the following 
expression.

\begin{lem}
\[
\frac 1 {2\pi i}  \int_0^\infty   \Big(N_{\alpha}^+(t)   - N_{\alpha}^-(t) \Big)  
    \frac{   X^+(t) }{\Lambda^+(t)}     
dt = 
-\frac { \mu_\eps^2 } {\kappa_\alpha}\frac 1 {2\beta}    
\Big(\frac {1 } {X(\beta)}-\frac {1 } {X(-\beta)} \Big).
\]
\end{lem}
\begin{proof}
In view of \eqref{Nalpha} and identity \eqref{ident1}, 
this integral  equals $-\frac 1 2 \dfrac{1}{  \sin (\frac \alpha 2\pi  )}I$ with 
\[
I= \frac {\sin (\pi \alpha)} { \pi  }\int_0^\infty \frac{ t^{\alpha-1} }{X(-t)} dt. 
\]
Integrating the function $f(z) = z^{\alpha-1}/X(z)$ over semicircular contours in the upper and lower half-planes, 
applying Jordan's lemma and subtracting the results, we obtain an alternative expression 
\[
I = 
\frac 1 {2\pi i} \int_0^\infty t^{\alpha-1} \Big(\frac 1 {X^+(t)}-\frac 1 {X^-(t)}\Big)dt.
\]
This can also be viewed as the limit $I = -\lim_{z\to\infty} z F(z)$ for 
\[
F(z) := \frac 1 {2\pi i} \int_0^\infty \frac{t^{\alpha-1}}{t-z}\Big(\frac 1 {X^+(t)}-\frac 1 {X^-(t)}\Big)dt.
\]
Now define the sectionally holomorphic function
\[
G(z) := (z^2-\beta^2) \Big(F(z)+ \frac{(-z)^{\alpha-1}}{X(z)}\Big),  \quad z\in \mathbb{C}\setminus \Real_+.
\]
Its limits across the positive real semiaxis satisfy 
\begin{align*}
&
G^+(t)-G^-(t) = \,
  (t^2-\beta^2)\left(F^+(t)-F^-(t)  +  \frac{(e^{-\pi i }t)^{\alpha-1}}{X^+(t)}
-
\frac{(e^{\pi i}t)^{\alpha-1} }{X^-(t)}\right) =\\
&
 2\frac {\cos(  \tfrac{\alpha-1} 2\pi )} {\kappa_\alpha} \left(   \frac {(t^2-\beta^2)N_\alpha^+(t) } {X^+(t)} 
-\frac {(t^2-\beta^2)N_\alpha^-(t) } {X^-(t)}
\right) = \\
&
 2\frac {\sin(  \tfrac{ \alpha} 2\pi )} {\kappa_\alpha} \left(   \frac {\Lambda^+(t)+\mu_\eps^2 } {X^+(t)} 
-\frac {\Lambda^-(t)+\mu_\eps^2 } {X^-(t)}
\right)=  
2 \frac { \sin(  \tfrac{ \alpha} 2\pi )\mu_\eps^2} {\kappa_\alpha} \left(   \frac {1 } {X^+(t)} 
-\frac {1 } {X^-(t)}
\right),
\end{align*}
where the last equality holds by virtue of \eqref{homo}. 
Since $G(z)= -I z \big(1+o(1)\big)$ and $1/X(z)\to 0$ as $z\to\infty$, by the 
Sokhotski-Plemelj theorem 
\[
G(z)  = 2\mu_\eps^2 \frac {\sin(  \tfrac{ \alpha} 2\pi )} {\kappa_\alpha}   \frac {1 } {X(z)} - I z + C,
\]
with a constant $C$. By definition of $G(z)$ it must vanish at $\pm \beta$ and hence 
\[
I  =  \frac {\mu_\eps^2} {\beta}  \frac {\sin(  \tfrac{ \alpha} 2\pi )} {\kappa_\alpha}   
\Big( \frac {1 } {X(\beta)}-\frac {1 } {X(-\beta)} \Big).
\]
\end{proof}

\subsection{Small noise asymptotics $\boldsymbol{\alpha\in (0,1)}$}

The formula in \eqref{showme_eps} is obtained using \eqref{small-noise}, 
continuity of \eqref{showme_T} with respect to $\beta$ and the following limit.
\begin{lem}\label{lem:limbeta}
For $\eps=1$, 
\[
\lim_{\beta\to 0+}\frac 1 \beta \log \frac{X(-\beta)} {X(\beta)} =
   \frac{2 }{\sin \frac \pi {1+\alpha}} \left(\frac{\kappa_\alpha}{\mu^2}\right)^{\frac 1 {1+\alpha}}.
\]
\end{lem}

\begin{proof}
Let $\widetilde \theta(t):= \theta(t)-\theta(\infty)$ and note that $\widetilde \theta(t)=\arg\big(\widetilde \Lambda^+(t)\big)$ 
where, c.f. \eqref{Lambdaz},  
\[
\widetilde \Lambda(z) = \frac{\Lambda(z)}{N_\alpha(z)}=z^2-\beta^2 -\frac{\mu^2}{N_\alpha(z)}.
\]
By definition \eqref{Xz}, for $\beta>0$, 
\begin{align*}
\frac{X(\beta)} {X(-\beta)} = 
&
- \exp \left(-\frac 1 \pi \int_0^\infty \frac {\widetilde \theta(t) }{t+\beta}dt
 +\frac 1 \pi \dashint_0^\infty \frac {\widetilde\theta(t)}{t-\beta}dt+  i  \theta(\beta) 
 \right) =\\
&
 \phantom{+\ } \exp \left(\frac {2\beta} \pi \dashint_0^\infty \frac {\widetilde \theta(t) }{t^2-\beta^2}dt 
 \right) =: \exp(J),
\end{align*}
where the integral is in the sense of the Cauchy principal value and the second equality holds since $\theta(\beta)= \pi$.
Due to symmetry \eqref{Lambdasym},
\[
\widetilde \theta(t)= \frac 1 {2i} \log\frac{\widetilde \Lambda^+(t)}{\widetilde \Lambda^-(t)},
\]
and since  
$
\displaystyle \dashint_0^\infty \frac 1 {t^2-\beta^2}dt =0,
$
we can write
\[
J = \frac { \beta} {\pi i} \int_0^\infty \frac { \log \widetilde \Lambda^+(t)- \log \widetilde \Lambda^+(\beta)
   }{t^2-\beta^2}dt 
-
\frac { \beta} {\pi i} \int_0^\infty \frac {  \log \widetilde \Lambda^-(t)-\log \widetilde \Lambda^-(\beta)  }{t^2-\beta^2}dt.
\]
Integrating the function  
\[
f(z) = \frac{\log \widetilde \Lambda(z)-\log \widetilde \Lambda^+(\beta)}{z^2-\beta^2}
\]
over the closed contour in the first quadrant, formed by the axes and a circular arc, and applying 
Jordan's lemma,  we find that
\[
\frac 1{2\pi i} \int_0^\infty \frac{\log \widetilde \Lambda^+(t)-\log \widetilde \Lambda^+(\beta)}{t^2-\beta^2}dt =
-\frac 1 {2\pi} \int_0^\infty \frac{\log \widetilde \Lambda^r(it)-\log \widetilde \Lambda^+(\beta)}{t^2+\beta^2}dt,
\] 
where $\widetilde\Lambda^r(it)$ stands for the limit of $\widetilde \Lambda(z)$ as $z\to it$ in the right half-plane.
Integrating the function 
\[
h(z) = \frac{\log \widetilde \Lambda(z)-\log \widetilde \Lambda^-(\beta)}{z^2-\beta^2}
\]
over similar contour in the fourth quadrant, we get 
\[
\frac 1 {2\pi i} \int_0^\infty \frac{\log \widetilde \Lambda^-(t)-\log \widetilde\Lambda^-(\beta)}{t^2-\beta^2}dt =  
\frac 1 {2\pi} \int_0^\infty \frac{\log \widetilde \Lambda^r(-it)-\log \widetilde \Lambda^-(\beta)}{t^2+\beta^2}dt.
\]
Subtracting, we obtain 
\begin{align*}
J &= 
-\frac {\beta}\pi \int_0^\infty \frac{\log \widetilde \Lambda^r(it)-\log \widetilde \Lambda^+(\beta)}{t^2+\beta^2}dt
-\frac {\beta}\pi \int_0^\infty \frac{\log \widetilde \Lambda^r(-it)-\log \widetilde \Lambda^-(\beta)}{t^2+\beta^2}dt \\
&= 
\frac 1 2 \Big(\log \widetilde \Lambda^+(\beta)+\log \widetilde \Lambda^-(\beta)\Big) 
-\frac {2\beta}\pi \int_0^\infty \frac{\Re\big(\log \widetilde \Lambda^r(it)\big)}{t^2+\beta^2}dt.
\end{align*}
A standard calculation, which uses the explicit expressions 
\[
\Re\big(\log \widetilde \Lambda^r(it)\big)=\log \Big(t^2+\beta^2+ \frac{\mu^2}{\kappa_\alpha}t^{1-\alpha}\Big)
\]
and 
\[
\log \widetilde \Lambda^+(\beta)+\log \widetilde \Lambda^-(\beta) = 2\log \left(\frac{\mu^2}{\kappa_\alpha}\beta^{1-\alpha}\right),
\]
yields the claimed asymptotics 
\[
J = -\frac{2\beta}{\sin \frac \pi {1+\alpha}} \left(\frac{\kappa_\alpha}{\mu^2}\right)^{\frac 1 {1+\alpha}} \big(1+o(1)\big), \quad \beta \to 0. 
\]
\end{proof}

\subsection{Large time limit $\boldsymbol{\alpha \in (1,2)}$}
We can use  \eqref{SzDz} and \eqref{SSDD} to express $\Phi_0(z)$ and $\Phi_1(z)$ in terms of solutions to 
\eqref{pqeq} and \eqref{qqpp}. If we plug the obtained expressions into equations \eqref{cond_poles} and
apply the estimates  \eqref{asym_est} and \eqref{pjqjest}, we arrive at the large time asymptotic relations, $T\to\infty$, 
\begin{equation}\label{asas}
\begin{aligned}
&  
(z_0+\beta)  X(z_0) \Big(  (k^S_1    + k^D_1)  z_0 + k^S_0  + k^D_0  + F_S(z_0)
    + F_D(z_0)
\Big) +  \mu_\eps^2 \psi(0) \asymp 0,\\
&
(z_0-\beta)  X(z_0) 
\Big(
  (k^S_1    - k^D_1)  z_0 + k^S_0 - k^D_0  + F_S(z_0)     - F_D(z_0)
\Big)  -  \mu_\eps  \asymp 0,
\end{aligned}
\end{equation}

In the case $\alpha\in (1,2)$ the function $H(z)$ in \eqref{Hz} can no longer be defined, but nevertheless
the formula in \eqref{Qexpr} remains valid, as can be checked directly using the Sokhotski-Plemelj theorem. 
Hence the second equation in \eqref{asas} is equivalent to    
\[
  (k^S_1    - k^D_1)  z_0 + k^S_0 - k^D_0   \asymp \frac {\mu_\eps} {z_0-\beta}\frac 1{X(\beta)}.
\]
Since $X(\beta)$ is purely real, this implies 
\begin{equation}\label{difsum}
\begin{aligned}
k^S_1 -k^D_1 & \asymp -\frac{\mu_\eps}{X(\beta)} \frac 1 {|z_0-\beta|^2}, \\
k^S_0 - k^D_0 & \asymp -\frac {\mu_\eps}{X(\beta)} \frac{\beta -z_0-\overline{z}_0}{|z_0-\beta|^2}.
\end{aligned}
\end{equation}

The filtering error can now be computed using \eqref{PTfla}. To this end, note that, in view of  \eqref{gTz_expr} and \eqref{SSDD},
the first term satisfies 
\begin{equation}\label{1st}
N_\alpha(-z)e^{ -zT} \widehat g(-z)  =   
\big(k^S_1-k^D_1\big)(-z)^{\frac {\alpha-1} 2-1}  + O\big(z^{\frac {\alpha-1} 2-2}\big), \quad  \Re(z)\to\infty,
\end{equation}
where  
\[
(-z)^{\frac {\alpha-1} 2-1} = -  \frac {\cos (\tfrac \alpha 2 \pi)}{ \pi  }\int_0^\infty     \frac{t^{ \frac{\alpha-1} 2-1}}{t-z}dt.
\]
Due to \eqref{gTz_expr} and \eqref{Nalpha}, the integral in \eqref{PTfla} takes the form
\begin{equation}\label{2nd}
\begin{aligned}
&
\frac 1 {2\pi i}  \int_0^\infty \frac 1 {t-z} \big(N_{\alpha}^+(t)   - N_{\alpha}^-(t) \big) e^{-tT}\widehat g(-t) dt =\\
&
\kappa_\alpha\frac {\cos (\frac \alpha 2 \pi)} { \pi }  \int_0^\infty \frac {t^{\alpha-1}} {t-z}  
\bigg(
(t-\beta) \frac { \Phi_1^+(t)}{\Lambda^+(t)} - \frac{\mu_\eps}{\Lambda^+(t)}   
\bigg) dt
+ z^{-1} R(z,T), 
\end{aligned}
\end{equation}
where the residual $R(z,T)$ vanishes as $T\to\infty$, uniformly over $z$. 
The value of the latter integral will not change asymptotically as $T\to\infty$, if we replace 
\[
\Phi_1^+(t) \asymp (k^S_1-k^D_1)t + k^S_0-k^D_0 + Q^+(t).
\]
Thus, substituting approximations \eqref{1st} and \eqref{2nd} into \eqref{PTfla} and using formula \eqref{Qexpr}, we arrive at  
\begin{equation}\label{PTbetamu}
P_T(\beta, \mu_\eps) \asymp \frac {\kappa_\alpha} {\mu_\eps}     \frac {\cos (\frac \alpha 2 \pi)}{ \pi  }\Big(
 \big(k^S_1-k^D_1\big)  I_2   + \big(k^S_0  - k^D_0\big)  I_1  +  I_0  \Big),
\end{equation}
with  
\begin{equation}\label{I123}
\begin{aligned}
I_0 := & -  \frac {\mu_\eps}{X(\beta)}    
\int_0^\infty  t^{\alpha-1}  \frac{X^+(t)}{\Lambda^+(t)}  dt,
\\
I_1 := &     \int_0^\infty   t^{\alpha-1}(t-\beta) \frac{X^+(t)}{\Lambda^+(t)} dt, \\
I_2 := & 
\int_0^\infty   \Big(t^ \alpha (t-\beta) \frac{X^+(t)}{\Lambda^+(t)}+\frac 1{\kappa_\alpha}  t^{ \frac{\alpha-1} 2-1}\Big)dt.
\end{aligned}
\end{equation}
To simplify the expression in \eqref{PTbetamu} we will need the following identity. 

\begin{lem}\label{lem:XzXz}
\begin{equation}\label{XzXz}
X(z)X(-z) = -   \frac { 1 } { \kappa_\alpha } \frac {\Lambda(z)} {(z^2-z_0^2) (z^2-\overline{z}_0^2)},
\quad z\in \mathbb C\setminus \Real.
\end{equation}
\end{lem}

\begin{proof}\
For $X(z)$ defined in \eqref{Xz} with $k=-1$,
\[
X(z)X(-z) =
(-z)^{-1+\frac{\alpha-1}{2}}z^{-1+\frac{\alpha-1}{2}} \exp \big(\Upsilon(z)\big),
\]
where 
\[
\Upsilon(z):=\frac 1 \pi \int_0^\infty \frac {\theta(t)-\theta(\infty)}{t-z } dt+
\frac 1 \pi \int_0^\infty \frac {\theta(t)-\theta(\infty)}{t+z } dt.
\]
Define the function 
\[
\widetilde{\Lambda}(z) := \frac {\Lambda(z)} {z^2 N_\alpha(z)} = 
 1-\beta^2z^{-2}  -\frac {\mu_\eps^2} {z^2N_\alpha(z)}.
\]
Since $\arg\big(N_\alpha^+(t)\big) = \theta(\infty)$ for all $t\in \Real_+$,  
\[
\theta(t)-\theta(\infty) = \arg\big(\widetilde{\Lambda}^+(t)\big) = \frac 1 {2i}\log \frac{\widetilde{\Lambda}^+(t)
}{\widetilde{\Lambda}^-(t)} =: \widetilde{\theta}(t),
\]
and, therefore,
\begin{equation}\label{kozel}
\begin{aligned}
\Upsilon(z) =\, &
\frac 1 \pi \int_0^\infty \frac {\widetilde\theta(t) }{t-z } dt+
\frac 1 \pi \int_0^\infty \frac {\widetilde\theta(t) }{t+z } dt \stackrel{\dagger}{=} 
\frac 1 \pi \int_{-\infty}^\infty \frac {\widetilde\theta(t) }{t-z } dt =\\
&
\frac 1 {2\pi i} \int_{-\infty}^\infty \frac { \log \widetilde{\Lambda}^+(t) }{t-z } dt
-\frac 1 {2\pi i} \int_{-\infty}^\infty \frac { \log \widetilde{\Lambda}^-(t) }{t-z } dt,
\end{aligned}
\end{equation}
where the equality $\dagger$ holds by the antisymmetry $\widetilde{\theta}(t) = - \widetilde{\theta}(-t)$. 
The last two integrals in \eqref{kozel} are well defined since $\lim_{t\to\pm \infty}\widetilde{\Lambda}^\pm(t)=1$. They 
can be evaluated by contour integration of the function 
\begin{equation}\label{fzeta}
f(\zeta):= \frac{\log \widetilde \Lambda (\zeta)}{\zeta-z},
\end{equation}
which must take into account the branch cut of the logarithm,
\begin{equation}\label{Cbr}
C = \big\{\zeta\in \mathbb{C}:  \widetilde{\Lambda}(\zeta) \in \Real_- \big\} = C_1\cup C_2 \cup C_3 \cup C_4,
\end{equation}
where $C_j$ denotes the intersection of $C$ with the $j$-th quadrant.

Let us determine the geometric shapes of each curve $C_j$'s starting with $C_1$. 
For $z = \rho e^{i\phi}$ in the first quadrant, 
with $\rho \in \Real_+$ and $\phi\in (0, \frac \pi 2)$,  
\[
\widetilde{\Lambda}(z)  =
1+\beta^2 \rho^{-2} e^{2 \widetilde \phi i}   
+\frac {\mu_\eps^2} { \kappa_\alpha }  \rho^{-1-\alpha}  e^{  (\alpha+1)\widetilde \phi i},
\]
where $\widetilde \phi := \frac \pi 2 -\phi\in (0 , \frac \pi 2)$. Hence 
$\Im\big(\widetilde{\Lambda}(z)\big)=0$ holds if and only if either $\widetilde\varphi=0$ 
or
\begin{equation}\label{rhotildephi}
\rho^{\alpha-1}  
=-\frac 1{\beta^2 }\frac {\mu_\eps^2} { \kappa_\alpha }   \frac{\sin ( (\alpha+1)\widetilde \phi )}{\sin(2 \widetilde \phi )  }.
\end{equation}
This equation has a solution only if the right hand side is positive, that is, when $\widetilde \phi\in (\frac \pi {\alpha+1}, \frac \pi 2)$. For all such  $\widetilde \phi$ and with $\rho$ as in \eqref{rhotildephi},
\begin{align*}
\Re\big(\widetilde{\Lambda}(z)\big) =\, &
1+\beta^2 \rho^{-2} \cos(2 \widetilde \phi)   
+\frac {\mu_\eps^2} { \kappa_\alpha }  \rho^{-1-\alpha} \cos(  (\alpha+1)\widetilde \phi ) =\\
&
1 +\rho^{-2}\beta^2   
 \frac{
  \sin ( (\alpha-1)\widetilde \phi ) 
  }{\sin ( (\alpha+1)\widetilde \phi )} = 
1 +    \left(  \beta^2   \right)^{ \frac   {\alpha+1}{\alpha-1}} \left(    \frac { \kappa_\alpha } {\mu_\eps^2}  \right)^{ \frac  2 {\alpha-1}} g(\widetilde \phi),
\end{align*}
where we defined 
\[
g(\widetilde \phi) := \left(- \frac{\sin(2 \widetilde \phi )  } {\sin ( (\alpha+1)\widetilde \phi )}\right)^{ \frac 2 {\alpha-1}}
 \frac{
  \sin ( (\alpha-1)\widetilde \phi ) 
  }{\sin ( (\alpha+1)\widetilde \phi )}.
\]

This function is strictly increasing on the interval $(\frac \pi {\alpha+1}, \frac \pi 2)$ and maps it onto $(-\infty, 0)$.
Hence $\Re\big(\widetilde{\Lambda}(z)\big)$ vanishes at the unique angle $\widetilde \phi_0 \in (\frac \pi {\alpha+1}, \frac \pi 2)$, and  $\Re\big(\widetilde{\Lambda}(z)\big) < 0$ if and only if $\widetilde \phi \in   (\frac \pi {\alpha+1}, \widetilde \phi_0)$. 
Therefore $C_1$ is the curve, which starts with $\widetilde \phi=\frac \pi {\alpha+1}$ at the origin and terminates at $z_0=\rho_0e^{i\phi_0}$, where $\phi_0=\frac \pi 2 -\widetilde \phi_0$ and the absolute 
value $\rho_0$ is determined by \eqref{rhotildephi}. The terminal point $z_0$ is precisely the zero of $\Lambda(z)$,
and hence also of $\widetilde \Lambda(z)$,  in the first quadrant.

The imaginary part $\Im(\widetilde \Lambda(z))$ vanishes on the positive imaginary semiaxis and on the continuation of 
$C_1$ corresponding to $\widetilde \phi\in [\widetilde \phi_0, \frac \pi 2)$, where $\Re(\widetilde \Lambda(z))$ remains 
positive. Hence $\Im(\widetilde \Lambda(z))$ preserves its sign on the subset of the first quadrant, which lies between these
curves, and it is readily checked to be positive. 
The rest of $C_j$'s have similar forms, starting at the origin and terminating at the other zeros of $\Lambda(z)$, as
shown on Figure  \ref{fig1}.  Along with the real and imaginary axes they divide the plane into eight subsets, 
on which the sign of $\Im(\widetilde \Lambda(z))$ is constant.

\begin{figure}[ht]
\begin{center}
\input{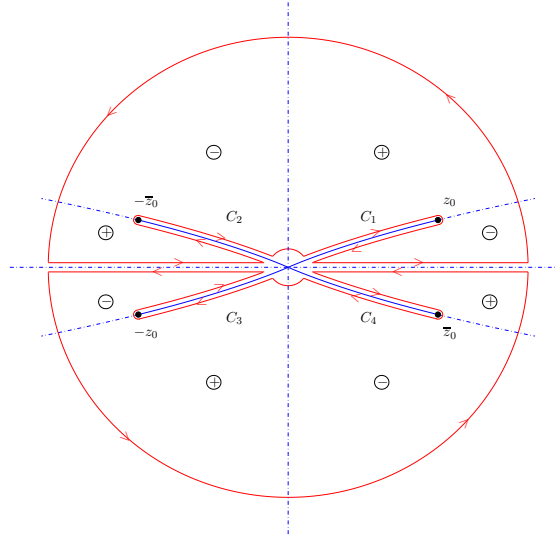}
\caption{\label{fig1}  
The branch cut $C$ in \eqref{Cbr} is depicted in solid blue; $\Im(\widetilde \Lambda(z))$ changes signs across the 
blue lines, both dashed and solid, being encircled over
the corresponding regions; the two integration contours are coloured in red.
}
\end{center}
\end{figure}

For definiteness, suppose $\Im(z)>0$. Then integrating $f(\zeta)$ from \eqref{fzeta} 
along the closed contour in the upper half plane, applying Jordan's lemma and Cauchy's residue theorem, we obtain
\[
\frac 1 {2\pi i }\int_{-\infty}^\infty \frac{\log \widetilde \Lambda^+(t)}{t-z}dt
= 
\log \widetilde \Lambda(z) 
- \frac 1 {2\pi i }\oint_{C_1} \frac{\log \widetilde\Lambda(\zeta)}{\zeta-z}d\zeta
- \frac 1 {2\pi i }\oint_{C_2} \frac{\log \widetilde\Lambda(\zeta)}{\zeta-z}d\zeta,
\]
where the  last two terms stand for the limiting values of the integrals over the shrinking contours around $C_1$ and $C_2$. 
Since $\big|\widetilde \Lambda(\zeta)\big|$ is continuous across $C_j$'s and taking into account the signs of 
$\widetilde \Lambda(\zeta)$, 
\[
\frac 1 {2\pi i }\oint_{C_1} \frac{\log \widetilde\Lambda(\zeta)}{\zeta-z}d\zeta = 
\frac {2\pi i}{2\pi i}\Big[ \log (\zeta-z)\Big]_0^{z_0} = \log \frac {z-z_0}{z}
\]
and 
\[
\frac 1 {2\pi i }\oint_{C_2} \frac{\log \widetilde\Lambda(\zeta)}{\zeta-z}d\zeta = 
 \frac{2\pi i}{2\pi i} \Big[\log (\zeta-z)\Big]_0^{-\overline{z}_0}=  \log \frac{z+\overline{z}_0}{z}.
\] 
Similarly, integration over the contour in the lower half plane gives 
\begin{align*}
\frac 1 {2\pi i }\int_{-\infty}^\infty \frac{\log \widetilde \Lambda^-(t)}{t-z}dt =\,  & 
 \frac 1 {2\pi i }\oint_{C_3} \frac{\log \widetilde\Lambda(\zeta)}{\zeta-z}d\zeta
+
 \frac 1 {2\pi i }\oint_{C_4} \frac{\log \widetilde\Lambda(\zeta)}{\zeta-z}d\zeta = \\
 &
  \log \frac{z+z_0}{z} + \log \frac{z-\overline{z}_0}{z}.
\end{align*}
Plugging this into \eqref{kozel} we obtain 
\[
\Upsilon(z) = 
\log \widetilde \Lambda(z) 
 \frac {z^4} {(z^2-z_0^2) (z^2-\overline{z}_0^2)}
\]
and, consequently, 
\begin{align*}
X(z)X(-z) = &
(-z)^{-1+\frac{\alpha-1}{2}}z^{-1+\frac{\alpha-1}{2}}   \frac {\Lambda(z)} {z^2 N_\alpha(z)} \frac {z^4} {(z^2-z_0^2) (z^2-\overline{z}_0^2)} = \\
-  & \frac { (-z)^{ \frac{\alpha-1}{2}}z^{\frac{\alpha-1}{2}}} {  N_\alpha(z)} \frac {\Lambda(z)} {(z^2-z_0^2) (z^2-\overline{z}_0^2)}=
-   \frac { 1 } { \kappa_\alpha } \frac {\Lambda(z)} {(z^2-z_0^2) (z^2-\overline{z}_0^2)}.
\end{align*}

\end{proof}

Using formula \eqref{XzXz}, the integrals in \eqref{I123} can now be written as 
\begin{equation}\label{III}
I_0   =      \frac {\mu_\eps}{X(\beta)}    \frac {J_0}{\kappa_\alpha},
\quad
I_1  = \frac {\beta  J_0-J_1}{\kappa_\alpha}, \quad 
I_2  =        \frac {\beta J_1-J_2}{\kappa_\alpha},
\end{equation}
where the basic elements are
\begin{equation}\label{Is}
\begin{aligned}
J_0 := & \int_0^\infty       \frac{1}{\big|t^2-z_0^2\big|^2} \frac {t^{\alpha-1}} {X(-t)} dt, \\
J_1 := & \int_0^\infty        \frac{t}{\big|t^2-z_0^2\big|^2} \frac { t^{\alpha-1}} {X(-t)} dt, \\
J_2 := & \int_0^\infty   \left( \frac{t^2 }{\big|t^2-z_0^2\big|^2} \frac {t^{\alpha-1}} {X(-t)} -  t^{ \frac{\alpha-1} 2-1}\right)dt.
\end{aligned}
\end{equation}

Closed form expressions for these integrals are derived in the following lemma.

\begin{lem}
The integrals in \eqref{Is} satisfy    
\begin{align}
\nonumber
J_0 = \frac 1 2\frac 1{z_0^2-\overline z_0^2} 
& 
\bigg( L \big(\tfrac{\alpha+1} 2\big)  -  \widetilde b_0   L\big(\tfrac{\alpha-1} 2\big) \\
&
\nonumber
+ z_0^{-1} M(z_0) -  z_0^{-1} M(-z_0) -  \overline z_0^{-1}   M(\overline z_0)+    \overline z_0^{-1} M(-\overline z_0) 
  \bigg),
  \\
\nonumber  
J_1  =    \frac 1 2 \frac 1{z_0^2-\overline z_0^2} 
& \bigg(   L\big(\tfrac{ \alpha+3} 2\big) -  \widetilde b_0   L \big(\tfrac{\alpha+1} 2\big) \\
&\label{LMz}
+M(z_0) + M(-z_0) -M(\overline z_0) -M(-\overline z_0) 
 \bigg),
 \\
\nonumber 
J_2  =    \frac 1 2 \frac 1{z_0^2-\overline z_0^2} 
&
\bigg(   L\big(\tfrac{\alpha+5} 2\big) -  \widetilde b_0    L\big(\tfrac{\alpha+3} 2\big) \\
\nonumber &
+ z_0 M(z_0) -  z_0  M(-z_0) - \overline z_0 M(\overline z_0) +    \overline z_0  M(-\overline z_0) 
\bigg),  
\end{align}
where 
\begin{equation}\label{Lz}
L(\gamma) = \frac \pi {\sin \pi \gamma }
\Big(z_0^{\gamma-1}+ (-z_0)^{\gamma-1}-\overline z_0^{\gamma-1}-(-\overline z_0)^{\gamma-1}\Big)
\end{equation}
and 
\begin{align}\label{Mfla}
M(z)  =\, & \frac 1 {\kappa_\alpha}\frac { \pi   } {\cos  (\frac \alpha 2 \pi)} \bigg(
\frac 1{X(z)} 
\frac {\Lambda(z)} {z^2-\beta^2}  
-    N_\alpha(z)  z ^{\frac{3-\alpha} 2} \big(1 -  \widetilde b_0 z^{-1}\big) 
\\
&\nonumber
+ \frac 1 2     \frac {\mu_\eps^2} {z^2-\beta^2}     
\Big(    \frac { 1}{X(\beta)} +\frac { 1}{X(-\beta)}\Big) -  \frac 1 2     \frac {\mu_\eps^2} {z^2-\beta^2}\frac z {\beta}\Big( \frac {1 }{X(\beta)}-    \frac { 1}{X(-\beta)}\Big)   \bigg).
\end{align}
\end{lem}

\begin{proof}
Define the function
\begin{equation}\label{Yz}
Y(z) := \frac 1{X(z)} - (-z)^{\frac{3-\alpha} 2} \big(1 +  \widetilde b_0 z^{-1}\big),
\end{equation}
where 
\[
\widetilde b_0 :=  \frac 1 \pi \int_0^\infty \big(\theta(t)-\theta(\infty)\big) d\tau <0.
\]
%
%
The integrals in  \eqref{Is} can be written as 
\begin{equation}\label{J012}
\begin{aligned}
J_0  = & \int_0^\infty       \frac{1}{\big|t^2-z_0^2\big|^2}  t^{\alpha-1} Y(-t)dt + 
U\big(\tfrac{\alpha+1} 2\big)  
-  \widetilde b_0 U\big(\tfrac{\alpha-1} 2\big),
  \\
J_1  = & \int_0^\infty        \frac{t}{\big|t^2-z_0^2\big|^2}   t^{\alpha-1} Y(-t) dt
+ U\big(\tfrac{ \alpha+3} 2\big) -  \widetilde b_0 U\big(\tfrac{\alpha+1} 2\big),
 \\
J_2  = & \int_0^\infty   
 \frac{t^2 }{\big|t^2-z_0^2\big|^2}  t^{\alpha-1} Y(-t)dt
+ V\big(\tfrac{\alpha-3} 2\big) -  \widetilde b_0 U\big(\tfrac{\alpha+3} 2\big),
\end{aligned}
\end{equation}
where 
\begin{align*}
U(\gamma):=& \int_0^\infty \frac{t^\gamma} {|t^2-z_0^2|^2}dt = \frac 1 2\frac 1{z_0^2-\overline z_0^2} L(\gamma), \\
V(\gamma):= & \int_0^\infty   t^{\gamma} \Big(     \frac{    t^4 }{\big|t^2-z_0^2\big|^2 }  -     1    \Big) dt = \frac 1 2\frac 1{z_0^2-\overline z_0^2} L(\gamma+4).
\end{align*}
Here $L(\gamma)$ is the function defined in \eqref{Lz} and the latter integrals are evaluated by standard contour integration. 
The formulas in  \eqref{LMz} are derived from \eqref{J012}, using the partial fraction decompositions 
\begin{align*}
\frac{ 1  }{|t^2-z_0^2|^2} = & 
\frac 1 2\frac 1{z_0^2-\overline z_0^2}\Big(\frac {z_0^{-1}} {t-z_0} - \frac {z_0^{-1}} {t+z_0} 
-
\frac {  \overline z_0^{-1}} {t-\overline z_0} + \frac {  \overline z_0^{-1}} {t+\overline z_0} \Big),
  \\
\frac{ t }{\big|t^2-z_0^2\big|^2 } = &
\frac 1 2 \frac 1{z_0^2-\overline z_0^2} \Big(\frac 1 {t-z_0} + \frac 1{t+z_0} -\frac 1 {t-\overline z_0} -\frac 1 {t+\overline z_0}\Big),
\\
\frac {t^2} {|t^2-z_0^2|^2}= & \frac 1 2 \frac 1{z_0^2-\overline z_0^2}\Big(  \frac {z_0} {t-z_0} - \frac {z_0}{t+z_0}
-   \frac {\overline z_0}{t-\overline z_0} +   \frac {\overline z_0} {t+\overline z_0} \Big),
\end{align*}
and  the notation 
\[
M(z):= \int_0^\infty \frac{t^{\alpha-1}}{t-z}Y(-t)dt.
\]

It is left to show that $M(z)$ satisfies the claimed formula. To this end, 
integrating the function $f(\zeta)=\frac{\zeta^{\alpha-1}}{\zeta-z}Y(\zeta)$ over semicircular contours in the upper and 
lower half planes 
and summing the obtained equations, we get
\begin{equation}\label{Mraz}
 \frac {\sin  (\alpha \pi)}{ \pi  } M(-z)  
 = 
 z^{\alpha-1}Y(z)-  P(z) 
\end{equation}
where   
\[
P(z) := \frac 1{2\pi i} \int_0^\infty \dfrac{t^{\alpha-1}}{t-z}\big( Y^+(t)-Y^-(t)\big)dt.
\]
To evaluate this integral, define  
\begin{equation}\label{Mdva}
H(z) = (z^2-\beta^2) \Big(P(z)+(-z)^{\alpha-1}Y(z)\Big).
\end{equation}
This function is sectionally holomorphic on $\mathbb{C}\setminus \Real_+$ and for $t>0$ 
\begin{align*}
 &
H^+(t)-H^-(t) =\,  \\
&
(t^2-\beta^2) \Big(P^+(t)-P^-(t)+(e^{-\pi i } t)^{\alpha-1}Y^+(t)-(e^{\pi i } t)^{\alpha-1}Y^-(t)\Big) =\\
&
2 \cos (\tfrac {\alpha-1} 2\pi) (t^2-\beta^2) t^{\alpha-1}\Big(   e^{-\frac {\alpha-1} 2\pi i }       Y^+(t)
-e^{\frac {\alpha-1} 2\pi i }   Y^-(t)\Big) =\\
&
2 \sin (\tfrac {\alpha } 2\pi) (t^2-\beta^2) t^{\alpha-1}
\Big(   
e^{-\frac {\alpha-1} 2\pi i }   \frac 1{X^+(t)} -e^{\frac {\alpha-1} 2\pi i } \frac 1{X^-(t)}
\Big) =\\
&
2 \sin (\tfrac {\alpha } 2\pi) \frac 1 {\kappa_\alpha}  
\Big(   
  \frac {(t^2-\beta^2)N_\alpha^+(t)}{X^+(t)} - \frac {(t^2-\beta^2)N_\alpha^+(t)}{X^-(t)}
\Big) = \\
&
2 \sin (\tfrac {\alpha } 2\pi) \frac {\mu_\eps^2} {\kappa_\alpha}  
\Big(   
  \frac {1}{X^+(t)} - \frac {1}{X^-(t)}
\Big).
\end{align*}
Since $H(z)$ grows not faster than linearly, it follows that
\begin{equation}\label{Mtri}
H(z) = 2 \sin (\tfrac {\alpha } 2\pi)  \frac {\mu^2_\eps}{\kappa_\alpha} \frac { 1}{X(z)} + c_1 z + c_0,
\end{equation}
where constants $c_1$ and $c_0$ are identified using the equations $H^+(\pm \beta)=0$,
\begin{align*}
c_0 = & -  \sin (\tfrac {\alpha } 2\pi)  \frac {\mu_\eps^2}{\kappa_\alpha} \Big( \frac { 1}{X(\beta)} +\frac { 1}{X(-\beta)}\Big), \\
c_1 = &  -\sin (\tfrac {\alpha } 2\pi)  \frac {\mu_\eps^2}{\kappa_\alpha} \frac 1 {\beta}\Big( \frac { 1}{X(\beta)}-    \frac { 1}{X(-\beta)}\Big).
\end{align*}
Plugging \eqref{Yz}, \eqref{Mdva} and \eqref{Mtri} into \eqref{Mraz} we obtain \eqref{Mfla} since 
\[
N_\alpha(z)= \frac {\kappa_\alpha} {2 \sin (\tfrac {\alpha } 2\pi)}\Big(z^{\alpha-1}  + (-z)^{\alpha-1}   \Big).
\]
\end{proof}
  
Now we are ready to find the ultimate expression for the filtering error in this case. Substituting 
\eqref{difsum}, \eqref{III} and \eqref{LMz} into \eqref{PTbetamu}, we get 
\begin{align*}
P_\infty(\beta, \mu_\eps) =\,   &   
\frac {1}{X(\beta)} \frac {\cos (\frac \alpha 2 \pi)}{ \pi  }\frac 1 {|z_0-\beta|^2}
\Big( J_2   -J_1 (z_0+\overline{z}_0)  +  J_0 z_0\overline z_0    \Big) =\\
&
\frac {1}{X(\beta)} \frac {\cos (\frac \alpha 2 \pi)}{ \pi  } \frac 1 {|z_0-\beta|^2}
\Big(B_1 + B_2+ B_3\Big),
\end{align*}
where 
\begin{align*}
B_1 := & \phantom{+\ }\frac 1 2 \frac 1{z_0^2-\overline z_0^2}
\Big( 
L\big(\tfrac{\alpha+5} 2\big) - (z_0+\overline{z}_0)   L\big(\tfrac{ \alpha+3} 2\big)+  z_0\overline z_0  L \big(\tfrac{\alpha+1} 2\big) 
\Big) =\\
&
\phantom{+\ } \frac 1{z_0 -\overline z_0 }
  \frac \pi {\cos \frac{\alpha }2\pi  } \Big(   z_0^{\frac{\alpha+1} 2} -  \overline z_0^{\frac {\alpha+1} 2} \Big),
\\
B_2 := & -  \frac {\widetilde b_0} 2 \frac 1{z_0^2-\overline z_0^2} \Big(  L\big(\tfrac{\alpha+3} 2\big)    -   (z_0+\overline{z}_0)  L \big(\tfrac{\alpha+1} 2\big)   +    z_0\overline z_0     L\big(\tfrac{\alpha-1} 2\big)  \Big) =\\
&
\phantom{+\ }   \widetilde b_0 \frac 1{z_0 -\overline z_0 } \frac \pi {\cos  \tfrac{\alpha } 2 \pi} \Big(z_0^{\frac{\alpha-1} 2} -   \overline z_0^{\frac{\alpha-1} 2}\Big),
\\
B_3 : = &\phantom{+\ } \frac 1{z_0 -\overline z_0 }
 \Big(M(-\overline z_0)  -M(-z_0)\Big).
\end{align*}
Since $-\overline z_0$ and $-z_0$  are zeros of $\Lambda(z)$, equation \eqref{Mfla} yields
\begin{align*}
&
M(-\overline z_0)  -M(-z_0) =\,  
 \frac { \pi   } {\cos  (\frac \alpha 2 \pi)}  \bigg\{ 
 \overline z_0^{\frac {\alpha+1} 2} -  z_0^{\frac {\alpha+1} 2}  
 + \widetilde b_0   \Big( \overline z_0^{\frac {\alpha-1} 2}    -   z_0^{\frac {\alpha-1} 2}    
\Big)
\\
&
+\frac 1 2 \frac {\mu_\eps^2} {\kappa_\alpha}  
  \frac {1} {\beta} \bigg( 
  \frac { 1}{X(-\beta)} \frac {1} {  z_0+\beta  } 
-\frac { 1}{X(\beta)}  \frac {1} { z_0 -\beta  } 
+ \frac {1}{X(\beta)}  \frac {1} {\overline z_0 -\beta   }   
-\frac { 1}{X(-\beta)}  \frac {1} { \overline z_0 +\beta }  
\bigg)
\bigg\}
\end{align*}
and, collecting all parts together, we finally get
\begin{align*}
P_\infty(\beta, \mu_\eps) =\, 
&
\frac {1}{X(\beta)}   \frac 1 {|z_0-\beta|^2}
  \frac {\mu_\eps^2} {\kappa_\alpha} \frac {1} {2\beta} 
\bigg( 
    \frac{1}{X(\beta)|z_0-\beta|^2} 
-    \frac {1 }{X(-\beta)|z_0+\beta|^2}
\bigg) =\\
&
   \frac {1} {2\beta} 
  \bigg( 
\frac{|z_0+\beta|^2}{|z_0-\beta|^2} \frac { X(-\beta)}{X(\beta)}  
- 1
\bigg),
\end{align*}
where  identity \eqref{XzXz} was used in the last equality. 
This is the large time limit claimed in \eqref{showme_T} for $H<\frac 1 2$.

\subsection{Small noise asymptotics $\boldsymbol{\alpha  \in (1,2)}$}

The corresponding small noise asymptotics \eqref{showme_eps} is derived as in Lemma \ref{lem:limbeta}.

\section{Proof of Theorem \ref{thm:2}} 

The limit expression for the filtering error \eqref{Pinfty_expr} was derived 
when $\alpha_2<\alpha_1<1$, in which case the structural function does not have zeros. 
To demonstrate how the method works in presence of zeros, in this section we will also detail the proof for 
$\alpha_1=\alpha \in (0,1)$ and $\alpha_2=1$. 
To derive the analog of formula \eqref{Pinfty_expr} in this case, we will have to return to the point, where the proof splits into cases,
and reformulate the equivalent problem accordingly.

\subsection{The equivalent problem}
For $\alpha_1 =: \alpha \in (0,2)$ and $\alpha_2=1$, cf. \eqref{Lambdaz},
\begin{equation}\label{LL}
\Lambda(z) =  z^2-\beta^2  - \mu_\eps^2 N_{\alpha}(z).
\end{equation}
Since $N_{\alpha_2}(z)=1$, equations \eqref{eq:6.11} can be written as 
\begin{equation}\label{tildePhieq}
\begin{aligned}
&
\widetilde \Phi_0^+  (t)   -  \frac {\Lambda^+(t)}{\Lambda^-(t)}    \widetilde \Phi_0^-(t) 
=\, 
e^{-tT} \widetilde\Phi_1(-t)\Big(\frac{ \Lambda^+(t) }{\Lambda^-(t)}-1\Big), \\
&
\widetilde \Phi_1^+   (t)   -  \frac{\Lambda^+(t)}{ \Lambda^-(t)} \widetilde \Phi_1^-(t)
=\,  
  e^{-tT}\widetilde \Phi_0(-t) \Big( \frac{\Lambda^+(t) }{ \Lambda^-(t)} - 1\Big) ,
\end{aligned}\qquad t\in \Real_+,
\end{equation}
where we defined, cf. \eqref{Phi0Phi1}, 
\begin{equation}\label{tildePhi}
\begin{aligned}
\widetilde \Phi_0(z):=\, & \Phi_0(z)+   \psi(0)( z-\beta ), \\
\widetilde \Phi_1(z) :=\, &  \Phi_1(z) -\frac {1}{\mu_\eps}( z+\beta ). 
\end{aligned}
\end{equation}
Unlike \eqref{eq:6.11}  equations \eqref{tildePhieq} do not contain additional free term in the right hand side.

When $\alpha \in (0,1)$ the structural function $\Lambda(z)$ has four zeros, see Lemma \ref{lem:4.2}.
In view of \eqref{LL} and definitions \eqref{tildePhi}, the expression  \eqref{gTz_expr} takes the form
\[
\widehat g(z)  = -(z+\beta)
\frac{
 \widetilde \Phi_0(z)  +e^{-zT}  \widetilde \Phi_1(-z)   
}{\Lambda(z)}
+
 \psi(0)+\frac {1}{\mu_\eps}e^{-zT},
\]
and hence removal of the poles implies the conditions, cf. \eqref{poles},
\begin{equation}\label{polesrem}
\begin{aligned}
&
\widetilde \Phi_0(z_0)  + e^{-z_0 T}\widetilde  \Phi_1(-z_0)   =0, \\
&
\widetilde \Phi_1(z_0) + e^{-z_0 T} \widetilde \Phi_0(-z_0)    =0.
\end{aligned}
\end{equation}

Finally, when $\alpha_2=1$, the integral terms in \eqref{cond1} and \eqref{PTfla} vanish, 
and in view of \eqref{gTz_expr}, $\lim_{z\to\infty} \Phi_0(z) =0$
and $\lim_{z\to\infty} \Phi_1(z) = v_{g,1}(T)$, or equivalently, 
\begin{equation}\label{ahad}
\begin{aligned}
& \widetilde \Phi_0(z)\asymp \psi(0)( z-\beta ), \\
& \widetilde \Phi_1(z) \asymp  v_{g,1}(T)-\frac {1}{\mu_\eps}( z+\beta ),
\end{aligned}
\qquad z\to\infty.
\end{equation}

\subsection{Large time limit $\boldsymbol{\alpha \in (0,1)}$}
For $\alpha\in (0,1)$ the angle $\theta(t) =\arg\big(\Lambda^+(t)\big)$ is negative with the limits
\[
\theta(0+) = \frac {1-\alpha}2\pi -\pi
\quad
\text{and}
\quad
\theta(\infty)=0.
\]
Define, cf. \eqref{SzDz},
\begin{equation}\label{tildeSzDz}
\widetilde S(z)   := \frac{\widetilde \Phi_0(z)+\widetilde \Phi_1(z)}{2X(z)}\quad \text{and}\quad 
\widetilde D(z)   := \frac{\widetilde \Phi_0(z)-\widetilde\Phi_1(z)}{2X(z)}.
\end{equation}
Since the functions in \eqref{tildePhi} have the same growth near the origin as in \eqref{apriori_est} and in view of 
\eqref{Xzest}, the choice $k=-1$ in \eqref{Xz} guarantees (square) integrability of the restrictions 
$\widetilde S(-t)$ and $\widetilde D(-t)$,  $t\in \Real_+$ near the origin. Due to the additional linear terms in \eqref{tildePhi}, 
it also implies that $\widetilde S(z)$ and $\widetilde D(z)$ are asymptotic to polynomials of degree two as $z\to\infty$ and therefore, by the Sokhotski-Plemelj theorem, cf. \eqref{eq:6.13}, 
\begin{equation}\label{eq:6.4}
\begin{aligned}
\widetilde S(z) = & \phantom{+\ }\frac 1 \pi \int_0^\infty  \frac{ e^{-tT}h(t)}{t-z} \widetilde S(-t)dt  + P_S(-z), \\
\widetilde D(z) = & -\frac 1 \pi \int_0^\infty  \frac{ e^{-tT}h(t)}{t-z} \widetilde D(-t)dt  + P_D(-z),
\end{aligned}
\end{equation}
with  polynomials 
\begin{equation}\label{PSPD}
P_S(z) =\,  k^S_2 z^2 + k^S_1 z + k^S_0\quad\text{and}\quad
P_D(z) =\,   k^D_2 z^2 + k^D_1 z + k^D_0,
\end{equation}
where the coefficients are constants, possibly dependent on $T$ and $\eps$.

In this case $\theta(t)=O(t^{\alpha-3})$ as $t\to\infty$ and hence the exponent in \eqref{Xz} satisfies   
\begin{align}
\nonumber
X_c(z) := & \exp \left(\frac 1 \pi \int_0^\infty \frac {\theta(t)}{t-z}dt\right) = 
\exp \Big(
- \frac 1 z m_0  
- \frac 1 {z^2} m_1
+O(z^{\alpha-3})  
\Big)= \\
& \label{Xcz}
1   - \frac 1 z  b_0  - \frac 1 {z^2} b_1 + O(z^{\alpha-3}), \quad z\to\infty,
\end{align}
where  $b_0 = m_0$,  $b_1 =  m_1 -\frac 1 2 m_0^2$ and
\[
m_j = \frac 1 \pi \int_0^\infty  t^j \theta(t) dt.
\]
Consequently, in view of \eqref{ahad}, the asymptotic  terms in \eqref{tildeSzDz} and \eqref{eq:6.4}  match, if the
coefficients in \eqref{PSPD}  satisfy  
\begin{equation}\label{kskdeq}
\begin{aligned}
k^S_2 =\, & -\frac 1 2 \Big(\psi(0)-\frac 1{\mu_\eps}\Big), & k^S_1 =\, & - k^S_2 b_0 - \frac   \beta 2 \Big(\psi(0)+\frac 1{\mu_\eps}\Big) +
\frac  { v_{g,1}(T)} 2, \\
k^D_2 = & -\frac 1 2\Big(\psi(0)+\frac  1{\mu_\eps}\Big), & 
k^D_1 =&  -k^D_2 b_0 -\frac  \beta 2 \Big(\psi(0)-\frac 1{\mu_\eps}\Big) - \frac  {v_{g,1}(T)} 2.
\end{aligned}
\end{equation}

As in the previous sections, the auxiliary integral equations 
\begin{equation}\label{qqpp}
\begin{aligned}
p_j(t) = & \phantom{+\ }\frac 1 \pi \int_0^\infty  \frac{ e^{-\tau T}h(\tau)}{\tau+t} p_j(\tau) d\tau  + t^j, \\
q_j(t) = & -\frac 1 \pi \int_0^\infty  \frac{ e^{-\tau T}h(\tau)}{\tau+t} q_j(\tau) d\tau  + t^j,
\end{aligned}
\end{equation}
have unique solutions, whose analytic extensions satisfy, cf. \eqref{asym_est},
\begin{equation}\label{pjqjest}
|p_j(z) - z^j| \le C \frac 1 z \frac 1 T \quad \text{and}\quad |q_j(z) - z^j| \le C \frac 1 z \frac 1 T.
\end{equation}
By linearity 
\begin{equation}\label{ksd}
\begin{aligned}
\widetilde S(z) =\, & k^S_2 p_2(-z) + k^S_1 p_1(-z) + k^S_0 p_0(-z),\\
\widetilde D(z) =\, & k^D_2 q_2(-z) + k^D_1 q_1(-z) + k^D_0 q_0(-z).
\end{aligned}
\end{equation}
We can now use \eqref{tildeSzDz} and \eqref{ksd} to express  
$\widetilde \Phi_0(z)$ and $\widetilde \Phi_1(z)$ in terms of the above constants and solutions to integral equations \eqref{qqpp}.
Then plugging these expressions into \eqref{polesrem} and using the estimates  \eqref{pjqjest} we obtain, as $T\to\infty$,
\begin{equation}\label{keqs}
\begin{aligned}
&
z_0^2  k^S_2 - z_0 k^S_1  + k^S_0 +z_0^2 k^D_2  -  z_0 k^D_1   + k^D_0    \asymp 0, \\
&
z_0^2 k^S_2  - z_0 k^S_1  + k^S_0   - z_0^2 k^D_2  + z_0 k^D_1  - k^D_0    \asymp 0.
\end{aligned}
\end{equation}

Powers of $z_0$ have nonzero complex parts and hence these are, in fact, four equations with real valued coefficients. 
Thus we arrive at a system of eight linear equations \eqref{kskdeq} and \eqref{keqs}
for the limiting values of the eight unknowns, namely $\psi(0)$, $v_{g,1}(T)$ and $k^S_j$ and $k^D_j$, $j=0,1,2$.

Taking the imaginary part of the equations above we get rid of $k^S_0$ and $k^D_0$:
\begin{align*}
&
\Im(z_0^2)  (k^S_2 +  k^D_2) - \Im(z_0) (k^S_1    + k^D_1)         \asymp 0, \\
&
\Im(z_0^2) (k^S_2  -   k^D_2) - \Im(z_0) (k^S_1        -  k^D_1)      \asymp 0.
\end{align*} 
Now from \eqref{kskdeq}
\begin{align*}
k^S_2 + k^D_2 =\, & -  \psi(0)  , & 
k^S_1 +k^D_1 =\, &   (b_0 -   \beta ) \psi(0)   \\
k^S_2  -k^D_2 =\, &    \frac 1{\mu_\eps},   & 
k^S_1 - k^D_1  =\, & - \frac 1{\mu_\eps} (b_0 +   \beta )  +v_{g,1}(T)      
\end{align*}
and hence 
\begin{align*}
&
-\Im(z_0^2)     \psi(0) - \Im(z_0) (b_0 -   \beta ) \psi(0)         \asymp 0, \\
&
\Im(z_0^2)  \frac 1{\mu_\eps}  + \Im(z_0)  \frac 1{\mu_\eps} (b_0 +   \beta )  -   \Im(z_0)v_{g,1}(T)     \asymp 0.
\end{align*} 
This implies that $\psi(0)\asymp 0$ and 
\[
P_T(\beta, \mu_\eps ) = \frac 1 {\mu_\eps} v_{g,1}(T) \asymp     \frac 1{\mu_\eps^2} \Big( b_0 +   \beta    + \frac{\Im(z_0^2)}{\Im(z_0)}  \Big)   =
\frac 1{\mu_\eps^2} \Big( b_0 +   \beta    + 2  \Re(z_0)   \Big),
\]
which is the formula claimed  in \eqref{Pinfty1}.

\subsection{Small noise asymptotics $\boldsymbol{\alpha  \in (0,1)}$}
The expression in \eqref{Peps0} follows from Theorem \ref{thm:main}, since for $\beta=0$ and $\eps=1$, 
the zero of $\Lambda(z)$ in the first quadrant can be found explicitly,
\[
z_0 = (\mu^2 \kappa_\alpha)^{\frac 1 {3-\alpha}}  \exp\Big(\tfrac{1-\alpha}{3-\alpha}\tfrac \pi 2 i\Big),
\]
and the first moment of $\theta(t)$ can be computed in the closed form  
\[
b_0 = - (\mu^2\kappa_\alpha)^{\frac 1{3-\alpha}} \frac{\sin \frac \pi 2 \frac{1+\alpha}{3-\alpha}}{\sin \frac \pi {3-\alpha}}.
\]

\subsection{Large time limit $\boldsymbol{\alpha \in (1,2)}$} 
In this case, $\theta(t)=\arg\big(\Lambda^+(t)\big)$ is positive and 
\[
\theta(0+) = \pi\quad \text{and}\quad\theta(\infty) = 0.
\]
In view of estimates \eqref{apriori_est} and \eqref{Xzest}, the suitable choice of the power factor in \eqref{Xz} is $k=1$, 
which guarantees that the restrictions of \eqref{tildeSzDz} to $\Real_-$ are (square) integrability near the origin. 
This choice and \eqref{tildePhi} imply that $\widetilde S(z)$ and $\widetilde D(z)$ are of order $O(1)$ as $z\to\infty$. 
Hence representation \eqref{eq:6.4} hold with 
$
P_S(z) = k_0^S
$
and
$
P_S(z) = k_0^D,
$
and hence, by linearity,  
\[
\widetilde S(z) = k^S_0 p_0(-z) \quad \text{and}\quad \widetilde D(z) = k^D_0 q_0(-z).
\]
Comparing this with \eqref{tildeSzDz} implies 
\begin{equation}\label{eqkk}
k^S_0 = -\frac 1 2 \Big(\psi(0)-\frac 1 {\mu_\eps}\Big)\quad \text{and}\quad k^D_0 = -\frac 1 2 \Big(\psi(0)+\frac 1 {\mu_\eps}\Big).
\end{equation}

The filtering error can now be found from \eqref{PTfla}, where for $\alpha_2 =1$ the last term vanishes. Plugging \eqref{tildePhi} into   \eqref{gTz_expr} yields
\begin{align*}
P_T  =\, &  \frac 1 {\mu_\eps}\lim_{\Re(z)\to 0}z e^{-zT} \widehat g(-z) = 
\frac 1 {\mu_\eps}\lim_{\Re(z)\to 0}
z 
(z-\beta) \frac{ \widetilde \Phi_1(z)+\frac {1}{\mu_\eps}( z+\beta )}{\Lambda(-z)}  = 
\\
&
\frac 1 {\mu_\eps}\lim_{\Re(z)\to 0}
\frac{z (z-\beta)}{\Lambda(-z)}\bigg( X(z) \Big(k^S_0 p_0(-z)-k^D_0q_0(-z)\Big)+\frac {1}{\mu_\eps}( z+\beta )\bigg) \asymp \\
& \frac 1 {\mu_\eps^2} (b_0  +\beta) = P_\infty(\beta,\mu_\eps),
\end{align*}
where we used \eqref{eqkk} and the approximation, cf. \eqref{Xcz},
\[
X(z) = -z \Big(1-b_0 z^{-1} + o(z^{-1})\Big),\quad z\to\infty. 
\]

\subsection{Small noise asymptotics $\boldsymbol{\alpha\in (1,2)}$}
In this case, for $\beta =0$ and $\eps=1$, 
\[
b_0 =   (\mu^2\kappa_\alpha)^{\frac 1{3-\alpha}} \frac{1}{\sin \frac \pi {3-\alpha}},
\]
and the small noise asymptotics follows by virtue of Theorem \ref{thm:main}.

\appendix
\section{More on solvability of \eqref{maineq}}\label{sec:A}

Solvability of equation \eqref{maineq} in a space, suitable for our purposes, is a subtle matter.
In essence, we construct such a solution which, moreover, turns out to be 
amenable to asymptotic analysis. The roadmap of our construction is outlined in Section \ref{sec:preview}.
and this section gives an extended discussion of the solvability question, outlines several alternative approaches and 
elaborates on the construction in this paper. 
 
\subsection{Integration of nonrandom functions with respect to fBm}
Let us briefly recall a construction of stochastic integrals with respect to fBm $B^H$ for nonrandom integrands (see, e.g., \cite{KLR00}).
For any $H\in (0,1)\setminus \{\frac 1 2\}$, consider the function space,
$$
\Lambda_T^{H-\frac 12 } = \Big\{f: [0,T]\mapsto \Real \ \text{such that\ } \int_0^T \big(s^{\frac 1 2-H}(\Psi_T f)(s)\big)^2ds<\infty \Big\},
$$
where $\Psi_T$ is the operator 
$$
(\Psi_T f)(s) :=- 2H \frac d {ds} \int_s^T f(r) r^{H-\frac 1 2}(r-s)^{H-\frac 1 2}dr.
$$
The bilinear form 
\begin{equation}\label{sp}
\langle f,g\rangle_{\Lambda_T^{H-\frac 12 }} = \frac{2-2H}{\lambda_H}\int_0^T s^{1-2H} (\Psi_T f)(s) (\Psi_T g)(s)ds,
\end{equation}
where  $\lambda_H$ is an explicit constant, 
defines a scalar product on $\Lambda_T^{H-\frac 12 }$. It can be shown that 
\begin{equation}\label{iso}
\begin{aligned}
\langle f,g\rangle_{\Lambda_T^{H-\frac 12 }}  = &
 \int_0^T  f(s) \frac \partial {\partial s} \int_0^T g(t)\frac \partial {\partial t}  K(s,t)dtds =\\
 &
H\int_0^T  f(s) \frac \partial {\partial s} \int_0^T g(t) |s-t|^{2H-1} \sign(s-t)  dtds,
\end{aligned}
\end{equation}
where $K(s,t)$ is the covariance function \eqref{KV} of the fBm. These and other related formulas 
can be found in, e.g., \cite[Subsection 3.3]{CCK} for a quick reference. 

\begin{rem}
Various useful relations of $\Lambda_T^{H-\frac 12 }$ to other spaces are known. In particular, 
it can be shown that $L^2([0,T])\subset \Lambda_T^{H-\frac 12 }$ for $H>\frac1 2$, and 
$\Lambda_T^{H-\frac 12 }\subset  L^2([0,T])$ for $H<\frac 1 2$,  see \cite{PT01}.
Also these inclusions follow from the eigenvalues asymptotics of the covariance operator of the fractional noise, the formal
derivative of the fBm, see \cite{ChK}.
\end{rem}

Let $\mathcal E$ be the space of all simple functions. 
For $g\in \mathcal E$, the stochastic integral $I_T(f):=\int_0^T g(s)dB^H_s$ is defined as the Riemann-Stieltjes
sum. In view of \eqref{iso}, for $f,g\in \mathcal E$,
\begin{equation}\label{isom}
\E \int_0^T f(s)dB^H_s\int_0^T g(s)dB^H_s = \langle f,g\rangle_{\Lambda_T^{H-\frac 12 }}.
\end{equation}
Hence the stochastic integral $I_T$ defines an isometry between $\mathcal E$ with the scalar product \eqref{sp}
and the linear subspace $\mathrm{sp}_T(B^H)\subset L^2(\Omega)$ of the finite linear combinations of increments of $B^H$ 
with the scalar product $\langle \xi,\eta\rangle =\E \xi \eta$ for $\xi, \eta \in \mathrm{sp}_T(B^H)$. 

It can be shown that $\mathcal E$ is dense in $\Lambda_T^{H-\frac 12 }$ and this allows to extend the isometry $I_T$ to any 
$g\in \Lambda_T^{H-\frac 12 }$ by means of the $L^2(\Omega)$ limit. Namely, let $g^n$ be any sequence of simple functions, 
such that $\|g^n-g\|_{\Lambda_T^{H-\frac 12 }}\to 0$, then, by the isometry property \eqref{isom}, 
$$
\E\big(I_T(g^n) - I_T(g^m)\big)^2 = \|g^n-g^m\|_{\Lambda_T^{H-\frac 12 }} \xrightarrow[n,m\to\infty] {}0, 
$$
which means that $I_T(g^n)$ is Cauchy in $L^2(\Omega)$. By completeness  of $L^2(\Omega)$, the limit 
$$
I_T(g) := \lim_{n\to\infty} I_T(g^n) =: \int_0^T g(s)dB^H_s 
$$
exists.  This limit does not depend on the choice 
of approximating sequence $(g^n)\subset\mathcal E$ and hence defines the stochastic integral of any $g\in \Lambda_T^{H-\frac 12 }$
unambiguously. Moreover, the extended map $I_T: \Lambda_T^{H-\frac 12 }\mapsto L^2(\Omega)$ preserves the isometric 
property \eqref{isom}.

Let $\overline {\mathrm{sp}}_T(B^H)$ be the closure of $\mathrm{sp}_T(B^H)$ in $L^2(\Omega)$, that is, the subspace of 
all $L^2(\Omega)$ limits of linear combinations of increments of $B^H$. The image of the extended isometry $I_T$ is some linear 
subspace of $\overline {\mathrm{sp}}_T(B^H)$. Does it coincide with the whole closure $\overline {\mathrm{sp}}_T(B^H)$? 
Since $\overline{\mathrm{sp}}_T(B^H)$ is a complete subspace and $I_T$ is an isometry, the answer to this 
question is affirmative if and only if the space $\Lambda_T^{H-\frac 12 }$ is complete. Indeed, by definition, 
for any $\xi \in \overline {\mathrm{sp}}_T(B^H)$ there is a sequence of random variables $\xi^n = I_T(g^n)$
such that $\xi_n \to \xi$ in $L^2(\Omega)$. Since $L^2(\Omega)$ is complete, $\xi_n$ is a Cauchy sequence and thus 
$$
\|g^n-g^m\|_{\Lambda_T^{H-\frac 12}} = \E (I_T(g^n)-I_T(g^m))^2\xrightarrow[n,m\to\infty]{}0,
$$
that is, $g^n$ is Cauchy in $\Lambda_T^{H-\frac 12}$. Now if $\Lambda_T^{H-\frac 12}$ is complete, then $g=\lim_n g^n\in \Lambda_T^{H-\frac 12}$ exists
and $I_T(g)=\xi$, $\P$-a.s. Conversely, let $g^n$ be a Cauchy sequence in $\Lambda_T^{H-\frac 12}$, then $I_T(g^n)$ is Cauchy in $L^2(\Omega)$.
By completeness of $L(\Omega^2)$ there is $\xi=\lim_n\xi_n\in  \overline {\mathrm{sp}}_T(B^H)$. Assume now that any random variable in 
$\overline {\mathrm{sp}}_T(B^H)$ is an image of $I_T$, then there exists $g\in \Lambda_T^{H-\frac 12}$ such that $\xi = I_T(g)$
and 
$$
\|g^n-g\|_{\Lambda_T^{H-\frac 12}}= \E (I_T(g^n)-I_T(g))^2= \E (\xi_n-\xi)^2\to 0,
$$
which means that $g^n$ is convergent in $\Lambda_T^{H-\frac 12}$ and hence the latter is complete.

It was shown in \cite{PT01} that the space $\Lambda_T^{H-\frac 12 }$ is complete for $H<\frac 1 2$, but it is 
incomplete for $H>\frac 1 2$. An important implication is that there are random variables in $\overline {\mathrm{sp}}_T(B^H)$
which cannot be represented as stochastic integrals of functions in $\Lambda_T^{H-\frac 12 }$ with respect to fBm 
with $H>\frac 1 2$.

An additional insight is given by the canonical innovation representation of the fBm, see \cite{KLR00}. The 
innovation Brownian motion 
$$
W_t = \int_0^t \sqrt{\frac{2-2H}{\lambda_H}}(\Psi_t^{-1}u^{H-\frac 1 2})(s)dB^H_s
$$
generates the same filtration as $B^H$. Then by the martingale representation theorem
for any $\xi\in \overline {\mathrm{sp}}_T(B^H)$ there exists a function $f\in L^2([0,T])$ such that 
\begin{equation}\label{xi}
\xi = \int_0^T f(s)dW_s.
\end{equation}
Moreover, 
\begin{equation}\label{WB}
\int_0^T f(s)dW_s = \int_0^T  \sqrt{\frac{2-2H}{\lambda_H}}(\Psi_T^{-1}u^{H-\frac 1 2}f(u))(s)dB^H_s,
\end{equation}
whenever the integrand in the right hand side is well defined and belongs to $\Lambda_T^{H-\frac 12}$.
For $H>\frac 1 2$ it is possible to find a function $f\in L^2([0,T])$ such that \eqref{WB} fails and 
for such a function the random variable \eqref{xi} will not be representable 
by the stochastic integral with respect to the fBm.  

\subsection{The linear filtering problem}

Let $X$ be a zero mean Gaussian process with integrable cadlag paths and $B^H$
an independent fractional Brownian motion with Hurst parameter $H\in (0,1)$. 
Define the process 
$$
Y_t = \int_0^t X_sds + B^H_t, \quad t\ge 0.
$$
Since all the processes are Gaussian, for any $T>0$, the conditional expectation $\widehat X_T = \E(X_T|\mathcal F^{Y}_T)$ 
is a random variable, which belongs to $\overline{\mathrm{sp}}_T(Y)$,  the closure in $L^2(\Omega)$ of all linear 
combinations of increments of $Y$ on the interval $[0,T]$. In view of the discussion in the previous section, the following question arises.

\begin{quote}
\em Does there exist a non-random function $g \in \Lambda_T^{H-\frac 1 2}\cap L^1([0,T])$ such that 
$$
\widehat X_T = \int_0^T g(s) dY_s, \quad \P-a.s.?
$$
If so how can it be found?
\end{quote}

\subsection{Refresh on the standard Brownian case ($\mathbf{H=\frac 1 2}$)}

Let us briefly recall how an affirmative answer to the above question is given in \cite{LS1} in the standard Brownian case.
Let $\F^{Y,n}_T$ be the $\sigma$-algbera generated by the increments of process $Y$ on the $2^n$-dyadic partition of $[0,T]$. 
Since $Y$ is a cadlag process, $\F^Y_T=\bigvee_{n=1}^\infty\F^{Y,n}_T$ and by Levy's zero-one law, 
$$
\widehat X^n_T=\E (X_T|\F^{Y,n}_T)\xrightarrow[n\to\infty]{a.s.} \E(X_T|\F^Y_T)=\widehat X_T,
$$
where the convergence holds in $L^2(\Omega)$ as well. 
By the normal correlation theorem
$$
\widehat X^n_T = \int_0^T g^n(t) dY_t,
$$
where $g^n(\cdot)$ is a simple function.  
Being convergent, the sequence $\widehat X^n_T$ is Cauchy in $L^2(\Omega)$.
On the other hand, by independence of $X$ and $B^{1/2}$, 
$$
\E\big(\widehat X^n_T-\widehat X^m_T\big)^2     
\ge \int_0^T (g^n(s)-g^m(s))^2 ds,
$$
and hence $g^n(\cdot)$ is Cauchy in $L^2([0,T])$. Since this space is complete, $g^n(\cdot)$ converges to some $g\in L^2([0,T])$.
For such a function the stochastic integral  $\int_0^T g(s)dY_s$ is well defined, and
$$
\E\Big(\widehat X_T - \int_0^T g(t)dY_t\Big)^2 \le 
2 \E\big(\widehat X_T -\widehat X^n_T)^2 +2\E\Big( \int_0^T \big(g^n(t)-g(t)\big)dY_t\Big)^2\to 0,
$$
and hence 
$$
\widehat X_T = \int_0^T g(t)dY_t, \quad \P-a.s.
$$
Now using the orthogonality property of the conditional expectation, the standard calculation shows that $g(\cdot)$ solves the equation 
\begin{equation}\label{WH}
g(t)+\int_0^T g(s) K_X(s,t)ds   = K_X(t,T), \quad \text{for a.a.}\ t\in [0,T],  
\end{equation}
where $K_X(s,t)$ is the covariance function of $X$.

Conversely, we could have started with the equation \eqref{WH}: if it has a solution $g\in L^2([0,T])$, then the stochastic 
integral $\widehat X_T := \int_0^T g(s)dY_s$ is well defined and satisfies the usual properties. Then \eqref{WH} implies 
that $X_T - \widehat X_T$ is orthogonal to any random variable in $\overline {\mathrm{sp}}_T(Y)$ and hence $\widehat X_T = \E(X_T|\F^Y_T)$.
In fact the Fredholm equation \eqref{WH} does have such a solution if the kernel $K_X(s,t)$ is Hilbert-Schmidt.

\subsection{The fractional case $\mathbf H\in (0,1)\setminus \{\frac 1 2\}$}

The arguments from the previous section apply to $\Lambda_T^{H-\frac 12}$ with $H<\frac 1 2$, since this subspace is complete and, 
moreover, $\Lambda_T^{H-\frac 12}\subset L^2([0,T])\subset L^1([0,T])$. The latter inclusion ensures that the limit weight function $g$
is not only integrable with respect to fBm, but also with respect to the time variable, so that $\int_0^T g(s)dY_s$  is indeed well defined. 

On the other hand, the space $\Lambda_T^{H-\frac 12}$ for $H>\frac 12$, being incomplete, must be treated differently.  
It suffices to show that equation, cf. \eqref{maineq},
\begin{equation}\label{fWH}
\frac \partial {\partial s} \int_0^T g(r) \frac {\partial} {\partial r} K(r,s)dr + \int_0^T K_X(r,s)g(r)dr   =K_X(s,T),\quad s\in (0,T),
\end{equation}
has a solution in $L^1([0,T])$. Note that in this case it also belongs to $\Lambda_T^{H-\frac 12}$: indeed, taking scalar product of  
\eqref{fWH} with $g$ implies  
\begin{equation}\label{LLb}
\|g\|_{\Lambda_T^{H-\frac 12}}^2\le \int_0^T g(s) K_X(s,T)ds \le \|K_X(\cdot,T) \|_\infty \|g\|_1.
\end{equation}
This is a special feature of the solution to the above equation, since $L^1([0,T])\not\subseteq \Lambda_T^{H-\frac 12}$.
For $g\in L^1([0,T])\cap \Lambda_T^{H-\frac 12}$, the stochastic integral 
$$
\widehat X_T = \int_0^T g(s)dY_s := \int_0^T g(s)dB^H_s + \int_0^T g(s)X_s ds
$$ 
is well defined, if, e.g. $X$ has continuous paths.
Moreover, equation \eqref{fWH} implies that
$$
\E \Big(X_T -\int_0^T g(s)dY_s\Big)\int_0^T h(s)dY_s =0
$$ 
for any $h\in \Lambda_T^{H-\frac 12}$. Since  $\xi \in \overline {\mathrm{sp}}_T(Y)$ can be approximated 
in $L^2(\Omega)$ by a sequence of stochastic integrals of {\em simple} functions, in fact, we have
$$
\E \Big(X_T -\int_0^T g(s)dY_s\Big)\xi =0,
$$ 
for any such $\xi$ and, therefore, $\widehat X_T$ is a version of conditional expectation.
In Section \ref{sec:A8} we detail the calculations which show that the solution to \eqref{fWH} constructed in this paper 
does indeed belong to $L^1([0,T])$.  

\subsection{Solvability through explicit inverse}
For simplicity, let $T=1$ and omit it from the notations. 
Define the operators 
$$
(A  f)(s) =   \frac \partial {\partial s} \int_0^1 f(s) \frac {\partial} {\partial r} K(r,s)dr  =
H  \frac \partial {\partial s} \int_0^1 f(t) |s-t|^{2H-1} \sign(s-t)  dt
$$
and 
$$
(R   f)(s) =  \int_0^1 K_X(r,s)f(r)dr.
$$
Then equation \eqref{fWH} reads
\begin{equation}\label{AK}
A  g +R  g = f,
\end{equation}
where $f(s) = K_X(s,1)$. The operator $A$ is invertible, see e.g. \cite{LB98}, 
\begin{align}
& \label{Ainv}
(A^{-1} f)(s) =  \\
\nonumber
& - c_H s^{\frac 1 2-H}\frac d{ds}\int_s^1 dw w^{2H-1} (w-s)^{\frac 1 2 -H}\frac d{dw} \int_0^w   z^{\frac 1 2-H}
(w-z)^{\frac 1 2-H}f(z)dz,
\end{align}
where $c_H$ is an explicit constant. 
If $A^{-1}$ is applicable to both sides of \eqref{AK}, then $g$ solves the equation
\begin{equation}\label{ARg}
g + A^{-1} R g = A^{-1} f.
\end{equation}
If moreover, the operator $A^{-1}R$ and the function $A^{-1} f$ are sufficiently regular, then \eqref{ARg} is a Fredholm equation of the second kind, 
and its solvability in an appropriate space follows from the general theory.

\subsubsection{Case $H<\frac 1 2$}
In this case, the derivatives and integrals in \eqref{Ainv} are interchangeable 
and the inverse of $A$ turns out to be a weakly singular integral operator, see \cite[Theorem 5.1 (iv)]{CCK}, 
$$
(A^{-1} f)(s)  = \int_0^1 L(u,v) f(v)du,
$$
where $L(u,v)=|u-v|^{-2H}N(u,v)$ with $N\in C([0,1]^2)$. Since $K_X\in C([0,1]^2)$ it follows that $A^{-1} R$ in this case is an integral 
operator with continuos kernel. Similarly, $A^{-1} f$ is a continuous function. This reduces \eqref{fWH} to a Fredholm 
equation with continuous kernel. The homogeneous equation $g + A^{-1} R g=0$, or equivalently, $A g +Rg=0$ has only trivial 
solutions, since all eigenvalues of $A$ are positive, see \cite{ChK}. 
Hence by the Fredholm alternative, the non-homogeneous equation \eqref{ARg}
has the unique (continuous) solution (see, e.g., \cite{RN55}).
Then, in view of \eqref{LLb}, $g\in \Lambda_T^{H-\frac 1 2}\cap L^1([0,T])$.

\subsubsection{Case $H>\frac 1 2$}Note that  $f(z)=\int_0^z f'(r)dr$ and hence \eqref{Ainv} in this case can be rewritten as 
$$
(A^{-1} f)(x) =  
\int_0^1 f'(r) p(x,r)dr,
$$
where
\begin{align}
&\nonumber 
p(x,u) = 
- c_H x^{\frac 1 2-H}\Big((2-2H)  \int_r^1 z^{\frac 1 2-H} \rho(x,z) dz + r^{\frac 3 2-H}  \rho(x,r)\Big),
\end{align}
and 
$$
\rho(x,r)=\frac d{dx}\int_{x\vee r}^1  w^{2H-2} (w-x)^{\frac 1 2 -H}  (w-r )^{\frac 1 2-H} dw.
$$
A calculation shows that 
$$
|\rho(x,r)| \le\,   C |x-r|^{1-2H}(x\vee r)^{2H-2} + (1-x)^{\frac 1 2 -H}  (1-r )^{\frac 1 2-H}\one{x>r}.
$$
Further computations become cumbersome, but if pushed, seem to lead to a Fredholm equation of the second kind, 
with $L^2([0,T]^2)$ kernel and $L^1([0,T])$ forcing function. Then it remains to be checked that such equations 
have $L^1([0,T])$ solution. We will not pursue this direction here.

\subsection{Solvability by construction in this paper}\label{sec:A8}

Let us recap the main steps of the construction in this paper, considering, for definiteness,  
the case $\alpha_2 \in (0,1)$ and $\alpha_1>\alpha_2$, which corresponds to $H_2>\frac 1 2$ and $H_1<H_2$; 
all other cases can be treated similarly. 
In order to avoid confusion with the notations used in the text, all the objects produced in the course of construction
will be marked by asterisk. 

\medskip

\begin{enumerate}
\addtolength{\itemsep}{0.7\baselineskip}
\renewcommand{\theenumi}{\alph{enumi}}

\item\label{a} Solve the integral equations \eqref{pqeq}, treating the constant $\psi(0)$, which appears in \eqref{fSfD} and thus also in 
\eqref{F_SF_D}, as a free parameter and denote it by $c$.   
As explained in the text, these equations have unique solution in $L^2(\Real_+)$,
at least for all sufficiently large $T$ or all sufficiently small $\eps>0$. 

\item   Form the functions $\Phi_0^*(z)$ and $\Phi_1^*(z)$ using the formulas \eqref{PhiPhi}, 
where $X(z)$ is defined by \eqref{Xz}.

\item\label{c}  Compute the function $\widehat g^*(z)$ by the formula \eqref{gTz_expr} using $\Phi_0^*(z)$ and $\Phi_1^*(z)$ and the parameter $c$ 
in place of $\psi(0)$. 
Substitute this expression into condition \eqref{cond1_alpha_small} and solve the obtained equation for $c$. 
Denote the obtained value by $c^*$. Thus we obtain 
\begin{equation}\label{gTzstar}
\widehat g^*(z)  = -\frac 1 {\Lambda(z)}
\bigg(
 (z+\beta) \Big(\Phi_0^*(z)+e^{-zT} \Phi_1^*(-z)\Big) +  \mu_\eps^2  N_{\alpha_1}(z) 
\Big(c^*+\frac {1}{\mu_\eps}e^{-zT}\Big)
\bigg).
\end{equation}

\item Compute $g^*(x)$ by applying the inverse Laplace transform to $\widehat g^*(z)$.

\end{enumerate}

\medskip

Our goal is to show that 

\medskip

\begin{enumerate}
\addtolength{\itemsep}{0.7\baselineskip}
\renewcommand{\theenumi}{\roman{enumi}}

\item\label{i} $g^*\in  L^1([0,T])$ and

\item\label{iig}  $g^*$ does indeed solve equation \eqref{fWH}.

\end{enumerate}

\medskip

Note that if  \eqref{i} and \eqref{iig} hold, then the bound \eqref{LLb} implies $g^* \in \Lambda^{H-1/2}_T$ as well. 

\medskip

\subsubsection{Proof of \eqref{i}}
The proof is by a careful inspection of all the functions involved in the construction. 
Let us estimate the growth of the functions $f_S(t)$ and $f_D(t)$ defined in \eqref{fSfD} at infinity and near the origin. 
To this end, write  
\begin{align*}
&
\bigg|\frac{N_{\alpha_2}^-(t)}{X^-(t)}-\frac{N_{\alpha_2}^+(t)}{X^+(t)} \bigg| = 
\bigg|\frac{N_{\alpha_2}^-(t)}{X^-(t)}\bigg|\bigg| \frac{N_{\alpha_2}^-(t) \Lambda^+(t)-N_{\alpha_2}^+(t)\Lambda^-(t)}{N_{\alpha_2}^-(t) \Lambda^+(t)}  \bigg | =\\
&
\mu^2_\eps  \frac{1}{|X^-(t)|} \frac 1{|\Lambda^+(t)|}
\Big|N_{\alpha_2}^+(t) N_{\alpha_1}^-(t)-  N_{\alpha_2}^-(t) N_{\alpha_1}^+(t)\Big|,
\end{align*}
where the last equality is obtained by plugging the expression \eqref{Lambdaz}.
In view of the formula \eqref{Nalpha} and the estimates \eqref{Xzest} (with $k=1$) and since $\alpha_1>\alpha_2$, it follows that 
$$
\bigg|\frac{N_{\alpha_2}^-(t)}{X^-(t)}-\frac{N_{\alpha_2}^+(t)}{X^+(t)} \bigg| = \begin{cases}
O(t^{\alpha_1-4+\frac {1-\alpha_2}2}), & t\to\infty, \\
O(t^{\alpha_1-1+\frac {1-\alpha_2} 2}), & t\to 0.
\end{cases} 
$$
Consequently, cf.  \eqref{fSfD}, 
$$
\big\{f_S(t), f_D(t)\big\} = \begin{cases}
O(t^{\alpha_1-3+\frac {1-\alpha_2}2}), & t\to\infty, \\
O(t^{\alpha_1-1+\frac {1-\alpha_2} 2}), & t\to 0,
\end{cases} 
$$
and, by definitions \eqref{F_SF_D}, 
$$
\big\{F_S(-t), F_D(-t)\big\} =  
\begin{cases}
O\big(t^{-1\vee (\alpha_1-3+\frac {1-\alpha_2}2)}\big), & t\to\infty, \\
O(t^{0\wedge (\alpha_1-1+\frac {1-\alpha_2} 2)}), & t\to 0.
\end{cases} 
$$
Note that $F_S,F_D\in L^2(\Real_+)$ and hence equations \eqref{pqeq} 
have unique solutions in $L^2([0,T])$, as explained in the text.  
Then in view of estimates \eqref{asym_est}, the formulas \eqref{PhiPhi} imply that $\Phi_0^{*\pm }(t), \Phi_1^{*\pm}(t)$ are 
locally square integrable on $\Real_+$ and 
$$
\big\{\Phi_0^{*\pm }(t), \Phi_1^{*\pm}(t)\big\} = O(t^{\alpha_1-2}) \quad \text{as\ } t\to\infty.
$$

By construction $\widehat g^*(z)$ is an entire function, the Laplace transform inversion can be carried out by integration on the
imaginary axis
\begin{equation}\label{gTx}
g^*(x) = \frac 1 {2\pi i}\int_{-i\infty}^{i\infty} \widehat g^*(z)e^{zx}dz.
\end{equation}
Since $\widehat g^*(z)$ is analytic and $\widehat g^*(z)\to 0$ as $z\to\infty$ and $\Re(z)>0$, its inverse Laplace transform 
vanishes on $\Real_-$, i.e., $g^*(x)=0$ for $x<0$. Similarly, since $e^{zT} \widehat g^*(z)\to 0$ as $z\to\infty$ and $\Re(z)<0$, 
it follows that $g^*(x)=0$ for $x>T$. For $x\in (0,T)$, \eqref{gTx} can be evaluated by 
plugging the expression \eqref{gTzstar} and integrating along arc sector counters in each quarter of the complex plane.
After a rearrangement this gives  
\begin{equation}\label{gTfla}
\begin{aligned}
g^*(x) = & -\frac 1 {2\pi i}\int_{-i\infty}^{i\infty} \frac 1 {\Lambda(z)} \Big( (z+\beta) \Phi_0^*(z)  + \mu_\eps^2 N_{\alpha_1}(z) c^* \Big) e^{zx}dz \\
&
-\frac 1 {2\pi i}\int_{-i\infty}^{i\infty} \frac 1 {\Lambda(z)} \Big((z+\beta)  \Phi_1^*(-z) + \mu_\eps N_{\alpha_1}(z) \Big)e^{z(x-T)}dz =\\
&
-  \frac 1 {2\pi i}\int_0^\infty \frac 1 {\Lambda^-(t)} \Big( (-t+\beta) \Phi_0^*(-t)  + \mu_\eps^2 N_{\alpha_1}^-(t) c^* \Big) e^{-tx}dt \\
&
+\frac 1 {2\pi i}\int_{0}^{\infty} \frac 1 {\Lambda^+(t)} \Big( (-t+\beta) \Phi_0^*(-t)  + \mu_\eps^2 N_{\alpha_1}^+(t) c^* \Big) e^{-tx}dt \\
&
-\frac 1 {2\pi i}\int_{0}^{\infty} \frac 1 {\Lambda^+(t)} \Big((t+\beta)  \Phi_1^*(-t) + \mu_\eps N_{\alpha_1}^+(t)     \Big)e^{-t(T-x)}dt \\
&
+\frac 1 {2\pi i}\int_{0}^{\infty} \frac 1 {\Lambda^-(t)} \Big((t+\beta)  \Phi_1^*(-t) + \mu_\eps N_{\alpha_1}^-(t)     \Big)e^{-t(T-x)}dt =\\
&  \frac 1 {2\pi i}\int_0^\infty e^{-tx} R_0(t) dt 
-\frac 1 {2\pi i}\int_{0}^{\infty}  e^{-t(T-x)} R_1(t) dt,
\end{aligned}
\end{equation}
where we defined 
\begin{align*}
R_0(t) = &
(\beta-t) \Phi_0^*(-t) \Big(\frac 1 {\Lambda^+(t)}-\frac 1 {\Lambda^-(t)}\Big) +\mu_\eps^2 
\Big(\frac {N_{\alpha_1}^+(t)} {\Lambda^+(t)} -    \frac {N_{\alpha_1}^-(t)} {\Lambda^-(t)} \Big)c^*,
\\
R_1(t)= & (\beta+t)  \Phi_1^*(-t) \Big( \frac 1 {\Lambda^+(t)}- \frac 1 {\Lambda^-(t)}\Big)+ \mu_\eps \Big( \frac {N_{\alpha_1}^+(t) } {\Lambda^+(t)}  -     \frac {N_{\alpha_1}^-(t) } {\Lambda^-(t)} \Big).
\end{align*}
In view of the above estimates these functions are bounded and decay to zero as a power function as $t\to\infty$.
This implies the desired claim: 
\begin{align*}
\|g^*\|_1 \le & \int_0^\infty |R_0(t)|\int_0^T e^{-tx}dx  dt + \int_0^\infty  |R_1(t)| \int_0^T e^{-t(T-x)}dx dt \\
&
\int_0^\infty |R_0(t)| \frac{1-e^{-Tt}}{t}  dt + \int_0^\infty  |R_1(t)| \frac {1-e^{-Tt}}{t} dt  <\infty.
\end{align*}

\subsubsection{Proof of \eqref{iig}}
The expression \eqref{gTfla} reveals that $g^*$ is continuous on $(0,T)$, 
possibly with integrable singularities at the endpoints. Thus $g^*$ is in the domain of both operators in equation \eqref{fWH}
and can be substituted into its left hand side. We want to show that  
\begin{equation}\label{eqstar}
\frac \partial {\partial s} \int_0^T g^*(r) \frac {\partial} {\partial r} K(r,s)dr + \int_0^T K_X(r,s)g^*(r)dr   = K_X(s,T), \quad \forall s\in (0,T).
\end{equation}
To do so we can apply the Laplace transform to both sides, extending them by zero outside the interval $(0,T)$, and check that 
equality is obtained on the whole plane. 
This amounts to repeating the calculations in the proof of Lemma \ref{lem:4.1}. 
Thus proving \eqref{eqstar} is equivalent to showing that, cf. \eqref{gTz_expr},
\begin{multline}\label{eqzz}
\int_0^T g^*(x)e^{-zx}dx  +\frac 1 {\Lambda(z)}
\bigg(
 (z+\beta) \Big(\Xi_0(z)+  e^{-zT}\Xi_1(-z)\Big) + \\ \mu_\eps^2  N_{\alpha_1}(z) 
\Big(\psi^*(0)+\frac {1}{\mu_\eps}e^{-zT}\Big)
\bigg)=0,
\end{multline}
where $\Xi_0(z)$, $\Xi_1(z)$ are the functions computed by plugging $g^*(x)$ and the corresponding function $\psi^*(x)$, 
defined by \eqref{psidef}, into \eqref{Phi0Phi1}.

As explained in the proof of \eqref{i} (see the paragraph following \eqref{gTx}), the function $g^*(x)$ vanishes outside $[0,T]$.
Thus the first term in \eqref{eqzz} coincides with the Laplace transform $\widehat g_*(z)$ and hence \eqref{eqzz} is equivalent to 
$$
\widehat g^*(z)  +\frac 1 {\Lambda(z)}
\bigg(
 (z+\beta) \Big(\Xi_0(z)+e^{-zT}\Xi_1(-z)\Big) +  \mu_\eps^2  N_{\alpha_1}(z) 
\Big(\psi^*(0)+\frac {1}{\mu_\eps}e^{-zT}\Big)
\bigg)=0.
$$
After plugging the expression for $\widehat g^*(z)$ constructed in the steps \eqref{a}-\eqref{c}, we see that this equality holds if  
$\Xi_0(z)$ and $\Xi_1(z)$ coincide with $\Phi_{0,*}(z)$ and $\Phi_{1,*}(z)$, and $\psi^*(0)$ coincides with the constant $c^*$ found in 
step \eqref{c}.

Let us first argue that $\psi^*(0)=c^*$. Since $\widehat g^*(z)$ is an entire function, so is $\widehat \psi^*(z)$. 
Letting $z:=-\beta$ in \eqref{psihatg},
shows that $\psi^*(0)= \widehat g^*(-\beta)-\frac 1 {\mu_\eps} e^{\beta T}$. On the other hand, by taking  $z\to \beta$ in the upper 
half plane in \eqref{gTzstar} implies $c^* = \widehat g^*(-\beta)-\frac 1 {\mu_\eps} e^{\beta T}$. Thus indeed $\psi^*(0)=c^*$.
Now we can show that $\Xi_0(z) = \Phi_{0,*}(z)$. 
The functions in the right hand side of \eqref{Phi0Phi1} are sectionally holomorphic on $\mathbb C\setminus \Real_+$ and
the first equation implies that  for $t>0$, 
\begin{equation}\label{PhiPhiA}
\Xi_0^+(t) -\Xi_0^-(t) = (t-\beta)  \big(\Psi_{g^*,0}^+(t)-\Psi_{g^*,0}^-(t)\big)
+ \mu^2_\eps \big(\Psi_{\psi_{T,*},0}^+(t)-\Psi_{\psi_{T,*},0}^-(t)\big).
\end{equation}
A direct calculation shows that the operator defined in \eqref{Psidefn} takes the form
$$
\Psi_{f,0}(z) = -\int_0^\infty \frac{N_\alpha^+(t)-N_\alpha^-(t)}{t-z}\widehat f(t)dt,
$$
when either $\alpha\in (0,1)$ and $\widehat f$ is locally bounded or $\alpha\in (1,2)$ and $\widehat f$ vanishes 
at the origin at a suitable rate. Hence this formula is valid for both $f:=g$ and $f:=\psi$.  

Since $\Phi_{0,*}(z)$ and $\Phi_{1,*}(z)$ satisfy, by construction, the boundary conditions  \eqref{eq:6.11},
the function  $\widehat g_*(z)$ is holomorphic and hence \eqref{PhiPhiA} becomes
\begin{align*}
& 
\Xi_0^+(t) -\Xi_0^-(t) = \\
&
- (t-\beta)   \big(N_{\alpha_2}^+(t)-N_{\alpha_2}^-(t)\big)\widehat g^*(t) 
- \mu_\eps^2 \big( N_{\alpha_1}^+(t)-N_{\alpha_1}^-(t)\big)\widehat \psi^*(t) \stackrel\dagger=\\
&
- (t-\beta)   \big(N_{\alpha_2}^+(t)-N_{\alpha_2}^-(t)\big)\widehat g^*(t) -
 \mu_\eps^2   \frac {N_{\alpha_1}^+(t)-N_{\alpha_1}^-(t)}{t+\beta}\Big(\psi^*(0)+\frac 1 {\mu_\eps} e^{-zT}-\widehat g^*(z)\Big)  =\\
&
-\frac 1 {t+\beta} \bigg( \big(\Lambda^+(t)-\Lambda^-(t)\big)\widehat g^*(t)  +
 \mu_\eps^2   \big(N_{\alpha_1}^+(t)-N_{\alpha_1}^-(t)\big)\Big(\psi^*(0)+\frac 1 {\mu_\eps} e^{-zT} \Big) 
 \bigg) \stackrel\ddagger =\\
&
\Phi_{0,*}^+(t) -\Phi_{0,*}^-(t),\quad t>0, 
\end{align*}
where in $\dagger$ we used the expression for $\widehat \psi^*(z)$ in \eqref{psihatg}
and in $\ddagger$ the definition of $\widehat g^*(z)$ through \eqref{gTzstar} (where, as we already showed, $c^*=\psi^*(0)$).

By construction, both
$\Xi_0(z)$ and $\Phi_{0,*}(z)$ are sectionally holomorphic on $\mathbb C\setminus \Real_+$,  
share the same asymptotic as $z\to\infty$ and, by the above calculations, have the same jump on the boundary. 
Hence they coincide, which is what we wanted to show. The same arguments apply to $\Xi_1(z)$ and $\Phi_{1,*}(z)$.


\def\cprime{$'$} \def\cprime{$'$} \def\cydot{\leavevmode\raise.4ex\hbox{.}}
  \def\cprime{$'$} \def\cprime{$'$} \def\cprime{$'$}

\end{document}